\documentclass[psamsfonts,12pt]{amsart}
\theoremstyle{plain}
\newtheorem{thm}{Theorem}[section]
\newtheorem{prop}[thm]{Proposition}
\newtheorem{lemma}[thm]{Lemma}
\newtheorem{cor}[thm]{Corollary}
\newtheorem{question}[thm]{Question}

\newtheorem{hyp}[thm]{Hypothesis}

\theoremstyle{definition}
\newtheorem{dfn}[thm]{Definition}
\newtheorem{example}[thm]{Example}
\newtheorem{notn}[thm]{Notation}

\theoremstyle{remark}
\newtheorem{rem}[thm]{Remark}

\newcommand{\beql}[1]{\begin{equation}\label{#1}}
\newcommand{\eeq} {\end{equation}}

\def\C{{\rm{C}}}

\def\N{{\rm{N}}}

\def\aut{\mathrm{Aut}}

\begin{document}

\title{The local lifting problem for actions of finite groups on curves}
\author{Ted Chinburg}
\address{T.C.: Department of Mathematics\\University of
Pennsylvania\\Philadelphia, PA
19104-6395}
\email{ted@math.upenn.edu}
\author{Robert Guralnick}
\address{R.G.: Department of Mathematics\\University of
Southern California\\3620 South Vermont Ave., KAP 108\\
Los Angeles, California 90089-2532
}
\email{guralnic@usc.edu}
\author{David Harbater}
\address{D.H..: Department of Mathematics\\University of
Pennsylvania\\Philadelphia, PA
19104-6395}
\email{harbater@math.upenn.edu}



\thanks{{\it French title:} Probl\`eme local de rel\`evement de l'action d'un groupe fini sur une courbe. }
\thanks{ The authors  were supported in part by NSF Grants DMS-0801030, DMS-0653873, DMS-0901164}

\begin{abstract} 

Let $k$ be an 
algebraically closed field of characteristic $p > 0$. We study obstructions to lifting to characteristic $0$ the faithful 
continuous action $\phi$ of a finite group $G$ on $k[[t]]$.  To each such $\phi$ a theorem
of Katz and Gabber associates an action of $G$ on a smooth projective curve $Y$ over $k$.  We say that the KGB obstruction
of $\phi$ vanishes if $G$ acts on a smooth projective curve $X$ in characteristic $0$ in such a way
that $X/H$ and $Y/H$ have the same genus for all subgroups $H \subset G$.  We determine for
which $G$ the KGB obstruction of every $\phi$ vanishes. 
We also consider
analogous problems in which one requires only that an obstruction to lifting $\phi$ due to Bertin vanishes for some $\phi$, or 
for  all sufficiently ramified $\phi$.   These results provide evidence for the strengthening of Oort's lifting conjecture
which is discussed in \cite[Conj.~1.2]{CGH}.
\vskip 0.2in{}
\noindent{\bf R\'esum\'e.} 
Soit $k$ un corps alg\'ebriquement clos de caract\'eristique $p>0$. Nous \'etudions les obstructions au rel\`evement en caract\'eristique 0 d'une action fid\`ele et continue $\phi$ d'un groupe fini $G$ sur $k[[t]]$. Le th\'eor\`eme de Katz-Gabber associe \`a $\phi$, une action du groupe  $G$ sur une courbe projective $Y$ lisse sur $k$. La KGB-obstruction de $\phi$ est dite nulle si $G$ agit sur une courbe projective lisse $X$ de caract\'eristique 0 avec \'egalit\'e des genres de $X/H$ et $Y/H$ pour tout sous-groupe $H\subset G$. Nous d\'eterminons les groupes $G$ pour lesquels la KGB-obstruction s'annule pour toute action  $\phi$. Nous consid\'erons \'egalement des situations analogues pour lesquelles il suffit d'annuler l'obstruction de Bertin \`a relever une action $\phi$ ou toutes actions $\phi$ suffisamment ramifi\'ees. Ces r\'esultats renforcent les convictions en faveur de la conjecture de Oort g\'en\'eralis\'ee aux rel\`evements
d'une action fid\`ele sur une courbe projective lisse ([8], Conj. 1.2).

\end{abstract}

\date{Oct. 4, 2009}
\maketitle

\tableofcontents

\section{Introduction.}
\label{s:intro}
\setcounter{equation}{0}

This paper concerns the problem of
lifting actions of finite groups on curves from positive characteristic to characteristic $0$.
Let $k$ be an algebraically closed field of characteristic $p > 0$, and let $\Gamma$ be a
finite group acting faithfully on a smooth projective curve $Y$ over $k$.  We will say this
action {\it lifts to characteristic} $0$ if there is a complete discrete valuation ring $R$ having characteristic 
$0$ and residue field $k$ for which the following is true.  There is an action of $\Gamma$ a smooth projective curve $\tilde Y$
over  $R$ for which there is a $\Gamma$-equivariant isomorphism between $k \otimes_R \tilde Y$ and $Y$.

We focus in this paper on the following local version of this problem.
Let $\phi:G \to \mathrm{Aut}_k(k[[t]])$ be an injective homomorphism 
from a finite group $G$ into the group of continuous automorphisms of the power series ring $k[[t]]$ over $k$.  The existence such a $\phi$ implies  $G$ is the semi-direct product of a cyclic group of order prime to $p$ (the maximal tamely ramified quotient) by a 
normal $p$-subgroup (the wild inertia group).  One says $\phi$ {\it lifts to characteristic} $0$ if there an $R$ as above such that $\phi$ can be lifted to an embedding $\Phi:G \to \mathrm{Aut}_R(R[[t]])$ in the sense that $k \otimes_R \Phi = \phi$.  

The local and global lifting problems are connected in the following way by a result of Bertin and Mezard
\cite{BM}.  For each wildly ramified closed point $y$ of $Y$,  fix an identification of the completion of the local ring of $Y$
at $y$ with $k[[t]]$, and let $\phi_y:\Gamma(y) \to \mathrm{Aut}_k(k[[t]])$ be the resulting action of the inertia group $\Gamma(y)$
of $y$ on this completion. Then $\phi$ lifts to characteristic $0$ if and only if each $\phi_y$ does.  

In \cite{CGH} we studied the global lifting problem.  We defined $\Gamma$ to be a (global) Oort group for $k$
if  every action of $\Gamma$ on a smooth projective curve over $k$ lifts to characteristic $0$. This terminology arises
from Oort's conjecture in \cite[\S I.7]{OO} that all cyclic groups $\Gamma$ have this property, or equivalently that every connected cyclic cover lifts to characteristic $0$.   We
showed in \cite{CGH} that all groups $\Gamma$ which are Oort groups for $k$ must be on a certain list of finite groups that is
recalled in Remark \ref{rem:globalOort} below, and we conjectured this list was complete.  Some results
by various authors concerning Oort's conjecture and the generalization proposed in \cite{CGH} are discussed after Remark \ref{rem:globalOort} below.

In this paper we will focus on three local versions of the results in \cite{CGH}.  We will consider which finite groups $G$ that are semi-direct products
of a cyclic prime to $p$-groups with a normal  $p$-subgroup  have
the following properties for the field $k$.  
\begin{enumerate}
\item[1.]  If every local action $\phi:G \to \mathrm{Aut}_k(k[[t]])$ lift to characteristic $0$ we 
call $G$ a {\it local Oort group for} $k$ (as in \cite{CGH}).
\smallskip
\item[2.]  If every local action $\phi:G \to \mathrm{Aut}_k(k[[t]])$ that is sufficiently ramified
lifts to characteristic $0$, we will call $G$ an {\it almost local Oort group for} $k$.   More precisely,
$G$ is an almost local Oort group if there an integer $N(G,k)\ge 0$ such that a local action $\phi$ lifts to characteristic $0$ provided
$t^{N(G,k)}$ divides $\phi(\sigma)(t) - t$
in $k[[t]]$ for all elements $\sigma \in G$ of $p$-power order. 
\smallskip
\item[3.]  If there is at least one local action $\phi:G \to \mathrm{Aut}_k(k[[t]])$ which
lifts to characteristic $0$ we will call $G$ a {\it weak local Oort group for} $k$.
\end{enumerate}

Our goal is to show that any $G$ which has
one of the three above properties must be on  a certain list of groups associated to 
that  property.  In view of Oort's conjecture concerning cyclic groups we will ask  to what extent these lists are complete.

The lists that we obtain will result from studying an obstruction to lifting $\phi$ that is due to Bertin \cite{Bertin}, 
as well as from a refinement of this obstruction that we will call the KGB obstruction.

The {\it Bertin obstruction of $\phi$ vanishes} if there is a finite $G$-set $S$ with non-trivial cyclic stabilizers
such that the character $\chi_S$ of the action of $G$ on $S$ equals $-a_\phi$ on the non-trivial
elements of $G$, where $a_\phi$ is the Artin character associated to $\phi$.  (For the definition
of $a_\phi$ see \cite[Chap. VI]{corps}.) The condition on $\chi_S$ is thus that 
\begin{equation}
\label{eq:chiS}
\chi_S = m\cdot  {\rm reg}_G - a_\phi
\end{equation}
for some integer $m$, where ${\rm reg}_G$ is the character of the regular
representation of $G$.

We will say that {\it Katz-Gabber-Bertin obstruction of $\phi$ vanishes}, or simply that the {\it KGB obstruction of $\phi$ vanishes},
if the following is true.   There is a field $K$ 
of characteristic $0$ and a $G$ cover $X \to X/G = \mathbb{P}^1_K$ of smooth geometrically irreducible projective curves
over $K$ such that $$\mathrm{genus}(X/H) = \mathrm{genus}(Y/H)$$ for all subgroups $H$ of $G$,
where $Y \to Y/G = \mathbb{P}^1_k$ is the $G$-cover of smooth projective curves associated
to $\phi$ by a theorem of Katz and Gabber (see \cite{KG}).   Up to isomorphism, the Katz-Gabber cover 
$Y \to Y/G = \mathbb{P}^1_k$ is characterized by the fact that this $G$-cover is 
totally ramified over one point $\infty \in \mathbb{P}^1_k$, at most tamely ramified over
another point $0 \in \mathbb{P}^1_k$,  unramified off of $\{\infty,0\}$, and the action of $G$ on the completion
$\hat {\mathcal O}_{Y,x}$ of the local ring of $Y$ at the unique point $x$ over $\infty$ corresponds
to $\phi$ via a continuous $k$-algebra isomorphism between $\hat {\mathcal O}_{Y,x}$ and $k[[t]]$. 

We prove in Theorem \ref{thm:KGBtwo} that the Bertin obstruction vanishes if the KGB obstruction vanishes.
In Appendix 2 we show that the KGB obstruction for $\phi$ need not vanish
when the Bertin obstruction of $\phi$ does.

\begin{dfn}
\label{dfn:termdef} Let $G$ be a finite group which is the semi-direct product of a cyclic prime to $p$ group by a normal $p$-subgroup.
If  the 
Bertin obstruction (resp.\ the KGB obstruction) vanishes for all $\phi$ then $G$ will be called a {\it  Bertin group} for $k$ (resp.\ 
a {\it  KGB group} for $k$).   If this is true for all sufficiently ramified $\phi$
we call $G$ an {\it almost Bertin group} for $k$ (resp.\  an {\it almost 
KGB group} for $k$). Finally, if there is at least one $\phi$ for which the Bertin obstruction (resp.\ the KGB obstruction)
vanishes, we will call $G$ a {\it weak Bertin group} for $k$ (resp.\ a {\it weak KGB group} for $k$).
\end{dfn}

Thus a local Oort group for $k$ must be a KGB group for $k$, which must in turn be a Bertin group for $k$.  One has a similar
statement concerning almost local Oort groups and weak local Oort groups for $k$.

We can now state our main result concerning Bertin and KGB groups for $k$. 

\begin{thm}
\label{thm:nonexs}Suppose $G$ is a finite group which  is a semi-direct product of a normal $p$-subgroup
with a cyclic group  of order prime to $p$.  Let $k$ be an algebraically closed field of characteristic $p$. Then $G$ is a KGB group for $k$ 
 if and only if it is a Bertin group for $k$, and this is true exactly when $G$ is isomorphic to a group of one of the following kinds:
\begin{enumerate}
\item[1.] A cyclic group.
\item[2.] The dihedral group $D_{2p^n}$ of order $2p^n$ for some $n \ge 1$.
\item[3.] $A_4$ when $p = 2$.
\item[4.] A generalized quaternion group $Q_{2^m}$ of order $2^m$
for some $m \ge 4$ when $p = 2$. 
\end{enumerate}
\end{thm}

Note that if $p = 2$ and $n = 1$ in item (2), $D_{4}$ is simply
$\mathbb{Z}/2 \times \mathbb{Z}/2$. 

By considering particular covers we showed in \cite[Thm.~3.3, 4.4]{CGH} that if $G$ is a local
Oort group for $k$ then it must either be on the list given in Theorem \ref{thm:nonexs} or else $p = 2$ and $G$ is a semi-dihedral group of order at least $16$.   Theorem
\ref{thm:nonexs} shows that in fact, semi-dihedral groups are not local Oort groups in
characteristic $2$.  Note also that Theorem \ref{thm:nonexs} is an if and only if statement concerning
KGB and Bertin groups. 
Theorem 3.3 of \cite{CGH} concerns only necessary conditions which must be satisfied by local Oort groups.

There are examples of particular actions $\phi$ for which the Bertin obstruction
to lifting vanishes but the KGB obstruction does not (see Example \ref{ex:explicit}).  Thus 
the fact that the
Bertin and KGB groups turn out to be the same has to do with the requirement
that the associated obstructions vanish for {\it all} such $\phi$.  

Pagot
 has shown in \cite[Thm. 3]{P2} (see also \cite{Mat}) that  there are $\phi$ which have vanishing
Bertin and KGB obstructions but which nonetheless do not lift to characteristic
$0$.  Thus the latter obstructions are not sufficient
to determine whether $\phi$ has a lift.  

In view of Theorem \ref{thm:nonexs}, we asked the following question:

\begin{question}
\label{eq:quest} Is the set of groups listed in items [1] - [4] of  Theorem \ref{thm:nonexs}
the set of all local Oort groups for algebraically closed fields $k$ of characteristic $p$?
\end{question}

\noindent Brewis and Wewers \cite{BWe} have announced a proof that the answer to this question
is negative because the generalized quaternion group of order $16$ is not a local
Oort group in characteristic $2$.  

\begin{rem}
\label{rem:globalOort}
Suppose the groups of type (1), (2)
and (3) in  Theorem \ref{thm:nonexs} are all local Oort groups.  It would then follow from  \cite[Thm.~2.4, Cor.~3.4, Thm.~4.5]{CGH} that 
a cyclic by $p$ group $\Gamma$ is a global Oort group for $k$ if and only if $\Gamma$ is
either cyclic, dihedral of order $2p^n$ for some $n$ or (if $p = 2$) the alternating
group $A_4$.  This implication is not dependent on determining which generalized quaternion
groups are local Oort groups in characteristic $2$.
\end{rem}

Oort's conjecture
in \cite{OSS} that cyclic groups are local and global Oort groups  was shown for cyclic groups having a $p$-Sylow subgroup of order $p$ (resp.\ $p^2$)  by Oort, Sekiguchi and Suwa \cite{OSS} (resp.\
by Green and Matignon \cite{GM}).  Pagot showed (see \cite{P2} and \cite{Mat})
that when $p = 2$, the Klein four group $D_{4}$ is a local and global
Oort group.  Bouw and Wewers have shown \cite{BW} that for all
odd $p$, the dihedral group $D_{2p}$ is a local and global Oort group,
and they have announced a proof that when $p = 2$, $A_4$ is
a local and global Oort group.  In \cite{BWZ}, Bouw, Wewers and Zapponi establish necessary
and sufficient conditions for a given $\phi$ to lift to characteristic $0$ whenever the $p$-Sylow
subgroup of $G$ has order $p$, regardless of whether $G$ is a local Oort group.  

The following is our main result concerning almost KGB groups and almost Bertin groups for $k$.

\begin{thm}
\label{thm:nonexsalmost}Suppose $G$ is a finite group which  is the semi-direct product of a cyclic group  of order prime to $p$ by a normal $p$-subgroup.  Then $G$ is an almost KGB group for $k$  if and only if it is an almost Bertin group for $k$.  The list of these groups
consists of those appearing in Theorem \ref{thm:nonexs} together with the groups
$\mathrm{SL}_2(\mathbb{Z}/3)$ and $Q_8$ when $p = 2$. 
\end{thm}

In a similar vein to Question \ref{eq:quest}, we  ask:

\begin{question}
\label{eq:quest2} Is the set of  groups described in Theorem \ref{thm:nonexsalmost}
the set of  almost local Oort groups for $k$?
\end{question}

We now consider $G$ which
are weak Bertin groups, i.e.\  for which there is at least one injection $\phi:G \to \mathrm{Aut}_k(k[[t]])$
that has vanishing Bertin obstruction. We will give a purely group theoretic characterization of such $G$ which requires no quantification over  embeddings of $G$ 
into $\mathrm{Aut}_k(k[[t]])$.

\begin{dfn}
\label{def:groovy} Let $G$ be the semi-direct product of a normal $p$-group $P$
by cyclic subgroup $Y$ of order prime to $p$.  Let $B$ be the maximal subgroup
of $Y$ of order dividing $p-1$.  We will call $G$ a {\it Green-Matignon group for} $k$, 
or more briefly a {\it GM group for}
$k$, if
there is a faithful character $\Theta: B \to \mathbb{Z}_p^*$ for which the following
is true:
\begin{enumerate}
\item[a.]  If $1 \ne c \in Y$, then $\C_P(c)=\C_P(C)$ and this group is cyclic.
\item[b.] Suppose $T$ is a cyclic subgroup of $P$ and that $\C_{C}(T)$ is trivial.  Then 
$$xyx^{-1} = y^{\Theta(x)}$$
for $y \in T$ and $x \in \N_B(T)$.
\end{enumerate}
\end{dfn}

Note that  if $|B| \le 2$, $\Theta$ is unique and so condition (b) is
vacuous.
Clearly, cyclic groups and $p$-groups are GM-groups.

\begin{thm}
\label{thm:antiB} Let $G$ be the semi-direct product of a normal $p$-group $G$
by cyclic subgroup $C$ of order prime to $p$. 
There is an injection $\phi:G \to \mathrm{Aut}_k(k[[t]])$
which has vanishing Bertin obstruction if and only if $G$ is
a  GM group for $k$.  Thus $G$ is a weak Bertin group for $k$ if and only if it is 
a  GM group for $k$.  
\end{thm}

This result generalizes a result of Green and Matignon in \cite{GM} which states that no $\phi$ can 
lift to characteristic $0$ if $G$ contains an abelian subgroup that is neither cyclic nor a $p$-group.
  In \S \ref{s:exGM} we give some
further examples and characterizations of GM groups.  In particular, in 
Theorem \ref{thm:GMexamples}(c)(ii-iv)  we describe some groups 
which are not GM groups   even though all of their abelian subgroups are either cyclic
or $p$-groups.

Following Question \ref{eq:quest}, we ask:

\begin{question}
\label{eq:quest3} Is the set of groups described in Theorem \ref{thm:antiB}
the set of groups $G$ for which some injection $\phi:G \to \mathrm{Aut}_k(k[[t]])$
lifts to characteristic $0$, i.e.\ the set of weak local Oort groups for $k$?
\end{question}
If the answer to this question is affirmative, then every $p$-group would be
a weak local Oort group for $k$.  Matignon has shown in \cite{Ma1}
that every elementary abelian $p$-group is a weak local Oort group for $k$.
Brewis has shown in \cite{Bre} that when $p = 2$ the dihedral group
of order $8$ is a weak local Oort group for $k$.  As one final instance
of Question \ref{eq:quest}, it follows from  work of Oort, Sekiuchi and Suwa and of Bouw, Wewers
and Zapponi that the answer
is affirmative if $\# G$ is exactly divisible by $p$;  see Example \ref{ex:cycpsub}.

We now discuss the organization of the paper.

In Proposition 
\ref{prop:special} of \S \ref{s:imp} we give a numerical reformulation of
the Bertin obstruction of $\phi$.  We show this obstruction vanishes if and only
if a constant $b_T(\phi) \in \mathbb{Q}$ associated to each non-trivial cyclic subgroup $T$
of $G$ is non-negative and integral.  
The basis
for this reformulation is Artin's Theorem that every character of $G$
with rational values is a unique rational linear combination of
the characters of $G$-sets of the form $G/T$ as $T$ ranges
over a set of representatives for the conjugacy classes of cyclic subgroups of $G$. 
In \S \ref{s:numresults} we compute the constants $b_T(\phi)$ when $T$ contains
a non-trivial element of order prime to $p$. 

In  \S \ref{s:KGBsect} we give an alternate characterization of the vanishing
of the KGB obstruction which shows that if it vanishes, then the Bertin obstruction
vanishes.

In  \S \ref{s:functor} we consider the functorial properties
of the Bertin and KGB obstructions on passage to subgroups
and quotient groups.  We show that the vanishing of the
Bertin (resp.\ KGB) obstruction for $\phi:G \to \mathrm{Aut}_k(k[[t]])$
implies that the corresponding obstruction vanishes for the
injection $\phi^\Gamma:\Gamma \to \mathrm{Aut}_k(k[[t]]^N)$
associated to the quotient $\Gamma$ of $G$ by a normal subgroup
$N$.  The vanishing of the Bertin obstruction of $\phi$ implies
that the Bertin obstruction of the restriction $\phi|_H$ of $\phi$ to any
subgroup $H$ of $G$ also vanishes.  

One consequence of \S \ref{s:functor} is that if the Bertin
obstruction of $\phi$ vanishes, then that of the restriction $\phi|_P$
vanishes when $P$ is the (normal) $p$-Sylow subgroup of $G$.  
In \S \ref{s:reduce} we sharpen this statement 
by showing that the Bertin obstruction of $\phi$ vanishes
if and only that of $\phi|_P$ vanishes and $G$ and $a_\phi$
satisfy some further conditions (see Theorem \ref{thm:reducetop}).
The extra conditions are purely group theoretic except for
one (condition c(ii) of Theorem \ref{thm:reducetop}) on the numerical size of the
wild ramification associated to $\phi$.  This reduction to $p$-groups
is central to the proof of Theorem \ref{thm:antiB}.  The proof 
of Theorem \ref{thm:reducetop} is carried out in \S \ref{s:pgroups}, using results from  \S \ref{s:numresults},
\S \ref{s:nonpgroups}
and \S \ref{s:pgroups}.  
We prove Theorem 
\ref{thm:antiB}  in 
\S \ref{s:therealend}.  In \S \ref{s:exGM} we give some examples and alternate group theoretic characterizations of 
GM groups.  

To prove Theorems \ref{thm:nonexs}
and  \ref{thm:nonexsalmost} we must introduce some 
further ideas.   Our strategy  is to exploit
the fact that any quotient of a Bertin group must be a Bertin group
(and similarly for almost Bertin groups).  
One can thus eliminate $G$ from the list of Bertin groups by  showing
it has a quotient that is not Bertin.  
In \S \ref{s:reduction} we recall from \cite{CGH} some purely group theoretic
results which show that if $G$ is not on a small list of groups
then it must have a quotient which is on a second list of groups.
We then work to show that every element of the second list
is not a Bertin group, while every element on the first list is
a KGB group (and thus also a Bertin group).  

The above strategy for proving Theorems \ref{thm:nonexs} and 
\ref{thm:nonexsalmost} is carried out
in the following way.  In 
\S \ref{s:notalmostBertin} various groups are shown not to
be almost Bertin.   To use local class field
theory we show in \S \ref{s:quasired} that we
can allow $k$ to be quasi-finite in the sense of \cite[\S XIII.2]{corps}
rather than algebraically closed.  
 In  \S \ref{s:exts2}, \S \ref{s:cft} and \S \ref{s:sl23} 
we analyze the case
of dihedral groups for all $p$, quaternionic and semi-dihedral groups when $p = 2$ and the
group $\mathrm{SL}_2(3)$ when $p = 2$.   These results provide a new proof
in Corollary \ref{cor:Serrecor} of a result of J P. Serre \cite[\S 5]{serre2}  and J. M. Fontaine \cite{JMF} concerning  local Artin
representations associated to generalized quaternion groups which are not realizable over $\mathbb{Q}$.  
The proofs of Theorems \ref{thm:nonexs} and 
\ref{thm:nonexsalmost} are completed in \S \ref{s:yipes} and 
\S \ref{s:nonexall} using results from Appendix 1. 
In Appendix 1 we prove a technical result which constructs solutions to certain embedding problems with $p$-group kernels
such that the Artin character of the solution has large values on
non-trivial elements of the kernel as well as further congruence properties.    To keep the details of the construction from obscuring the
arguments in the main theorems we put them in Appendix 1.

This is the second in a series of three papers concerning lifting problems.  
In the third paper of  this series we will study the implications of Theorem \ref{thm:nonexs}
to the structure of the global Oort groups considered in \cite{CGH}.

\medbreak
\noindent {\bf Acknowledgements:}  We would like to thank J. Bertin,  I. Bouw, O. Gabber, M. Matignon, A. Mezard, F. Oort, F. Pop and S. Wewers for useful conversations.

\section{The Bertin obstruction.}
\label{s:imp}
\setcounter{equation}{0}

Let $k$ be an algebraically closed field of characteristic $p$.  Suppose  $G$ is
a finite group, and let $\phi:G \to \aut_k(k[[t]])$ be an embedding.   
Let $\mathcal{C}$ be a set of representatives for the conjugacy classes of cyclic subgroups of $G$.  For each subgroup $H$ of $G$, let $1_H$ be the one-dimensional trivial character of $H$, and let $1_H^G = \mathrm{Ind}_H^G 1_H$ be the induction of $1_H$ from $H$ to $G$.

\begin{prop}
\label{prop:special}  Let $a_\phi$ be the Artin character of $\phi$. 
\begin{enumerate}
\item[i.] There are unique rational numbers $b_T = b_T(\phi)$ for $T \in \mathcal{C}$ such that  
\begin{equation}
\label{eq:aphiequal}
-a_\phi = \sum_{T \in \mathcal{C}} b_T \ 1_T^G.
\end{equation}
\item[ii.]
The following conditions are equivalent:
\begin{enumerate}  
\item[a.]  The Bertin obstruction of $\phi$ vanishes.
\item[b.] One has
$0 \le b_T \in \mathbb{Z}$ for all $T\in \mathcal{C}$ such that $T \ne \{e\}$.    
\end{enumerate} 
\item[iii] Suppose the conditions in part (ii) hold.  Let $S$ be the $G$-set whose character appears 
in the definition of the vanishing of the Bertin obstruction in (\ref{eq:chiS}).   Then $m = -b_{\{e\}} \ge 0$ in (\ref{eq:chiS}), and there is an isomorphism of $G$-sets
\begin{equation}
\label{eq:easystuff}
S \cong \coprod_{\{e\} \ne T \in \mathcal{C}} \ \ \coprod_{i = 1}^{b_T} \ (G/T).
\end{equation}
\end{enumerate}
\end{prop}

\begin{proof}  We first prove (i). By Artin's
Theorem \cite[Thm. 13.30, Cor. 13.1]{SerreRep}, every character of $G$ with rational
values is a $\mathbb{Q}$-linear combination of the characters $\{1_T^G: T \in \mathcal{C}\}$,
and the dimension over $\mathbb{Q}$ of the space of all $\mathbb{Q}$-valued characters is $\# \mathcal{C}$.  
Since the character $-a_\phi$  takes rational values, we conclude that (\ref{eq:aphiequal}) holds 
for a unique function $T \mapsto b_T$ from $ \mathcal{C}$ to $\mathbb{Q}$. 

We now prove (ii).  Suppose statement (a) in part (ii) holds.    By considering the $G$-orbits of elements of $S$,
we see that  there is a $G$-set isomorphism 
 \begin{equation}
 \label{eq:Gsets}
 S \cong \coprod_{\{e\} \ne T \in \mathcal{C}} \ \ \coprod_{i = 1}^{n_T} \ G/T
 \end{equation}
 for some integers $n_T \ge 0$.  By (\ref{eq:chiS}) one has 
 \begin{equation}
 \label{eq:Schar}
m\ \mathrm{reg}_G - a_\phi = \chi_S =  \sum_{\{e\} \ne T \in \mathcal{C}} n_T \ 1_T^G.
 \end{equation}
Hence
\begin{equation}
\label{eq:aequation}
-a_\phi = -m \ \mathrm{reg}_G + \sum_{\{e\} \ne T \in \mathcal{C}} n_T\   1_T^G
\end{equation}
where $\mathrm{reg}_G =  1_{\{e\}}^G$.  So by the uniqueness of the rational numbers $b_T$ in (\ref{eq:aphiequal}), we conclude
that $m = - b_{\{e\}}$ and $n_T = b_T$ for $\{e\} \ne T \in \mathcal{C}$.  Thus statement (b) in
part (ii) holds since $0 \le n_T \in \mathbb{Z}$
if $T \ne \{e\}$. 

Suppose now that condition (b) of part (ii)  holds.  Define
\begin{equation}
\label{eq:easystuffed}
S = \coprod_{\{e\} \ne T \in \mathcal{C}} \ \ \coprod_{i = 1}^{b_T} \ (G/T)
\end{equation}
where $b_T \ge 0$ for the $T$ appearing in this coproduct.  The stabilizers of 
elements of $S$ are then conjugates of those $T$  for
which $b_T > 0$, so they are non-trivial cyclic subgroups.  By (\ref{eq:aphiequal}) we have
\begin{equation}
\chi_S = \sum_{\{e\} \ne T \in \mathcal{C}} b_T \  1_T^G = -b_{\{e\}} \mathrm{reg}_G - a_\phi.
\end{equation}
Therefore condition (a) of part (ii) holds.

It remains to show part (iii) of Proposition \ref{prop:special}, so we assume that the conditions in part (ii) hold.  The inner product $\langle a_\phi,\chi_0\rangle$ of $a_\phi$ with 
the one-dimensional trivial representation $\chi_0$ of $G$ is $0$ by \cite[\S VI.2]{corps}.  Hence 
(\ref{eq:chiS}) gives $m = \langle \chi_S,\chi_0\rangle / \langle {\rm reg}_G, \chi_0 \rangle \ge 0$. It  will now suffice to show that  (\ref{eq:easystuffed}) defines up to isomorphism  the unique
$G$-set $S$ with non-trivial cyclic stabilizers for which condition (a) of part (ii) holds.  Since condition (a) of part (ii) determines the character of $S$
up to an integral multiple of $1_{\{e\}}^G$, this is a consequence of the fact that the characters 
$\{1_{T}^G: T \in \mathcal{C}\}$
are linearly independent over $\mathbb{Q}$ by Artin's Theorem.
\end{proof}

We now develop some formulas for the constant $b_T$ appearing in (\ref{eq:aphiequal}).

\begin{notn}
\label{def:nontrivT}
Since $k$ is algebraically closed, $G = G_0$ is the inertia group of $G$ as an automorphism
group of $k[[t]]$.  Let $G_i$ be the $i^{th}$ ramification subgroup of $G$ in the lower numbering.
For all non-trivial subgroups $\Gamma$ of $G$,  let $\iota(\Gamma) = i+1$ when $i$ is the largest positive integer for which $\Gamma \subset G_i$.  Let $\mu$ be the Mobius function, and let $\N_G(T)$ be the normalizer of a subgroup 
$T$ in $G$.  Let $S(T) = S_G(T)$ be the set of all non-trivial cyclic subgroups $\Gamma$ of $G$ which contain $T$.
If $T'$ is also a subgroup of $G$, let $\delta(T,T') = 1$ if $T = T'$ and let $\delta(T,T') = 0$ if $T \ne T'$.
\end{notn}

\begin{thm}
\label{thm:nontrivcase}
For $T \in \mathcal{C}$ the constant $b_T$ 
appearing in (\ref{eq:aphiequal}) is given by
\begin{equation}
\label{eq:theform}
b_T = \frac{1}{[\N_G(T):T]} \left( -\delta(T,\{e\}) a_\phi(e) + \sum_{\Gamma \in S(T)} \mu([\Gamma:T]) \iota(\Gamma) \right ).
\end{equation}
\end{thm}

The proof is based on the following formula from  Proposition VI.2 of \cite{corps}.

\begin{prop}
\label{prop:serrethm} One has
\begin{equation}
\label{eq:formulaartin}
a_\phi = \sum_{i = 0}^\infty \frac{1}{[G_0:G_i]} \left ( 1_{\{e\}}^G - 
 1_{G_i}^G \right )
\end{equation}
\end{prop}

We also need two lemmas. The first is Proposition 9.27 of \cite{SerreRep} while 
the second follows easily from the Mobius inversion formula.

\begin{lemma}
\label{lem:serrelem}
For any cyclic group $A$ let $\theta_A$ be the character defined
for $\sigma \in A$ by $\theta_A(\sigma) = \# A$ if $\sigma$ is a generator of $A$ and $\theta_A(\sigma) = 0$ otherwise.  Let $J$ be an arbitrary  finite group.  Then 
\begin{equation}
\label{eq:dummy}
\# J \cdot 1_J = \sum_{A \subset J} \mathrm{Ind}_A^J(\theta_A)
\end{equation}
where the sum is over the cyclic subgroups $A$ of $J$.
\end{lemma}

\begin{lemma}
\label{lem:triv}
With the notations of Lemma \ref{lem:serrelem}, one has
\begin{equation}
\label{eq:thetaform}
\theta_A = \sum_{H \subset A} \#H \cdot \mu([A:H]) \cdot 1_H^A.
\end{equation}
where the sum is over all subgroups $H$ of the cyclic group $A$.
\end{lemma}

\medbreak
\noindent {\bf Proof of Theorem \ref{thm:nontrivcase}.}
\medbreak
Choose an integer $N$ large enough so that $G_i = \{e\}$ if $i \ge N$.  Then
(\ref{eq:formulaartin}) becomes
\begin{equation}
\label{eq:nicartin}
a_\phi = \left (\sum_{i = 0}^N \frac{1}{[G_0:G_i]}\right ) \cdot   1_{\{e\}}^G - 
\sum_{i = 0}^N \left (\frac{1}{[G_0:G_i]} \cdot 1_{G_i}^G\right )
\end{equation}
Apply Lemmas \ref{lem:serrelem} and \ref{lem:triv} to $J = G_i$. This gives
\begin{equation}
\label{eq:thebigone}
 1_{G_i}= \frac{1}{\#G_i} \cdot \sum_{\mathrm{cyclic} \ A \subset G_i} \mathrm{Ind}_A^{G_i} \left (
\sum_{H \subset A} \#H \cdot \mu([A:H]) \cdot 1_H^A\right ).\\
\end{equation}
 We now induce (\ref{eq:thebigone}) up from $G_i$ to $G$,
multiply by $\frac{1}{[G_0:G_i]}$ and sum over $i$.  This leads to 
\begin{equation}
\label{eq:gettingthere}
\sum_{i = 0}^N \left (\frac{1}{[G_0:G_i]} \cdot 1_{G_i}^G\right )
= \frac{1}{\# G_0} \sum_{i = 0}^N  \ \ \ \sum_{H\subset A  \subset G_i, \ A  \    \mathrm{cyclic}} \# H \cdot \mu([A:H]) \cdot 1_H^G.
\end{equation}
For this proof only, extend Notation \ref{def:nontrivT} by setting $\iota(\{e\}) = N+1$.  We can now rewrite (\ref{eq:gettingthere}) as
\begin{equation}
\label{eq:gettingbetter}
\sum_{i = 0}^N \left (\frac{1}{[G_0:G_i]} \cdot 1_{G_i}^G\right )
=  \  \sum_{H\subset A \subset G,\ A \ \mathrm{cyclic}} \frac{\mu([A:H])}{ [G_0:H]}  \cdot \iota(A) \cdot 1_H^G
\end{equation}

Group the terms on the right side of (\ref{eq:gettingbetter}) according to which $T \in \mathcal{C}$
is conjugate to $H$.   Since $G = G_0$, and $\iota(A)$ (resp. $1_H^G$) depends only on the conjugacy class
of $A$ (resp. $H$), this  leads to 
\begin{eqnarray}
\label{eq:gettingbest}
\sum_{i = 0}^N \left (\frac{1}{[G_0:G_i]} \cdot 1_{G_i}^G\right )
&=& 
   \sum_{ \Gamma \in  \{\{e\}\} \cup S(\{e\})} \frac{[G_0:\N_G(\{e\})] \cdot \mu([\Gamma:\{e\}])}{ [G_0:\{e\}]}  \cdot \iota(\Gamma) \cdot 1_{\{e\}}^G\nonumber\\
&+& \sum_{\{e\} \ne T \in \mathcal{C}} \  \sum_{ \Gamma \in  S(T)} \frac{[G_0:\N_G(T)]\cdot \mu([\Gamma:T]) }{ [G_0:T]} \cdot \iota(\Gamma) \cdot 1_T^G
\end{eqnarray}
where $S(T)$ is as in Notation \ref{def:nontrivT}.  
Substituting this back into (\ref{eq:nicartin}) results in 
\begin{equation}
\label{eq:yup}
-a_\phi = \sum_{T \in \mathcal{C}} b_T 1_T^G
\end{equation}
in which $b_T$ has the form in (\ref{eq:theform}) if $T \ne \{e\}$.  Suppose now that $T = \{e\}$.
We get 
\begin{equation}
\label{eq:btriv}
b_{\{e\}} = - \left (\sum_{i = 0}^N \frac{1}{[G_0:G_i]}\right ) + \frac{1}{\# G_0} (N+1) + 
 \sum_{ \Gamma \in   S(\{e\})} \frac{\mu([\Gamma:\{e\}])}{ [G_0:\{e\}]}  \cdot \iota(\Gamma). 
 \end{equation}
Evaluating (\ref{eq:nicartin}) on the identity element $e$ of $G$ leads to 
\begin{equation}
a_\phi(e)  = \left (\sum_{i = 0}^N \frac{1}{[G_0:G_i]}\right ) \cdot   \# G_0 - 
(N+1)
\end{equation}
Substituting this into (\ref{eq:btriv}) leads to the formula in (\ref{eq:theform}) when $T = \{e\}$.\hfill $\square$

 \section{Constants associated to cyclic subgroups which are not $p$-groups.}
\label{s:numresults}
\setcounter{equation}{0}

 In this section we analyze the constants $b_T = b_T(\phi)$ when $T$ 
 is not a $p$-group.  This is needed to relate the KGB obstruction to the Bertin
 obstruction in the next section.  
 
 \begin{dfn}
\label{def:invariantgroup}
If $H$ is a cyclic subgroup of a finite group $J$, define 
$$\psi(H,J) = \sum_{ H \subset \overline{\Gamma} \subset J, \overline{\Gamma} \mathrm{cyclic}} \mu([\overline{\Gamma}:H]).$$  
\end{dfn}

\begin{prop}
\label{prop:subtle}
Suppose $T$ is a cyclic subgroup of $G$ that contains a non-trivial element of order prime to $p$.  Then
$b_T$ is integral if and only if one of the following alternatives occurs:
\begin{enumerate}
\item[a.] $\N_G(T) = T$.  Then $b_T = 1$.
\item[b.] One has $\psi(T,\C_G(T)) = 0$. Then  $b_T = 0$.
\end{enumerate}
If $T$ has order prime to $p$, then (b) is equivalent to
\begin{enumerate}
\item[b$'$.] $\psi(\{e\},\C_G(T)/T) = 0$. 
\end{enumerate}
\end{prop}

\begin{proof} Let $\pi:\C_G(T) \to J = \C_G(T)/T$ be the quotient homomorphism.  Recall from Notation \ref{def:nontrivT} that $S_G(T)$ is the set of all non-trivial cyclic
subgroups of $G$ which contain $T$. We will first show that there is an injection
\begin{equation}
\label{eq:fnow}
f:S_G(T) \to S_J(\{e\}) \cup \{\{e\}\} \quad \mathrm{defined \ by}\quad f(\Gamma) = \pi(\Gamma).
\end{equation}
which is a bijection if $T$ has order prime to $p$.  
Clearly $f$ is well defined, and since  $\Gamma = \pi^{-1}(\pi(\Gamma))$,  $f$ is injective.  Suppose $T$ has order prime to $p$ and $\overline \Gamma$ is a cyclic subgroup of $J$.  It will suffice to show  $\pi^{-1}(\overline \Gamma)$ is cyclic, since then $\pi^{-1}(\overline \Gamma) \in S(T)$ because $T$ is non-trivial.
Since $\pi^{-1}(\overline \Gamma)$ is an extension of the cyclic group $\overline \Gamma$ by the cyclic subgroup group $T$,  $\pi^{-1}(\overline \Gamma)$ is nilpotent.  So it will suffice to show
$\pi^{-1}(\overline \Gamma)$ has cyclic Sylow subgroups.   The Sylow subgroups of $\pi^{-1}(\overline \Gamma)$ associated to primes $\ell \ne p$
are cyclic since  $G$ has cyclic Sylow subgroups at such $\ell$. When $\ell = p$, the $p$-Sylow subgroup of  
 $\pi^{-1}(\overline \Gamma)$ maps isomorphically to that of $\overline \Gamma$ since $T$ has order prime to $p$,
 so this group is cyclic because $\overline \Gamma$ is cyclic.  This completes the proof that (\ref{eq:fnow}) is a bijection if $T$
has order prime to $p$.  

We now return to arbitrary cyclic $T$ which contain a non-trivial element of order prime to $p$.  
Each $\Gamma \in S_G(T) $ contains a non-trivial element of order prime to $p$, so $\iota(\Gamma) = 1$.  Theorem \ref{thm:nontrivcase} gives 
\begin{equation}
\label{eq:zipper}
b_T = \frac{1}{[\N_G(T):T]} \sum_{\Gamma \in S_G(T)} \mu([\Gamma:T]) \iota(\Gamma) = 
\frac{1}{[\N_G(T):T]} \sum_{\Gamma \in S_G(T)} \mu([\Gamma:T])
\end{equation}
Using the injection (\ref{eq:fnow}) we have
\begin{equation}
\label{eq:okey}
b_T = \frac{1}{[\N_G(T):\C_G(T)]}\cdot  \frac{1}{\# J}\sum_{\overline {\Gamma} \in f(S_G(T)) } \mu(\# \overline{\Gamma}) 
\end{equation}
where $J = \C_G(T)/T$.  Thus if $b_T \in \mathbb{Z}$, we have to have 
\begin{equation}
\label{eq:yup2}
\sum_{\overline{\Gamma} \in f(S_G(T)) } \mu(\# \overline {\Gamma})  \equiv 0\quad \mathrm{
mod}\quad  [\N_G(T):\C_G(T)]\cdot \# J.
\end{equation}

Let $\phi(z)$ be the value of the Euler phi function on an integer $z$.  If $\overline{\Gamma}$ is a 
cyclic subgroup of $J$, there are exactly $\phi(\#\overline{\Gamma})$ generators for $\overline{\Gamma}$, each of which
has order $\# \overline{\Gamma}$.  This leads to the inequality  
\begin{equation}
\label{eq:congruence}
\sum_{\overline{\Gamma} \in f(S(T)) } \mu(\# \overline{\Gamma}) 
\le \sum_{\overline{\Gamma} \subset J, \ \overline{\Gamma} \ \mathrm{cyclic}} 1  = \sum_{g \in J} \frac{1}{\phi(\mathrm{ord}(g))}.
\end{equation}
The sum on the right is bounded by $\# J$, and is less than  $\# J$ unless $J$ is the trivial group.
 We conclude that the congruence (\ref{eq:yup2}) holds if and only if either $\N_G(T) = \C_G(T) = T$
 and $b_T = 1$ or the sum on the left in (\ref{eq:yup2}) is $0$, in which case $b_T = 0$. In the latter
 case, the elements $\Gamma$ of $S_G(T)$ are exactly the cyclic subgroups $\Gamma \subset \C_G(T)$
 containing $T$, which leads to condition (b) in Proposition \ref{prop:subtle} via (\ref{eq:zipper}).
 Conversely, if either condition (a) or (b) hold, then $b_T$ is integral by (\ref{eq:zipper}).  Finally,
 if $T$ has order prime to $p$, then (b) is equivalent to (b') because we have shown
 that the map  $f$ in (\ref{eq:fnow}) is bijective. 
\end{proof}

In view of Corollary \ref{cor:notquot} we have:

\begin{cor}
\label{cor:flipit}
Suppose  $J$ is a subquotient of $G$ with the following property.
There is a cyclic subgroup
$T$ of $J$ which contains a non-trivial element of order prime to $p$ such that $\N_J(T) \ne T$ and $\psi(T,\C_J(T)) \ne 0$. (If $T$ has order prime to $p$, the second condition is equivalent to $\psi(\{e\},\C_J(T)/T) \ne 0$.) Then $ b_{T,J}$ is not
integral.  In particular, 
$G$ is not a weak Bertin group in characteristic $p$, i.e.\  no local $G$-cover in characteristic $p$ has vanishing  Bertin obstruction.  As a result, no such cover can be lifted to characteristic $0$.
\end{cor}

\begin{cor}
\label{cor:greenMat} (Green \cite[Prop. 3.3 and 3.4]{Green}, Green - Matignon \cite{GM}) Suppose that the center $\C(G)$ of $G$ is neither cyclic nor a $p$-group.  Then there is a subquotient $J$ of $G$ and a non-trivial cyclic
subgroup $T$ of $J$ of order prime to $p$ such that $b_{T,J}$ is not integral.
Thus $G$ is not a weak Bertin group.
\end{cor}

\begin{proof}If $\C(G)$ is neither cyclic nor a $p$-group, then there is a subquotient $J$ of $G$ which is the product of a non-trivial cyclic group $T$ of order prime to $p$ with an elementary abelian $p$-group  $E = C_p^2$ of rank $2$.  Then
$J = \N_J(T) = \C_J(T) \ne T$, and $\psi(T,\C_J(T)) = \psi(\{e\},\C_J(T)/T) = \psi(\{e\},E) = 1 + (p+1)\mu(p) = -p \ne 0$.  Thus Corollary \ref{cor:flipit} shows that 
$b_{T,J}$ is not integral, which  completes the proof.
\end{proof}

\section{The Katz-Gabber-Bertin obstruction}
\label{s:KGBsect}
\setcounter{equation}{0}

As in \S \ref{s:imp} we suppose in this section that 
$k$ is an algebraically closed field of characteristic $p$,  that $G$ is
a finite group, and that $\phi:G \to \aut_k(k[[t]])$ is an embedding.  
 Let $\tilde \phi:G \to \aut_k(Y)$
be the action of $G$ on the Katz-Gabber cover $Y \to Y/G = \mathbb{P}^1_k$ associated
to $\phi$, whose properties were recalled in the Introduction.   Let $a_\phi$ be the
Artin character of $G$ associated to $\phi$. 
 
Recall from the introdution that the KGB obstruction of $\phi$ vanishes if there is a field $K$ 
of characteristic $0$ and a $G$ cover $X \to X/G = \mathbb{P}^1_K$ of smooth geometrically
irreducible projective curves
over $K$ such that $\mathrm{genus}(X/H) = \mathrm{genus}(Y/H)$ for all subgroups $H$ of $G$.

\begin{thm}
\label{thm:KGB}If the embedding $\phi:G \to \aut_k(k[[t]])$ lifts to characteristic $0$,
then the KGB obstruction vanishes.
\end{thm}

\begin{proof} 
By the local-global principle for lifting
$G$-covers proved by Bertin and Mezard in 
\cite[Cor. 3.3.5]{BM} (see also \cite[Cor. 2.3]{CGH}), there is a lifting of $\phi$ to characteristic $0$ if and only if there is a lifting of $\tilde \phi:G \to \aut_k(Y)$ to characteristic $0$.  Suppose such a lift exists.  Thus there is a complete discrete valuation ring $R$ having characteristic $0$
and residue field $k$ and an action of $G$ on a 
smooth projective curve $\mathcal{X}$ over $R$ with the following property.  There is an isomorphism
of the special fiber $k \otimes_R \mathcal{X}$ of $\mathcal{X}$ with $Y$ which carries the action of $G$ on $k \otimes_R \mathcal{X}$
to the action of $G$ on $Y$ specified by $\tilde \phi$. 

Let $K$ be the
fraction field of $R$, and let  $X = K \otimes _R \mathcal{X}$.  Since $\mathcal{X}$ is smooth over $R$, flat base change implies $\mathrm{genus}(X/H) = \mathrm{genus}(Y/H)$ for all subgroups $H$ of $G$.  By formal smoothness (cf.\ \cite[Remark I.3.22]{milne} and \cite[17.1.1, 17.5.1]{EGAIV}), each of $X$ and $X/G$ have a point with residue field $K$ because $k$ is algebraically
closed and $\mathcal{X}$ lifts $Y$.  Therefore  $X$ is geometrically irreducible and $X/G$ is isomorphic to $\mathbb{P}^1_K$, so the KGB obstruction of $\phi$ vanishes.
\end{proof}

\begin{thm}
\label{thm:KGBtwo} 
The KGB obstruction vanishes if and only if there is a finite $G$-set $S$ for which the following is true:
\begin{enumerate}
\item[a.]  The stabilizer of each element of $S$ is a non-trivial cyclic subgroup of $G$,
and the character of the action of $G$ on $S$ is
\begin{equation}
\label{eq:chiStwo}
\chi_S = m\cdot  {\rm reg}_G - a_\phi
\end{equation}
for some integer $m$.
\item[b.]  There is a set of representatives $\Omega$ for the $G$-orbits
in $S$ and a subset
 $\{g_t\}_{t\in \Omega}\subset G$ such that $g_t$ generates
the stabilizer $G_t \ne \{e\}$ of $t$ in $G$, $\{g_t\}_{t \in \Omega}$ generates $G$, and the order of $\prod_{t \in \Omega} g_t$ is the index $[G:G_1]$ in $G$ of the wild inertia subgroup $G_1$.
\end{enumerate}
In particular, the vanishing of the KGB obstruction implies the vanishing of the Bertin obstruction,
and both of these obstructions vanish if $\phi$ lifts to characteristic $0$.
\end{thm}

\begin{proof}   Suppose
first  that the KGB obstruction vanishes, so that there is a $G$-cover $X \to X/G = \mathbb{P}^1_K$
of smooth geometrically irreducible curves over  a field $K$ of characteristic $0$ such that 
\begin{equation}
\label{eq:genusH}
\mathrm{genus}(X/H) = \mathrm{genus}(Y/H)
\end{equation}
for all subgroups $H$ of $G$.  By making a base change from $K$ to an algebraic closure of $K$,
we can assume that $K$ is algebraically closed.   For each point $q$ of $X/G = \mathbb{P}^1_K$
let $x(q)$ be a point of $X$ over $q$.  The inertia group $I_{x(q)}\subset G$ is cyclic and
equal to the decomposition group of $q$ since $K$ is algebraically closed of characteristic $0$.
Consider the character function 
\begin{equation}
\label{eq:fX}
f_X = \sum_{q \in X/G} \mathrm{Ind}_{I_{x(q)}}^G a_{x(q)} = \sum_{q \in X/G} (1_{\{e\}}^G - 1_{I_{x(q)}}^G )
\end{equation}
where $a_{x(q)} = 1_{\{e\}}^{I_{x(q)}} - 
 1_{I_{x(q)}}$ is the Artin character of the action of $I_{x(q)}$ on $\hat{O}_{X,x(q)}$.  
 Let $H$ be a subgroup of $G$.  Then by the calculation of relative discriminants in \cite[\S 3]{corps}
 we have
 \begin{equation}
 \label{eq:Xgen}
 2\cdot \mathrm{genus}(X) - 2 - \# H \cdot (2\cdot  \mathrm{genus}(X/H) -2) = \langle f_X,1_H^G \rangle
 \end{equation}
 where $\langle \ , \ \rangle$ is the usual inner product on characters.  We now apply the
 same reasoning to the Katz-Gabber cover $Y \to Y/G = \mathbb{P}^1_k$ associated to
 $\phi$, which is totally ramified over $\infty \in \mathbb{P}^1_k$, tamely ramified with
 cyclic inertia group isomorphic to $Y$ over $0 \in \mathbb{P}^1_k$ and unramified over
 all other points of $\mathbb{P}^1_k$.  Define
 \begin{equation}
 \label{eq:fC}
 f_Y = a_\phi + \delta_C \cdot (1_{\{e\}}^G -  1_C^G)
 \end{equation}
 where the Artin character $a_\phi$ is the one associated to the action of $G$ on the unique point
 of $Y$ over $\infty$ and $\delta_C = 0$ (resp. $\delta = 1$) if $C$ is trivial (resp. non-trivial).  Then
 \begin{equation}
 \label{eq:Cgen}
 2\cdot \mathrm{genus}(Y) - 2 - \# H \cdot (2 \cdot \mathrm{genus}(X/C) -2) = \langle f_Y,1_H^G \rangle
 \end{equation}
 for all subgroups $H$ of $G$.  
 
 We conclude from (\ref{eq:genusH}), (\ref{eq:Xgen}) and (\ref{eq:Cgen})
 that $\langle f_X,1_H^G \rangle = \langle f_Y,1_H^G\rangle$ for all $H$.  Since $f_X$ and $f_Y$
 take rational values and the characters of the form $1_H^G$ generate the $\mathbb{Q}$-vector
 space of rational valued characters, this implies that 
 \begin{equation}
 \label{eq:XCcomp}
 \sum_{q \in X/G} (1_{\{e\}}^G - 1_{I_{x(q)}}^G) = f_X = f_Y = a_\phi + \delta_C \cdot (1_{\{e\}}^G -  1_C^G)
 \end{equation}
We now rewrite this equation using the formula for $-a_\phi$ in part (i) of Proposition \ref{prop:special}.
We can assume $C$ is included in the set $\mathcal{C}$ of representatives for the cyclic subgroups of $G$.  From
(\ref{eq:XCcomp}) we have  
 \begin{equation}
 \label{eq:XCcomp2}
 \delta_C \cdot 1_C^G + \sum_{T \in \mathcal{C}} b_T 1_T^G = \delta_C \cdot 1_C^G -a_\phi  =  \sum_{q \in X/G} 1_{I_{x(q)}}^G - m \cdot 1_{\{e\}}^G = \chi_{S'} - m \cdot 1_{\{e\}}^G
 \end{equation}
 for some integer $m$, where $S'$ is the set of points of $X$ which ramify over $X/G$, and the stabilizer
 in $G$ of each element of $S'$ is cyclic and non-trivial.  
 
 The uniqueness of the values of the $b_T$ was proved in Proposition \ref{prop:special}(i).  Hence 
 (\ref{eq:XCcomp2}) shows that $\delta_C  + b_C$ is the number of $q \in X/G$ such that $I_{x(q)}$ is a conjugate
 of $C$.  This implies $b_C$ is integral.  Thus if $C$ is non-trivial, Proposition \ref{prop:subtle}
 of \S \ref{s:numresults} 
 shows that $b_C \ge 0$.  Since $\delta_C = 1$ if $C$ is non-trivial, we conclude
 in this case that there is a point $q_0 \in X/G$ such that $I_{x(q_0)}$ is a conjugate
 of $C$.  We now define $S$ to be the complement of the $G$-orbit of $x(q_0)$ in $S$
 if $C$ is non-trivial, and we let $S = S'$ if $C$ is trivial.  The uniqueness of the $b_T$
 in (\ref{eq:XCcomp2}) now implies
 $$-a_\phi = \chi_S - m \cdot r_G$$
 where $r_G = 1_{\{e\}}^G$ is the regular representation of $G$.  Since the elements of $S$
 have non-trivial cyclic stabilizers in $G$, $S$ has the properties in part (a) of Theorem \ref{thm:KGBtwo}.
 
 By the classical description the fundamental group of $\mathbb{P}^1_{K} - Z$ (see \cite[Chap.~XIII, Thm. 2.12]{SGAI}),
we can find a set $\Omega'$ of representatives for the $G$-orbits in $S'$
and a generator $g_{t}$ of the inertia group of each $t \in \Omega'$ such that
$\prod_{t \in \Omega'} g_t = e$ is the identity of $G$ and $\{g_t\}_{t \in \Omega'}$ generates $G$.  Letting $\Omega = \Omega' \cap S$ shows
$\prod_{t \in \Omega} g_t$ is either trivial if $C$ is trivial or is a generator of a conjugate of the cyclic group $I_{x(q_0)}$
of order $\# C$ if $C$ is non-trivial.  This proves that $S$ has all properties stated in Theorem
\ref{thm:KGBtwo}.

Conversely, suppose that we have an $S$ with the properties stated in Proposition \ref{prop:special}.
Let $K$ be an algebraically closed field of characteristic $0$. By reversing the above
steps, we can construct a $G$-cover $X \to \mathbb{P}^1_K$ of smooth irreducible 
curves for which (\ref{eq:XCcomp}) holds.  Now (\ref{eq:fX}) and (\ref{eq:fC}) establish
(\ref{eq:genusH}) for all $H$.

The vanishing of the Bertin obstruction is equivalent, by definition, to condition (a) of Theorem \ref{thm:KGBtwo}. 
So the Bertin obstruction vanishes if the KGB obstruction does.  Both obstructions vanish if $\phi$ lifts to characteristic $0$ by Theorem \ref{thm:KGB}.
\end{proof}

\begin{rem}
\label{rem:tatel} Using \cite[\S 3]{corps}, the vanishing KGB obstruction can also be formulated in the following way. There
is a $G$-cover of smooth geometrically irreducible projective curves $X \to X/G$
such that for all primes $\ell$ different from the characteristic of $k$, the $\ell$-adic
Tate modules of $X$ and of $Y$ are $G$-isomorphic, where $Y \to Y/G$ is
as before the Katz-Gabber cover associated to $\phi$.  We will not need this
interpretation in what follows.
\end{rem}

\section{Functorality.}
\label{s:functor}
\setcounter{equation}{0}

Suppose $\phi:G \to \aut_k(k[[t]])$ as in \S \ref{s:KGBsect} is given.  Let
$N$ be a normal subgroup of $G$, and let $\Gamma = G/N$.   Define $\phi^\Gamma: \Gamma \to
\aut_k(k[[t]]^N)$ by $\phi^\Gamma(\gamma)= \phi(g)$ if $g \in G$ has image $\gamma \in \Gamma$.  Since
$k[[t]]^J$ is a complete discrete valuation ring with residue field $k$, it is isomorphic
to $k[[z]]$ for some $z \in k[[t]]^J$ by \cite[Prop. II.5]{corps}.  If $H$ is an arbitrary subgroup of $G$,
define $\phi_H:H \to \aut_k(k[[t]])$ to be the restriction of $\phi$ from $G$ to $H$.

\begin{thm}
\label{thm:functorial}    If the Bertin obstruction (resp.\ the KGB obstruction) vanishes for $\phi$,
then then same is true of $\phi^\Gamma$.  If the Bertin obstruction of $\phi$ vanishes,
then it does for $\phi_H$.  
\end{thm}

To prove this result we need a lemma concerning characters $\chi$ of $G$.  Define a character $\chi^\sharp$
of $\Gamma = G/N$ by
\begin{equation}
\label{eq:sharpdef}
\chi^\sharp(\gamma) = \frac{1}{\# N} \sum_{g \in G, \  q(g) = \gamma} \chi(g)
\end{equation}
where $q:G \to G/N$ is the quotient map.  By \cite[Prop.~VI.3]{corps}, 
\begin{equation}
\label{eq:quotientrel}
a_{\phi^\Gamma} = (a_\phi)^\sharp.
\end{equation}

The next result follows directly from Frobenius reciprocity \cite[Thm.~13, Chap.~7]{SerreRep}.  

\begin{lemma}
\label{lem:backsl}If $S$ is a left $G$-set, let $N\backslash S$ be the $G/N$ set formed by the orbits $Ns$ of
elements $s \in S$ under the left action of $N$.   Let $\chi_S$ (resp.\ $\chi_{N \backslash S}$)
be the character of the permutation representation of $G$ (resp.\ $G/N$) defined by $S$ (resp.\ $N \backslash S$).  Then
\begin{equation}
\label{eq:easy}
\chi_{N\backslash S} = \chi_S^\sharp
\end{equation}
\end{lemma}

\noindent {\bf Proof of Theorem \ref{thm:functorial}:}\hfill  \medbreak Suppose that the Bertin obstruction vanishes for $\phi$, and that $S$ is
as in Proposition \ref{prop:special}(ii)(a).  Let $S''$ be the $G/N$ set
$N \backslash S$.  Let $S''_0$ be the set of the elements of $S''$ which
have trivial stabilizer, and let $S'$ be the complement $S'' - S''_0$. The
stabilizers in $G/N$ of the elements of $S'$ are then non-trivial, and these
are cyclic because they are the images in $G/N$ of the cyclic stabilizers
of  elements of $S$.  
By hypothesis,
\begin{equation}
\label{eq:permrep}
\chi_S = m\cdot  {\rm reg}_G - a_\phi
\end{equation}
for some integer $m$.  By (\ref{eq:easy}) and (\ref{eq:quotientrel}), one has
$$\chi_{S''} = \chi_S^\sharp = m \cdot ({\rm reg}_G)^\sharp - a_{\phi}^\sharp
= m \cdot ({\rm reg}_{G/N}) - a_{\phi^\Gamma}$$
Since $\chi_{S''}$ and $\chi_{S'}$ differ by an integer multiple of the character
of the regular representation of $G/N$, we conclude that 
\begin{eqnarray}
\label{eq:permrepsound}
\chi_{S'} = m' \cdot ({\rm reg}_{G/N}) - a_{\phi^\Gamma}
\end{eqnarray}
for some integer $m'$, so $S'$ satisfies condition (ii)(a) of Proposition \ref{prop:special} for $\phi^\Gamma$.

Suppose now that the KGB obstruction vanishes for $\phi$.  Let $S$, $\Omega$ 
and $\{g_t : t \in \Omega\}$ be as in Theorem \ref{thm:KGBtwo}. 
By hypothesis, $\Omega$ is a set of representatives for the $G$ orbits in $S$,
 $\{g_t\}_{t\in \Omega}$ is a subset of $G$ such that  $g_t$ generates
the stabilizer $G_t \ne \{e\}$ of $t$ in $G$, $\{g_t\}_{t \in \Omega}$ generates $G$, and $\prod_{t \in \Omega} g_t$
has order $[G:G_1]$.  

Let $t' = Nt$ be the image of $t \in \Omega$ in $S'' = N \backslash S$.
One has $G_t \subset N$ if and only if the stabilizer $G_{t'}$ is trivial;
this is true if and only if $t' \in S''_0$.  Define $$\Omega' = \{t' = Nt:  t \in \Omega, t' \not \in S''_0\}.$$  Now
$\Omega'$ is a set of representatives for the $G/N$ orbits in $S'$, and for $t' \in \Omega'$ the image $g_{t'}$
of $g_t$ in $G/N$ is non-trivial and generates the stabilizer of $t'$ in $G/N$.  
By hypothesis, $\{g_t: t \in \Omega\}$ generates $G$ and $\prod\{g_t:t \in \Omega\}$ has
order $[G:G_1]$.  The image of $g_t$ in $G/N$ is non-trivial if and only if
$t' = Nt \not \in S''_0$, so we conclude that $\{g_{t'}:t' \in \Omega\}$ generates $G/N$,
and the image of $\gamma = \prod \{g_t: t \in \Omega\}$ in $G/N$ is $\gamma' = \prod \{g_{t'}:t' \in \Omega'\}$.
Thus $\gamma'$ has order $[G/N:G_1N/N]$.   Here $G_1N/N$ is the $p$-Sylow subgroup
of $G/N$, so  $\Omega'$ and $\{g_{t'}:t' \in \Omega'\}$ satisfy condition (b) of
Theorem \ref{thm:KGBtwo}.

We now have to show that if the Bertin obstruction of $\phi$ vanishes,
then it does for $\phi_H$ for all subgroups $H$ of $G$.  Since the residue
field $k$ is algebraically closed, \cite[Prop. VI.4]{corps} shows
\begin{equation}
\label{eq:balance}
\mathrm{res}_G^H a_\phi = \lambda \cdot \mathrm{reg}_H + a_{\phi_H}
\end{equation}
where $\lambda$ is the valuation in $k((t))^H$ of the discriminant
of $k((t))$ over $k((t))^H$.  This implies that if $S$ is a $G$-set satisfying
condition (ii)(a) of Proposition \ref{prop:special} for $\phi$, then the restriction of $S$
to $H$ satisfies this condition for $\phi_H$.  This completes the proof 
of Theorem  \ref{thm:functorial}.\hskip 0.1in $\square$

\begin{rem} In a later paper we will show that the KGB obstruction for $\phi_H$
vanishes if that of $\phi$ vanishes;  we will not need this result in this paper.
\end{rem}

\begin{notn}
\label{notn:bprimedef}  Let $J$ be a subquotient of $G$, i.e.\  a quotient
of a subgroup $H$ of $G$ by a normal subgroup $D$ of $H$. Suppose $T$ is
a cyclic subgroup of $J$.  Define $b_{T,J} = b_{T,J}(\phi)$ to be the constant
$b_T$ appearing in Proposition \ref{prop:special}  when $G$ is replaced
by $J$ and $\phi$ is replaced by the induced injection $\phi_J:J \to \mathrm{Aut}_k(k[[t]]^D)$. 
\end{notn}

Combining Theorem \ref{thm:functorial} with Proposition \ref{prop:special} we obtain:

\begin{cor}
\label{cor:notquot}  Suppose the Bertin obstruction of $\phi$ vanishes.  With the notations of
Notation \ref{notn:bprimedef}, one has $0 \le b_{T,J} \in \mathbb{Z}$ for all non-trivial cyclic subgroups $T$ of $J$.
\end{cor}

 \begin{cor}  
\label{cor:subquots}  Let $G$ be a finite group, and suppose $H$ is a quotient group
of $G$. 
\begin{enumerate}
\item[a.]  If $H$ is not a Bertin group (resp.\ almost Bertin group), then
$G$ is not a Bertin group (resp.\ almost Bertin group).
\item[b.]  If $G$ has a subquotient $J$ which is an not a weak Bertin group, then $G$
is not a weak Bertin group.
\end{enumerate}
\end{cor}

\begin{proof} In view of Theorem \ref{thm:functorial}, part (a) is clear for Bertin groups,
and to show part (a) for almost Bertin groups  it will suffice
to prove the following.  Suppose that $\phi_H:H \to \mathrm{Aut}_k(k[[t]])$ is an injection
with the property that the Artin character $a_{\phi_H}$ has the property that 
$-a_{\phi_H}(\tau)  \ge M$ for some integer $M$ and all non-trivial elements $\tau \in H$
of $p$-power order.  It will be enough to show that there is an injection $\phi_G:G \to \mathrm{Aut}_k(k[[z]])$ inducing $\phi_H$ when we identify $k[[z]]^T$ with $k[[t]]$ when $T$ is the kernel
of the surjection $G \to H$, and for which $-\phi_G(\tau') \ge M$  for all non-trivial
elements $\tau'$ of $G$ of $p$-power order.  This statement is a consequence of
the parts (i) and (ii) of Proposition 
 \ref{prop:subquots}(i) of Appendix 1.
Part (b) of Corollary \ref{cor:subquots} follows directly from Theorem \ref{thm:functorial},
which shows that if the Bertin obstruction of some injection $\phi:G \to \mathrm{Aut}_k(k[[t]])$
vanishes, then Bertin obstruction of the induced local $J$-cover would also have to vanish.
\end{proof}

 \begin{rem}
 \label{rem:subobs}
 Suppose $H$ is a subgroup of $G$, and that $H$ is not a Bertin group.  In general
 this need not imply $G$ is not a Bertin group, since it may not be possible to 
 realize a local $H$-cover with non-zero Bertin obstruction as the restriction of a
 local $G$-cover.  We will show later in Proposition \ref{prop:evenqsemi} that this occurs, for example, when $H$
 is the quaternion group of order $8$ and $G$ is a generalized quaternion group
 of order at least $16$.
 \end{rem}

 \section{The reduction to $p$-groups.}
 \label{s:reduce}
 \setcounter{equation}{0}
 
 Throughout this section we suppose given an injection $\phi:G \to \aut_k(k[[t]])$ as in \S \ref{s:KGBsect}.  
 
 By Theorem \ref{thm:functorial}, if the Bertin obstruction of $\phi$
vanishes, then so does the Bertin obstruction of the restriction
$\phi_P$ of $\phi$ to a $p$-Sylow subgroup $P$ of $G$.  In this section we state our main result concerning
exactly which further conditions $G$ and $\phi$ must satisfy in
order for the Bertin obstruction of $\phi$ to vanish provided
that of $\phi_P$ vanishes.  The proof of this result is given in
\S \ref{s:numresults} - \S \ref{s:pgroups}.  

We begin with a well-known result about the structure of $G$ which follows
from a theorem of P. Hall \cite[Thm. 18.5]{Asch}.

\begin{lemma}
\label{lem:Gstruct}
The group $G$ is the semi-direct product of a normal $p$-group $P$ and a
cyclic group $C$ of order prime to $p$.  
Let $z \in G$ be an element which is not of $p$-power order.
 Let  $m$ be the smallest positive integer such that
 $w = z^m$ has order prime to $p$.  Let $t$ be the unique element of $C$ having the same image
as $w$ in $G/P$.  Then $w$ is conjugate to $t$ in $G$.  In particular,
all subgroups of $G$ having a given prime-to-$p$ order are conjugate.
\end{lemma}


\begin{notn}
\label{notn:Bprime}
Let $T $ be a non-trivial cyclic
subgroup of $G$.  Define
\begin{equation}
\label{eq:Btprime}
b'_T = b'_{T,G} = \sum_{P \not \supset \Gamma \in S(T)} \mu([\Gamma:T])\quad \mathrm{and}\quad b''_T = b''_{T,G} = \sum_{\Gamma \in S(T)} \mu([\Gamma:T])
\end{equation}
where as before $S(T) = S_G(T)$ is the set of cyclic subgroups $\Gamma$ of $G$
which contain $T$.
\end{notn}

\begin{dfn} 
\label{def:Teich} Let $u$ be a uniformizer in the discrete valuation ring $k[[t]]^{\phi(P)}$.
Since $G = P.C$, the field $k((t))^{\phi(P)}$ is a cyclic $C = G/P$ extension 
of $k((t))^{\phi(G)}$ via $\phi$.  Let $\theta:C \to k^*$ be the faithful character
defined by $\phi(\sigma)(u)/u \equiv \theta(\sigma)$ mod $uk[[u]]$ for $\sigma \in C$.
Suppose $C_1$ is a subgroup of $C$ and that $j$ is an integer such that
the restriction $\theta^j|_{C_1}$ of $\theta^j$ to $C_1$ takes values
in $(\mathbb{Z}/p)^*$.  We define the \textit{Teichm\"uller} lift of $\theta^j|_{C_1}$
to be the unique character $\psi:C_1 \to \mathbb{Z}_p^*$ whose reduction
mod $p$ is equal to $\theta^j|_{C_1}$.
\end{dfn}

\begin{rem}
\label{rem:clarify}  The definition of $\theta$ does not depend on
the choice of uniformizer $u$.  In Lemma \ref{lem:urk} below we consider the character $\theta_0:C \to k^*$ defined by $\theta_0(\sigma) \equiv \sigma(t)/t$ mod
$tk[[t]]$.  This also does not depend on the choice of the uniformizing parameter
$t$.  Hence on letting $u = \prod_{\gamma \in P} \gamma(t) = \mathrm{Norm}_P(t)$,
we see that $\theta(\sigma) = \theta_0(\sigma)^{\# P}$.  Since $\# P$
is a power of $p$, we conclude that $\theta(\sigma) = \theta_0(\sigma)$
if $\theta(\sigma) \in (\mathbb{Z}/p)^*$.
\end{rem}

\begin{notn}
\label{notn:charlift}
Suppose $T$ is a cyclic subgroup of $P$.  The normalizer $\N_{C}(T)$
of $T$ in $C$ acts on $T$ by conjugation.  Since $T$ is cyclic
of order a power of $p$, and $\N_{C}(T)$ is cyclic of order prime
to $p$, we write $\chi_T:\N_{C}(T) \to \mathbb{Z}_p^*$ for the
unique character 
such that $xyx^{-1} = y^{\chi_T(x)}$ for $x \in \N_{C}(T)$ and $y \in T$.
\end{notn}

\begin{thm}
\label{thm:reducetop}
The Bertin obstruction of $\phi$ vanishes if and only if all the following
conditions hold:
\begin{enumerate}
\item[a.] The Bertin obstruction of the restriction $\phi_P$ of $\phi$ to
$P$ vanishes.
\item[b.] If $t$ is a non-trivial element of $C$, the centralizer $\C_G(t)$
of $t$ in $G$ is cyclic and equal to the product group $\C_P(t) \times C = 
\C_P(C) \times C$.
\item[c.] For each non-trivial cyclic subgroup $T$ of $P$, both of the following
statements are true:
\begin{enumerate}
\item[i.] $b'_{T,G} \equiv 0$ mod $[\N_P(T):T]\mathbb{Z}$.
\item[ii.] $[\N_P(T):T] b_{T,P} \ge -b'_{T,G}$.
\end{enumerate}
\item[d.] For each non-trivial cyclic subgroup $T$ of $P$,
one of following is true:
\begin{enumerate}
\item[i.] The centralizer $\C_{C}(T)$ of $T$ in $C$ is non-trivial, and 
$$b''_{T,G} \equiv 0\quad \mathrm{mod} \quad \#\N_{C}(T)\mathbb{Z}.$$
\item[ii.] The centralizer $\C_{C}(T)$ is trivial.  Then the restriction of $\theta$
to $\N_{C}(T)$ is faithful, takes values in $(\mathbb{Z}/p)^*$ and has Teichm\"uller
lift $\chi_T^{-1}$, and 
$$b'_{T,G} \equiv 0 \quad \mathrm{mod} \quad \#\N_{C}(T)\mathbb{Z}.$$
\end{enumerate}
\end{enumerate}
\end{thm}

The proof of this result is completed at the end of \S \ref{s:pgroups}
using the results in \S \ref{s:numresults}, \S \ref{s:nonpgroups} and
\S \ref{s:pgroups}. 

We should point out that conditions (b), (c)(i) and (d)(i) of the above
theorem are purely group theoretic, in the sense that they
do not depend on $\phi$.  Condition (c)(ii)
should be interpreted as saying that the wild ramification groups
of $G$ are sufficiently large relative to the constant $-b_{T',G}$
which does not depend on the ramification filtration of $G$ determined by $\phi$.  Condition
(d)(ii) (when it applies) should be viewed as saying that the
conjugation action of $\N_{C}(T)$ on $T$ is the inverse of
the Teichm\"uller lift of the restriction of the tame character $\theta:C \to k^*$
to $\N_{C}(T)$.  Note that $\theta$ does not depend on $T$ and is
also independent of $\phi$.  Thus condition (d)(ii) does involve the arithmetic information
contained in $\theta$, but this information is connected only
with tame ramification.  The higher ramification filtration of $G$ therefore
enters into only condition (a) and condition (c)(ii) of Theorem \ref{thm:reducetop}.

\section{Obstructions associated to cyclic subquotients which are not $p$-groups.}
\label{s:nonpgroups}
\setcounter{equation}{0}

Throughout this section we suppose that the finite group $G = P.C$ is a semi-direct product of a 
normal $p$-group $P$ with a cyclic prime to $p$ group $C$.
Our goal is to prove the following proposition concerning the constants $b_{T,J} = b_{T,J}(\phi)$ in Corollary \ref{cor:notquot}:

\begin{prop}
\label{prop:infoprime}
The following conditions are equivalent.
\begin{enumerate}
\item[a.] For each cyclic subgroup $T$ of a subquotient $J$ of $G$ such
that $T$ is not a $p$-group, then the constant $b_{T,J} = b_{T,J}(\phi)$ in Corollary \ref{cor:notquot}  has the property that $0 \le b_{T,J} \in \mathbb{Z}$.
\item[b.]  If $t$ is a non-trivial element of a cyclic subgroup $C_0$ of $G$ of maximal prime to $p$ order,
the centralizer $\C_G(t)$ is cyclic and equal to the
product group $\C_P(t) \times C_0 = \C_P(C_0) \times C_0$.
\end{enumerate}
If one (and hence both) of these conditions hold, then all of the constants
$b_{T,J}$ in part (a) are either $0$ or $1$.  Both conditions hold
if the Bertin obstruction of $\phi$ vanishes.  
\end{prop}

\medbreak

\noindent {\bf Proof that Condition (a) of Proposition \ref{prop:infoprime}
implies condition (b).}
\medbreak

We assume that condition (a) of the proposition holds.  Condition (b) holds trivially if $P$ is trivial, so we assume  $P$ is not trivial.  All of the cyclic subgroups of $G$ of maximal prime-to-$p$ order are conjugate to $C$ by
Lemma \ref{lem:Gstruct}.  Hence to prove (b), we can reduce to the case $C_0 = C$.  The fact that $\C_G(C) = \C_P(C) \times C$ is clear because $G$ is the 
semi-direct product of the abelian group $C$ with $P$.  

We first show that $\C_P(t)$ is cyclic.  Since condition (a) applies to all subquotients of $G$, to show $\C_P(t)$ is cyclic, 
we can replace $G$ by $\C_P(t) \times \langle t \rangle$.  We may thus temporarily assume that 
$G$ is the product group $P \times \langle t \rangle$,
with $\C_P(t) = P$.  Let $P_1$ be the Frattini subgroup of $G$, so that $P_1 = [P,P] \cdot P^p$ is the normal subgroup of $G$ generated
the commutator subgroup $[P,P]$ of $P$ and the $p^{th}$ powers of elements of $P$.  Thus $P/P_1$ 
is an elementary abelian group of rank equal to that of $P$.  
By Corollary \ref{cor:greenMat} and the fact that condition (a) applies to all 
subquotients, we may conclude that if $H$ is a 
subquotient of $G $, then $\C(H)$
must be either a $p$-group or a cyclic group.  But $G/P_1 = (P/P_1) \times \langle t \rangle$
is abelian and not a $p$-group, so this group must be cyclic.  Hence the $p$-Frattini quotient
$P/P_1$ is cyclic, so $P = \C_P(t)$ itself must be cyclic, as asserted.

We now drop the assumption that $G = P \times \langle t \rangle$.  To show the last
equality in part (b) of Proposition \ref{prop:infoprime}, it will suffice to prove 
 \begin{equation}
\label{eq:todo}
\C_P(t) = \C_P(C).
\end{equation}

To prove (\ref{eq:todo}), we can use
 induction on $\# C$ to reduce to  the case in which the  index of
$\langle t \rangle $ in $C$ is a prime number by replacing $G$ by $P.\C_P(t)$.  
We have already shown that $\C_P(t)$
is cyclic, and clearly $\C_P(C) \subset \C_P(t)$.  We now check that $C$
normalizes $\C_P(t)$.  Suppose $c \in C$ and $g \in \C_P(t)$.  Then $cgc^{-1}  = g' \in P$ since $P$ is normal, and $$g' t g'^{-1} = 
cgc^{-1} t cg^{-1} c^{-1} = c g t g^{-1} c^{-1} = ctc^{-1} = t$$ since
$c$ and $t$ are in the abelian group $C$ and $g \in \C_P(t)$.  Thus the
subgroup $\C_P(t). C$ generated by $\C_P(t)$ and $C$
is a semi-direct product of these two groups.   We now replace $G$ by $\C_P(t). C$ to
be able to assume that $P = \C_P(t)$ is cyclic
and that $[C:\langle t \rangle] = \ell$ is prime (and prime to $p$).  If $C$ centralizes
$P$, then (\ref{eq:todo}) holds.  So we will now assume that $C$ does not centralize
$P$ and derive a contradiction.

Since $t$ commutes
with $P = \C_P(t)$ and with all of $C$, we find that the centralizer
$\C_G(\langle t \rangle)$ is equal to all of $G$.  We now apply 
Proposition \ref{prop:subtle} to the subgroup $T = \langle t \rangle$ of $G$.  Since
$\N_G(T) = \C_G(T) = G$ is not $T$,  this proposition implies that 
\begin{equation}
\label{eq:yuppie}
\psi(\{e\},\C_G(T)/T) = \sum_{{\rm cyclic\ }H \subset \C_G(T)/T} \mu(\# H) = 0
\end{equation}
where the sum is over the cyclic subgroups $H$ of $\C_G(T)/T$, including the trivial subgroup, 
and $\mu$ is the Mobius function.  

In our situation, $\C_G(T)/T = G/\langle t \rangle =
P . (C/\langle t \rangle)$ is a non-trivial semi-direct product of the cyclic $p$-group
$P$ with the cyclic group $C/\langle t \rangle$ of prime order $\ell \ne p$.  Then
$C/\langle t \rangle$ acts non-trivially by conjugation on every element of $P$.
It follows that any element $g \in \C_G(T)/T$ which does not lie in $P$ must
have order exactly $\ell$, since otherwise conjugation by $g$ would fix
the non-trivial element $g^{\ell}$ of $P$.  The number of elements of $\C_G(T)/T$ which do not lie in $P$ is
$(\ell -1)(\# P)$, and these generate the $\# P$ subgroups
$H$ of order $\ell$ in $\C_G(T)/T$.  The other groups $H$ appearing on the right
hand side of (\ref{eq:yuppie}) are subgroups of $P$, and the only groups $H
$ of this kind for which $\mu(\#H) \ne 0$ are the trivial group $\{e\}$ and
the unique subgroup $P_0$ of order $p$ in $\C_G(T)/T$.  Thus 
\begin{equation}
\label{eq:computeit}
\psi(\{e\},\C_G(T)/T) = \sum_{H \subset \C_G(T)/T} \mu(\# H)  = \mu(1) + \mu(p) + \# P  \cdot \mu(\ell) = - \# P \ne 0
\end{equation}
This contradicts (\ref{eq:yuppie}), which completes the proof that part (a) of Proposition
\ref{prop:infoprime} implies part (b).
\medbreak

\noindent {\bf Conclusion of the proof of Proposition \ref{prop:infoprime}.}
\medbreak

We first prove two lemmas.

\begin{lemma}
\label{lem:passage}
Every subquotient $J$ of $G$ is the semi-direct product of a
$p$-group with a cyclic prime to $p$-group.  If condition (b) of Proposition \ref{prop:infoprime}
holds for $G$, then it also holds when $G$ is replaced by $J$.
\end{lemma}

\begin{proof}  The first statement is a consequence of Hall's Theorem \cite[Thm. 18.5]{Asch} together
with the fact that $J$ is an extension of a cyclic group of order prime to $p$ by
a normal $p$-subgroup.  Suppose now that $G$ satisfies condition (b)
of Proposition \ref{prop:infoprime}.  It is clear that every subgroup of $G$
then satisfies this condition.  We are thus reduced to showing that $J = G/H$
satisfies this condition for all normal subgroups $H$ of $G$.  It is enough to consider
prove this when $H$ is either a $p$-group or a cyclic prime to $p$ group.

   Suppose first that $H$ is a $p$-group.   Then $H \subset P$
and $J = (P/H).C$ is the semi-direct product
of the $p$-group $P/H$ with $C$.  Let $t$ be a non-trivial
element of $C$.  By hypothesis, $\C_G(t) = \C_P(t) \times C$
is cyclic.  Hence to show $\C_J(t) = \C_{P/H}(t).C$
is cyclic, it will suffice to show that the quotient
homomorphism $P \to P/H$ gives a surjection $\C_P(t) \to \C_{P/H}(t)$.
Here $t$ acts on $P$ and $P/H$ by conjugation, so we are to show
that  the invariants $P^{\langle t \rangle}$
surject onto $(P/H)^{\langle t \rangle}$.  Since $P$ is a $p$-group
and $t$ has order prime to $p$, this follows from the taking the
non-abelian
cohomology with respect to $\langle t \rangle$ of the sequence of $1 \to H \to P \to P/H \to 1$ (see \cite[Chap.~VII, Annexe]{corps}).  

Finally, suppose that $H$ is cyclic of order prime to $p$.    By Lemma \ref{lem:Gstruct},
$H$ is a subgroup of $C$ since it is
conjugate to such a subgroup and is normal.  Then $P.H$ contains
the normal subgroups $P$ and $H$, so $P.H$ is isomorphic
to $P \times H$, and $H$ and $P$ commute.  Since $H \subset C$ commutes
with $C$, we conclude $H$ is in the center of $G$.  Suppose $t' \in J = G/H = P.(C/H)$ has order prime to $p$.
Since $H$ is central, it is clear that condition (b) of Proposition \ref{prop:infoprime}
implies 
$\C_{J}(t') = \C_P(t') \times (C/H)$ is cyclic, so Condition (b) holds for $J$.   
\end{proof}

For $z \in G$, we will write $\N_G(z)$ for $\N_G(\langle z \rangle)$, where $\langle z \rangle$ is the subgroup
generated by $z$.

\begin{lemma}
\label{lem:zap}
Suppose that condition (b) of Proposition \ref{prop:infoprime} holds.  
Let $z \in G$ be an element which is not of $p$-power order.  Let $m$ be the  smallest positive such that $w = z^m$ is a (non-trivial) element
of order prime to $p$.
\begin{enumerate}
\item[a.]  The group $\C_G(z) = \C_G(w)$ is cyclic and conjugate to
$\C_P(C) \times C$.
\item[b.]
If $\N_G(z)$ properly contains $\langle z \rangle$, then so does $\C_G(z)$.
 \end{enumerate}
 \end{lemma}

\begin{proof} 
By Lemma \ref{lem:Gstruct}, $w$ is conjugate to an element $t$ of $C$.  By replacing $z$ by a conjugate of itself, we can assume that $w = t \in C$.  Since we assume that condition (b) of Proposition \ref{prop:infoprime}
holds, $\C_G(w) = \C_G(t) = \C_P(C) \times C$ is cyclic.  We have $z \in \C_G(w)\supset \C_G(z) $ since $w$ is a power of $z$.  Because $\C_G(w)$ is abelian,
this implies $\C_G(w) \subset \C_G(z)$, so $\C_G(z) = \C_G(w) = \C_P(C) \times C$.
This proves part (a).  

To show part (b), we assume to the contrary that 
\begin{equation}
\label{eq:cgz}
\C_G(z) = \C_P(C) \times C = \langle z \rangle.
\end{equation}  Note that this forces forces $w = t$ above to be a generator of $C$. 
It will suffice to show
\begin{equation}
\label{eq:noteq}
\N_G(z)  \subset   \C_G(z)
\end{equation}  
since then $\C_G(z) = \N_G(z)$ will equal $\langle z \rangle$.

The group $\langle z \rangle$ is obviously normal
in $\N_G(z)$, and $\langle w \rangle = \langle t \rangle  =  C \subset  \langle z \rangle$ is characteristic in $\langle z \rangle$.  Hence $C$ 
is a normal subgroup of $\N_G(z)$.  The group $\N_G(z) \cap P$ 
is also normal in $\N_G(z)$ since $P$ is normal in $G$.  Since
$\N_G(z) \cap P$ and $C$ have coprime orders, and the product of these
orders is $\# \N_G(z)$, we conclude that $\N_G(z)$ is isomorphic to
the product group $(\N_G(z) \cap P) \times C$.  This means that $C$ commutes
with $\N_G(z) \cap P$.  

Thus $\N_G(z) \cap P$ is contained in the cyclic group
$\C_G(C) = \C_P(C) \times C$.  Hence $\N_G(z) \cap P$ is abelian.  Since 
$\langle z \rangle \cap P$  is contained in $\N_G(z) \cap P$, this means
that $\N_G(z) \cap P$ centralizes $\langle z \rangle \cap P$.  However,
we have already shown that $C$ commutes with $\N_G(z) \cap P$.  Thus
$\C_G(\N_G(z) \cap P) $ contains both $\langle z \rangle \cap P$ and $C = \langle w \rangle$,
and the latter two groups generate $\langle z \rangle$.  Hence 
$\C_G(\N_G(z) \cap P) $ contains $\langle z \rangle$, so $\N_G(z) \cap P$
is contained in $\C_G(z)$.  

We now use the fact that  $\N_G(z)$ is generated by 
$\N_G(z) \cap P \subset \C_G(z)$ and $C = \langle w \rangle \subset \C_G(w)$, where we
have shown $\C_G(w) = \C_G(z)$ already. This
implies $\N_G(z) \subset \C_G(z)$ and proves (\ref{eq:noteq}).
 \end{proof}

 \begin{cor}
 \label{cor:whoop}Condition (b) of Proposition \ref{prop:infoprime}
implies that if $T$ is a cyclic subgroup of a subquotient $J$ of $G$ such
that $T$ is not a $p$-group, then $b_{T,J}$ is equal to $0$ or $1$.  
 \end{cor}
 
 \begin{proof} By Lemma \ref{lem:passage}, it will be enough to consider
 the case in which the 
 subquotient $J$ is $G$ itself.  By Proposition \ref{prop:subtle},
 $b_T = 1$ if $\N_G(T) = T$, and $b_T = 0$ provided
 \begin{equation}
 \label{eq:yippe}
 \psi(T,\C_G(T)) = \sum_{T \subset W \subset \C_G(T), W \ \mathrm{cyclic}} \mu([W:T]) = 0.
 \end{equation}
 To prove that one or the other of these alternatives applies, let $z$ be a generator
 of $T$, and suppose that $\N_G(T) \ne T$.  By Lemma \ref{lem:zap},
 $\C_G(T)$ is a cyclic group which strictly contains $T$. We conclude
 that $\psi(T,\C_G(T)) = \psi(\{e\},\C_G(T)/T) = \sum_{d|[\C_G(T):T]} \mu(d) = 0$ in (\ref{eq:yippe}), so
 the corollary holds.
 \end{proof}  

 The last two assertions in 
 Proposition \ref{prop:infoprime} now follow from Corollaries \ref{cor:whoop} and \ref{cor:notquot},
and this completes the proof.\hskip .1in $\square$

\section{Obstructions associated to cyclic  $p$-subgroups.}
\label{s:pgroups}
\setcounter{equation}{0}

We will fix the following hypotheses and notation throughout this
section.

\begin{hyp}
\label{hyp:setup}
Let $\phi:G \to \mathrm{Aut}_k(k[[t]])$ be an injection, and write
$G$ as the semi-direct product $P.C$ of a normal $p$-group $P$
and a cyclic group $C$ of order prime to $p$.  Let $T $ be a non-trivial cyclic
subgroup of $G$ of $p$-power order.  Define $b'_T = b'_{T,G}$ as in 
(\ref{eq:Btprime}).
We assume finally that if $t$ is a non-trivial element of $C$, then $\C_G(t) = \C_P(t) \times C$ is cyclic.
\end{hyp}

Note that by Proposition \ref{prop:infoprime}, the final assumption
in this hypothesis holds if $\phi$ has vanishing Bertin obstruction. 

The goal of this section is to compare the constants $b_{T} = b_{T,G}$
and $b_{T,P}$.  

\begin{lemma}
\label{lem:blahT}After replacing $T$ by  a conjugate subgroup, 
the centralizer $\C_G(T)$ is the semi-direct product $\C_P(T).\C_{C}(T)$
and the normalizer $\N_G(T)$ is $\N_P(T).\N_{C}(T)$.
\end{lemma}

\begin{proof} The group $\N_P(T) = \N_G(T) \cap P$ is normal in $\N_G(T)$ since
$P$ is normal in $G$.  The quotient group $\N_G(T)/\N_P(T)$ injects
into the cyclic prime to $p$-group $G/P$, so $\N_G(T)$ is the semi-direct
product of $\N_P(T)$ with a subgroup $C'$ of order prime to $p$.  By Lemma \ref{lem:Gstruct} we can replace $T$ by a conjugate of itself to be able to
assume that $C' \subset C$.  After this replacement we have $C' = \N_{C}(T)$.  Since
$\C_G(T) \subset  \N_G(T)$, we can write each element of $\C_G(T)$
in a unique way in the form $\alpha  \beta$ with  $\alpha \in \N_P(T)$
and $\beta \in C' = \N_{C}(T)$.  Then the conjugation action of $\beta$
on $T$ must be the inverse of the conjugation action of $\alpha$ on $T$.
Since $\beta$ and $\alpha$ have co-prime orders, this implies
that each of these actions are trivial, so $\alpha \in \C_P(T)$
and $\beta \in \C_{C}(T)$.  Thus $\C_G(T) = \C_P(T).\C_{C}(T)$.
\end{proof}

\begin{cor}
\label{cor:numconseq}
One has
\begin{eqnarray}
\label{eq:dopey}
b_T &=& \frac{b_{T,\C_G(T)}}{[\N_G(T):\C_G(T)]}\nonumber \\
&=& \frac{b_{T,P}}{\# \N_{C}(T)} + \frac{b'_{T,G}}{\# \N_{C}(T) \cdot [\N_P(T):T]}
\end{eqnarray}
\end{cor}

\begin{proof}
By Theorem \ref{thm:nontrivcase}, since $T$ is non-trivial, 
\begin{equation}
\label{eq:theform1}
b_T = \frac{1}{[\N_G(T):T]} \left(  \sum_{\Gamma \in S(T)} \mu([\Gamma:T]) \iota(\Gamma) \right ).
\end{equation}
where  $\iota(\Gamma) = \iota_G(\Gamma)$ is defined in Notation \ref{def:nontrivT}.
Clearly if $\Gamma \in S(T)$ then $\Gamma \subset \C_G(T)$ since $\Gamma$
is abelian and contains $T$.  Thus $S_G(T) = S_{\C_G(T)}(T)$, and the compatibility
of the lower numbering of ramification groups with passing from $G$ to
subgroups leads to the first equality in (\ref{eq:dopey}).  To prove the second 
equality, note that the $\Gamma \in S_G(T)$ which are $p$-groups
are exactly the elements of $S_P(T)$;  the other $\Gamma \in S_G(T)$
have $\iota(\Gamma) = 1$ since no higher ramification group can contain
a group which is not a $p$-group.  This leads to 
$$[\N_G(T):T] b_T =  \sum_{\Gamma \in S(T)} \mu([\Gamma:T]) \iota(\Gamma)  = [\N_P(T):T] b_{T,P} + b'_{T,G}$$
The second equality
in (\ref{eq:dopey}) follows from this and Lemma \ref{lem:blahT}.  
\end{proof}

\begin{cor}
\label{cor:keepitup}
Suppose that $0 \le b_{T,P} \in \mathbb{Z}$.  Then $0 \le b_T \in \mathbb{Z}$
if and only if all of the following are true:
\begin{enumerate}
\item[(a.)] $b'_{T,G} \equiv 0$ mod $[\N_P(T):T] \mathbb{Z}$.
\item[(b.)]  $\sum_{\Gamma \in S(T)} \mu([\Gamma:T]) \iota(\Gamma) \equiv 0$ mod $\#\N_{C}(T)\mathbb{Z}$.
\item[(c.)]  $b_{T,P} \ge \frac{-b'_{T,G}}{[\N_P(T):T]}$.
\end{enumerate}
\end{cor}

\begin{proof}    We have $b_T \in \mathbb{Z}$
if and only if $\# \N_{C}(T) b_T \in \mathbb{Z}$
and $[\N_P(T):T] b_T \in  \mathbb{Z}$, since $[\N_P(T):T]$
is a power of $p$ while $\# \N_{C}(T)$ is prime to $p$.  {From}
(\ref{eq:dopey}) we have
$$\# \N_{C}(T) b_T =  b_{T,P} + \frac{b'_{T,G}}{[\N_P(T):T]}.$$
So since we suppose $b_{T,P} \in \mathbb{Z}$, we see that
this is in $\mathbb{Z}$ if and only if condition (a) of Corollary
\ref{cor:keepitup} holds.  {From} (\ref{eq:theform1}) we have
$$[\N_P(T):T] b_T = \frac{[\N_P(T):T]}{[\N_G(T):T]} \left(  \sum_{\Gamma \in S(T)} \mu([\Gamma:T]) \iota(\Gamma) \right )$$
Since $[\N_G(T):T]/[\N_P(T):T] =  \# \N_{C}(T)$
by Lemma \ref{lem:blahT}, condition (b) of Corollary \ref{cor:keepitup} 
is equivalent to $[\N_P(T):T] b_T \in \mathbb{Z}$.  Finally,
Corollary \ref{cor:numconseq} shows condition (c)
is equivalent to $b_T   \ge 0$.
\end{proof}  

\begin{rem}
\label{rem:meaning}  The hypothesis that $0 \le b_{T,P} \in \mathbb{Z}$ 
holds if the Bertin obstruction of the restriction $\phi|_{P}$ of $\phi$ to $P$ vanishes
by Proposition \ref{prop:special}.
Corollary \ref{cor:keepitup} has
to do with the further conditions which must hold if the Bertin
obstruction of $\phi$ is to vanish.  (Recall that 
if the Bertin obstruction of $\phi$ vanishes then so does that of $\phi_{P}$ by Theorem \ref{thm:functorial}.)
In condition (a) of the Corollary, the constant $b'_{T,G}$ is a purely
group theoretic invariant which does not depend on $\phi$.  Condition (c) can be thought of as a lower bound
on the size of the wild ramification groups of $G$.  The object of
the rest of this section is to quantify the arithmetic information
contained in condition (b).
\end{rem}

\begin{lemma}
\label{lem:gammalem}There is a unique character $\chi:\N_{C}(T) \to \mathbb{Z}_p^*$
such that 
\begin{equation}
\label{eq:conjaction}
ghg^{-1} = h^{\chi(g)} \quad \mathrm{for}\quad g \in \N_{C}(T)\quad \mathrm{ and}\quad h \in T
\end{equation} 
where $h^{\chi(g)}$ is well-defined because $h$ has $p$-power order. For each $\Gamma \in S_P(T)$, one has $\N_{C}(\Gamma) \subset \N_{C}(T)$, and (\ref{eq:conjaction}) holds for $g \in \N_{C}(\Gamma) $ and
$h \in \Gamma$.  
\end{lemma}

\begin{proof}   The first statement is clear from the fact that $T$ is a cyclic $p$ group and
$\N_{C}(T) \subset C$ is cyclic of order prime to $p$.  Recall that $\Gamma \in S_P(T)$ must be a cyclic $p$-group
containing $T$.  Hence $T$ is characteristic in $\Gamma$, so
$\N_{C}(\Gamma) \subset \N_{C}(T)$.  By the existence of Teichm\"uller lifts,  there is a unique
character $\psi:\N_{C}(\Gamma) \to \mathbb{Z}_p^*$ giving the conjugation
action of $\N_{C}(\Gamma)$ on $\Gamma$.  Since the kernel of the 
restriction homomorphism $\mathrm{Aut}(\Gamma) \to \mathrm{Aut}(T)$
is a $p$-group, this $\psi$ must be the restriction of $\chi$ to $\N_{C}(\Gamma)$.
\end{proof}

The following result is Proposition 9 of \S IV.2 of \cite{corps}.
 
\begin{lemma}
\label{lem:urk}
({\rm Serre}) Let $\mathbf{p}$ be the maximal ideal $tk[[t]]$ of $k[[t]]$, and recall that $G_j$
is the $j^{th}$ ramfication subgroup of $G$ in the lower numbering.  Let $\theta_0:C = G_0/G_1 \to k^*$ be the faithful character defined by $$\theta_0(\sigma) \equiv \frac{\phi(\sigma)(t)}{t}\quad\mathrm{mod} \quad \mathbf{p}\quad \mathrm{for}\quad \sigma \in C.$$ For each $j \ge 1$ we have an
injective group homomorphism $\theta_j:G_j/G_{j+1} \to \mathbf{p}^j/\mathbf{p}^{j+1}$ defined  by
$\phi(\sigma)(t)/t \equiv 1 + \theta_j(\sigma) $ mod $\mathbf{p}^{j+1}$.  Then
\begin{equation}
\label{eq:yowie}
\theta_j(sxs^{-1}) = \theta_0(s)^j \cdot \theta_j(x)
\end{equation}
for $x \in G_j/G_{j+1}$ and $s \in C$.  
\end{lemma}

\begin{cor}
\label{cor:connectit}
Let $\theta_0:C \to k^*$ and $\theta:C \to k^*$ be the characters defined in Lemma \ref{lem:urk}
and Definition \ref{def:Teich}.  Then $\theta = \theta_0^{\# P}$ where $\# P$ is a power of $p$.
If $g \in C$ then $\theta(g) \in (\mathbb{Z}/p)^*$ if and only if $\theta_0(g) \in (\mathbb{Z}/p)^*$,
and in this case $\theta(g) = \theta_0(g)$.
\end{cor}

\begin{proof}  The equality $\theta = \theta_0^{\# P}$ is clear from the fact that if $u$ is the uniformizer in $k[[t]]^{\phi(P)}$ used in Definition \ref{def:Teich} then $u = t^{\# P} v$ for some unit $v$ in $k[[t]]$.  The 
second statement in the corollary follows from the fact that $\# P$ is a power of $p$.
\end{proof}

\begin{cor}
\label{cor:umph}
Suppose $\Gamma \in S_P(T)$.  Let $i$ be the largest integer such
that $\Gamma \subset G_i$, so that $\iota(\Gamma) = i + 1$.  Suppose
$\N_{C}(\Gamma)$ is not trivial.  The character $\theta^i_0:\N_{C}(\Gamma) \to k^*$
takes values in $(\mathbb{Z}/p)^*$ and is trivial if $p = 2$.  The resulting 
Teichm\"uller lift of this character is the restriction of $\chi:C \to \mathbb{Z}_p^*$
to $\N_{C}(\Gamma)$.
\end{cor}

\begin{proof}  Since $\Gamma$ is cyclic $p$-group and $\Gamma_i = \Gamma$ properly contains
$\Gamma_{i+1}$, we have $i \ge 1$, and the group $\Gamma_i/\Gamma_{i+1}$ is a non-trivial
cyclic $p$-group. This group must in fact be of order $p$, since
$\theta_i$ is an embedding of it into $\mathbf{p}^i/\mathbf{p}^{i+1}$.  Thus $\theta_i(\Gamma_i/\Gamma_{i+1})$
is a one dimensional $\mathbb{Z}/p$ vector space inside $\mathbf{p}^i/\mathbf{p}^{i+1}$ which by Lemma \ref{lem:urk} is stable by multiplication by the elements
of $\theta^i_0(\N_{C}(\Gamma)) \subset k^*$. This implies $\theta^i_0(\N_{C}(\Gamma)) \subset (\mathbb{Z}/p)^*$.  If $p = 2$, this shows $\theta^i_0$ restricts to the
trivial character on $\N_{C}(\Gamma)$.  In general, Lemma \ref{lem:urk} shows
that $\theta^i_0|_{\N_{C}(\Gamma)}$ gives the conjugation action of $\N_{C}(\Gamma)$
on $\Gamma$.  By Lemma \ref{lem:gammalem}, this action is also given
by the restriction of $\chi$ to $\N_{C}(\Gamma)$, so the final statement
of Corollary \ref{cor:umph} follows from the uniqueness of 
Teichm\"uller
lifts for characters of cyclic groups of order prime to $p$.  
\end{proof}

\begin{lemma}
\label{lem:abstractit}   There is a unique residue class $j_T \in \mathbb{Z}/(\# \N_{C}(T) \mathbb{Z})$ such that the restriction $\theta_0^{j_T}|_{\N_{C}(T)}$ of $\theta_0^{j_T}:C \to k^*$ to $\N_{C}(T)$ takes values in $(\mathbb{Z}/p)^*$ and has Teichm\"uller lift the character $\chi_T$ of Lemma \ref{lem:gammalem}.  Suppose that $\Gamma \in S_P(T)$, so that $\Gamma$ is a
cyclic $p$-group which contains $T$.  The group $\N_{C}(T)$ acts by conjugation
on $S_P(T)$.  Let $S_P^T(\Gamma)$ be the orbit of $\Gamma$ under this action.
One has 
\begin{equation}
\label{eq:zounds}
\sum_{\Gamma_1 \in S_P^T(\Gamma)} \mu([\Gamma_1:T]) \iota(\Gamma_1) \equiv  
\sum_{\Gamma_1 \in S_P^T(\Gamma)} \mu([\Gamma_1:T])(1 + j_T) \quad \mathrm{mod}\quad \# \N_{C}(T) \mathbb{Z}
\end{equation}
\end{lemma}

\begin{proof}  The first statement, about the existence of $j_T$, is a consequence
of the fact that $\theta_0$ is a faithful character of $C$
and $\chi$ is a character of $\N_{C}(T)$ with values in $\mathbb{Z}_p^*$.
Since $\N_{C}(T)$ conjugates $T$ to itself, it acts on $S_P(T)$,
and elements in the each orbit have the same order and value for
$\iota$.  The stabilizer of $\Gamma$ under this action is $\N_{C}(\Gamma)$,
so
\begin{equation}
\label{eq:sumitup}
\sum_{\Gamma_1 \in S_P^T(\Gamma)} \mu([\Gamma_1:T]) \iota(\Gamma_1) = 
[\N_{C}(T):\N_{C}(\Gamma)] \mu([\Gamma:T]) \iota(\Gamma)
\end{equation}
By Lemma \ref{lem:gammalem}, the action of $\N_{C}(\Gamma)$ on $\Gamma$ by conjugation  is given by the restriction of $\chi$ from $C$ to $\N_{C}(\Gamma)$.
As in Corollary \ref{cor:umph}, let $i$ be the largest integer such
that $\Gamma \subset G_i$, so that $\iota(\Gamma) = i + 1$. 
By Corollary \ref{cor:umph}, $\theta_0^i|_{\N_{C}(\Gamma)}$ has 
Teichm\"uller
lift $\chi|_{\N_{C}(\Gamma)}$.  However, $\theta_0^{j_T}|_{\N_{C}(\Gamma)}$ also has Teichm\"uller
lift $\chi|_{\N_{C}(\Gamma)}$.  Since $\theta_0$ is a faithful character of $C$,
this forces $i \equiv j_T$ mod $\# \N_{C}(\Gamma)$.  Therefore $\iota(\Gamma) \equiv 1 + j_T$ mod $\# \N_{C}(\Gamma)$.  Substituting this
into (\ref{eq:sumitup}) proves (\ref{eq:zounds}) since $[\N_{C}(T):\N_{C}(\Gamma)] \cdot \# \N_{C}(\Gamma) = \# \N_{C}(T)$.
\end{proof}
 
\begin{lemma}
\label{lem:thetrivialcase} Suppose that $\chi$ in Lemma \ref{lem:gammalem}
has a non-trivial kernel.  Then $\N_{C}(T) = \C_{C}(T) = C$ and $\chi$ is trivial.
Condition (b) of Corollary \ref{cor:keepitup}
is equivalent to
\begin{equation}
\label{eq:tada}
\sum_{\Gamma \in S(T)} \mu([\Gamma:T]) \equiv 0\quad \mathrm{mod} \quad \#\N_{C}(T)\mathbb{Z}
\end{equation}
which is independent of $\phi$.
\end{lemma}

\begin{proof} Suppose that $t \in \N_{C}(T)$ is a non-trivial element of
the kernel of $\chi$.  Then $t$ acts trivially on the cyclic $p$-group $T$
by conjugation.  The final assumption of Hypothesis \ref{hyp:setup}
now says that $\C_G(t) = \C_P(t) \times C$ is cyclic.  Thus $T$ 
is contained in $\C_G(t)$, and every element of $\C_G(t)$ commutes
with $T$.   In particular, $C$ is contained in $\C_G(T)$, so
we get that $\N_{C}(T) = \C_{C}(T) = C$ and that $\chi$ is trivial.  Hence
the residue class $j_T \in \mathbb{Z}/\# \N_{C}(T) \mathbb{Z}$ defined
in Lemma \ref{lem:abstractit} is trivial.  Summing (\ref{eq:zounds}) over
the $\N_{C}(T)$ orbits in $S_P(T)$ now gives 
\begin{equation}
\label{eq:allrighty}
\sum_{\Gamma_1 \in S_P(T)} \mu([\Gamma_1:T]) \iota(\Gamma_1) \equiv  \sum_{\Gamma_1 \in S_P(\Gamma)} \mu([\Gamma_1:T])\ \mathrm{mod} \ \#\N_{C}(T) \mathbb{Z}.
\end{equation}
If  $\Gamma \in  S(T)$ is not in $S_P(T)$ then $\iota(T) = 1$ since $\Gamma$
is not a $p$-group.  Hence summing $\mu([\Gamma:T]) \iota(\Gamma) = 
\mu([\Gamma:T])$ as $\Gamma$ runs over these groups to both
sides of (\ref{eq:allrighty}) leads to the reformulation of
condition (b) stated in Lemma \ref{lem:thetrivialcase}.
\end{proof}

\begin{lemma}
\label{lem:nontrivial}
Suppose that $\chi$ in Lemma \ref{lem:gammalem}
has trivial kernel, which is equivalent to $\C_{C}(T) = \{e\}$.  Let $D$ be the subgroup of $G$
generated by $T$ and by $\N_{C}(T)$.  Then the constant
$b_{T,D}$ equals $\iota(T)/ \# \N_{C}(T)$,
where $\iota(T) = i_T + 1$ when $i_T$
is the largest integer $i$ such that $T \subset G_{i}$.
\begin{enumerate}
\item[a.] One has $b_{T,D} \in \mathbb{Z}$ if
and only if  the residue class $j_T$ in Lemma \ref{lem:abstractit} is
$-1$ mod $\# \N_{C}(T) \mathbb{Z}$.
\item[b.]
Suppose  $b_{T,D} \in \mathbb{Z}$.  Then condition (b) of Corollary \ref{cor:keepitup}
is equivalent to the congruence 
\begin{equation}
\label{eq:urkurk}
b'_T = b'_{T,G} = \sum_{P \not \supset \Gamma \in S(T)} \mu([\Gamma:T]) \equiv 0 \quad \mathrm{mod}\quad \# \N_{C}(T) \mathbb{Z}.
\end{equation}
\end{enumerate}
\end{lemma}

\begin{proof} Since $\chi$ has trivial kernel, $D$ is the semi-direct
product of the normal $p$-group $T$ with the cyclic prime to $p$ group $\N_{C}(T)$, and the
action of $\N_{C}(T)$ on $T$ is  faithful.  If $\Gamma$ is a cyclic subgroup
of $D$ containing $T$, then $\Gamma/T \subset D/T \cong \N_{C}(T)$
acts faithfully by conjugation on $T$.  This forces
$\Gamma = T$, so the set $S_D(T)$ of such $\Gamma$ is simply
$\{T\}$.  Therefore $b_{T,D} = \iota(T)/\# \N_{C}(T)$ by Theorem \ref{thm:nontrivcase}.
Replacing $G$ by $D$ in Lemma \ref{lem:abstractit} shows $b_{T,D}$ is
integral if and only $j_T \equiv  -1$ mod $\#\N_{C}(T)$, which we will suppose is
the case for the rest of the proof.  We now return to $G$ as before, so that $G$
need not be $D$.  Summing the formula in Lemma \ref{lem:abstractit} over
the $\N_{C}(T)$ orbits in $S_P(T)$ now gives
$$\sum_{\Gamma \in S_P(T)} \mu([\Gamma_1:T]) \iota(\Gamma_1) \equiv
 0 \quad \mathrm{mod}\quad \# \N_{C}(T) \mathbb{Z}$$
 since $j_T + 1 \equiv 0 $ mod $\# \N_{C}(T) \mathbb{Z}$.  Hence condition (b)
 of Corollary \ref{cor:keepitup} becomes the congruence
 \begin{eqnarray}
 \label{eq:doitnow}
 0 &\equiv& \sum_{\Gamma \in S(T)} \mu([\Gamma:T]) \iota(\Gamma)\ \mathrm{ mod} \  \#\N_{C}(T)\mathbb{Z}\nonumber\\
 &\equiv& \sum_{P \not \supset \Gamma \in S(T)} \mu([\Gamma:T]) \iota(\Gamma)\ \mathrm{ mod} \  \#\N_{C}(T)\mathbb{Z}\nonumber\\
 &= & b'_{T,G} \ \mathrm{ mod} \  \#\N_{C}(T)\mathbb{Z}
 \end{eqnarray}
 since $\iota(\Gamma) = 1$ if $\Gamma \not \subset P$.
\end{proof}

\begin{rem}
\label{rem:interpnontriv}  In view of Lemma \ref{lem:abstractit}, the arithmetic condition in part (a) of Lemma \ref{lem:nontrivial} is that the conjugation action of $\N_{C}(T)$ on $T$ is
via the inverse of the Teichm\"uller lift of the character $\theta_0|_{\N_{C}(T)}$.
Note that $\theta_0|_{\N_{C}(T)}$ takes values in $(\mathbb{Z}/p)^*$,
so it agrees with the restriction $\theta|_{\N_{C}(T)}$ of the character
$\theta:C \to k^*$ defined in Definition \ref{def:Teich} because of 
Remark \ref{rem:clarify}.  
\end{rem}
\medbreak
\noindent {\bf Completion of the proof of Theorem \ref{thm:reducetop}}
\medbreak
We split the proof into two parts:
\medbreak
\noindent {\bf Part 1:}  {\bf Suppose the Bertin obstruction of $\phi$ vanishes.}  
\medbreak
The Bertin obstruction of $\phi_P$ then vanishes by Theorem \ref{thm:functorial},
so condition (a) of Theorem \ref{thm:reducetop} holds.  Condition (b) of 
the theorem follows from Corollary \ref{cor:notquot} and Proposition  \ref{prop:infoprime}. Condition (c)
of Theorem \ref{thm:reducetop} is a consequence of  Corollary \ref{cor:notquot}  and conditions (a) and (c)
of Corollary \ref{cor:keepitup} together with Proposition \ref{prop:special}  and Theorem \ref{thm:functorial} .   

We now 
suppose as in condition (d)
of Theorem \ref{thm:reducetop} that $T$ is a non-trivial cyclic subgroup of $P$.  Since we have
supposed the Bertin obstruction of $\phi$ vanishes, we have $b_T \ge 0$ by Proposition \ref{prop:special}(ii).
Therefore conditions (a), (b) and (c) of Corollary \ref{cor:keepitup} hold.  

Suppose first that
$\C_{C}(T)$ is non-trivial. By the definition of $\chi$ in Lemma \ref{lem:gammalem},
$\chi$ is trivial on $\C_{C}(T)$.   Hence the hypothesis of Lemma \ref{lem:thetrivialcase}  holds.
This Lemma \ref{lem:thetrivialcase} now shows that $\C_{C}(T) = C$.  This lemma also shows that since 
 (b) of Corollary \ref{cor:keepitup} holds, the
congruence claimed in condition (d)(i) of Theorem \ref{thm:reducetop} is true, since 
$b''_{T,G}$ is the constant on the left side of (\ref{eq:tada}) by Notation \ref{notn:Bprime}.  
This completes the proof of condition (d) of Theorem \ref{thm:reducetop} if $\C_{C}(T)$ is not trivial.

Suppose now that  $\C_{C}(T)$ is trivial.    The hypotheses of Lemma \ref{lem:nontrivial} now hold, and 
the character $\chi$ in this lemma has trivial kernel.
As in this lemma, let $D$ be the subgroup generated by $T$ and $\N_{C}(T)$.  Since 
we supposed that the Bertin obstruction of $\phi$ vanishes,  $b_{T,D}$ is an integer by
Corollary \ref{cor:notquot}.   The character $\theta$ in 
Definition~\ref{def:Teich} and Theorem \ref{thm:reducetop}(d)(ii) now
has the properties claimed because of Lemma \ref{lem:nontrivial}(a), Lemma \ref{lem:abstractit}  and Remark \ref{rem:interpnontriv}.  Finally,
since we have already proved that  (b) of Corollary \ref{cor:keepitup} is true under the above hypotheses, the 
remaining congruence to be proved in Theorem \ref{thm:reducetop}(d)(ii) follows from Corollary  \ref{cor:notquot} and Lemma \ref{lem:nontrivial}(b).
This completes the proof of condition (d) of Theorem \ref{thm:reducetop}.
We have now shown that if the Bertin obstruction of $\phi$ vanishes, then (a) - (d) of Theorem \ref{thm:reducetop}
hold.  
\medbreak
\noindent {\bf Part 2:}  {\bf Suppose conditions (a) - (d) of Theorem \ref{thm:reducetop}
hold. } 
\medbreak
By Proposition \ref{prop:special}(ii) it will suffice to show that $0 \le b_T \in \mathbb{Z}$
for all non-trivial cyclic subgroups $T$ of $G$, since then the Bertin obstruction of $\phi$
vanishes.

Suppose first that $T$ is not a $p$-group.  We have supposed that condition (b) of
Theorem \ref{thm:reducetop} holds. By Lemma \ref{lem:Gstruct}, all cyclic subgroups
$C_0$ of $G$ of maximal prime-to-$p$ order are conjugate to $C$.  If now follows
from condition (b) of Theorem \ref{thm:reducetop} that condition (b) of Proposition 
\ref{prop:infoprime} holds.  Therefore by letting $J = G$ in Proposition \ref{prop:infoprime}(a)
we see that 
$0 \le b_T = b_{T,J} \in \mathbb{Z}$, as required.

Now suppose that $T$ is a non-trivial cyclic $p$-subgroup of $G$.   We have supposed
in condition (a) of Theorem \ref{thm:reducetop} that the Bertin obstruction of the restriction of $\phi$ to $P$ vanishes.  Thus $ 0 \le b_{T,P} \in \mathbb{Z}$ by Notation \ref{notn:bprimedef} 
and Proposition \ref{prop:special}.  Thus to complete the proof, it will suffice to 
show that conditions (a), (b) and (c) of Corollary \ref{cor:keepitup} hold.  Hypothesis 
(c) of Theorem \ref{thm:reducetop} is that (a) and (c) of Corollary \ref{cor:keepitup}
hold, so we are reduced to checking (b) of that corollary.  

Suppose first that the centralizer $\C_{C}(T)$ of $T$ in $C$ is non-trivial.  This is equivalent
to supposing that the character $\chi$ in Lemma \ref{lem:gammalem} has a non-trivial
kernel.  Lemma \ref{lem:thetrivialcase}  and Notation \ref{notn:Bprime} now show that Hypothesis d(i) of Theorem \ref{thm:reducetop} is equivalent to condition (b) of Corollary \ref{cor:keepitup},
so we are done in this case.

Finally, suppose that $\C_{C}(T)$ is trivial.  In view of Remark \ref{rem:interpnontriv},
the hypothesis in part (d)(ii) of Theorem \ref{thm:reducetop} concerning the character $\theta$ is equivalent to condition (a) of Lemma \ref{lem:nontrivial}.  Therefore part (b) of Lemma \ref{lem:nontrivial} shows that condition (b) of Corollary \ref{cor:keepitup} is equivalent to 
the hypothesis on $b'_{T,G}$ in Theorem \ref{thm:reducetop}.  Therefore
condition (b) of Corollary \ref{cor:keepitup}  holds in all cases and the proof is complete.\ $\square$

\section{Proof of Theorem \ref{thm:antiB}.}
\label{s:therealend}
\setcounter{equation}{0}
We begin by proving that GM-groups have certain properties that we require.

\begin{thm} 
 \label{groovythm}   Let $G$ be a GM group with respect to
the character $\Theta$.  Let $P,B$ and $C$ be as in the Definition \ref{def:groovy}.  
Set $D:=\C_P(C)$.  Suppose $T$ is a non-trivial cyclic subgroup of $P$. Recall that $S_G(T)$ is the set of cyclic subgroups of $G$ which
contain $T$, and $\mu(x)$ is the Mobius $\mu$ function.  Let $b'_{T,G}$ and $b''_{T,G}$ be as in Notation \ref{notn:Bprime}.
\begin{enumerate}
\item[a.]  One has
\begin{equation}
\label{bprimeGM}
b'_{T,G} = \sum_{P \not \supset \Gamma \in S_G(T)} \mu([\Gamma:T])  \equiv 0\quad \mathrm{mod}\quad [\N_P(T):T] \mathbb{Z}.
\end{equation}
\item[b.] Suppose $\C_{C}(T)$ is not trivial.  Then $T \subset D$ and 
\begin{equation}
\label{eq:sumhome}
b''_{T,G} = \sum_{ \Gamma \in S_G(T)} \mu([\Gamma:T]) \equiv 0 \quad \mathrm{mod} 
\quad \# C\mathbb{Z}
\end{equation}
\item[c.]  Suppose $\C_{C}(T)=1$.  Then either $\N_{C}(T) = 1$
or every cyclic overgroup of $T$ is contained in $P$.
\end{enumerate}
\end{thm}

\begin{proof}   The results are obvious if $C \ne 1$.  So
assume this is not the case.   Set $H=D \times C$.
Let $1 \ne W \le C$.  Then 
$\C_G(W)=DC = H$ (since $D=\C_P(W)$ and $C$ is a complement
to any Sylow $p$-subgroup of $\C_G(W)$). 
Similarly, $\N_G(W) =\N_P(W)C=\C_P(W)C=\C_G(W)$.
 It follows
that $C \cap C' = 1$ for any conjugate $C'$ of $C$ other than $C$.

We first prove (c).
Suppose that $T$ is contained in some cyclic subgroup
 not contained in $P$.  It follows that $T$ centralizes some
element $c' \ne 1$ of order prime to $p$.  By Lemma \ref{lem:Gstruct},
  $c'$ is conjugate to
some element $c \in C$.  Thus,  $T$ centralizes a conjugate $C'$
of $C$.  It follows that $[G:\C_G(T)]$ is a power of $p$ and
so $\N_G(T)/\C_G(T)$ is a $p$-group.  Thus, $\N_{C}(T) \le \C_G(T) \cap C =1$,
and (c) follows. 

We next prove (a) and (b).   By Lemma \ref{lem:blahT} we can replace $T$ by  
a conjugate subgroup
in order to have  $\C_G(T) = \C_P(T).\C_{C}(T)$ and $\N_G(T) = \N_P(T).\N_{C}(T)$.
We may assume that $\C_{C}(T) \ne 1$  (this is the assumption in (b);
and in (a) if this is not the case, then we are summing over the empty set).
Thus,  $1 \ne T \le D$ by Definition \ref{def:groovy}(a).      So   $\N_G(T)=\N_P(T)C$. 

Since $H$ is cyclic, it is clear that $b'_{T,H} = -1$ if $T =D$
and $b'_{T,H}=0$ if $T < D$.    Also, $b''_{T.H}=0$. 

We claim that every cyclic subgroup $E$ of $G$ that contains $T$ and that is
not contained in $P$
is conjugate to a unique subgroup of $H$ containing $T$
via an element of $\N_G(T)$.  Here uniqueness is clear because $H$ is cyclic and so has
a unique subgroup of each order.  Existence follows by Lemma \ref{lem:zap} since  $E$ contains a nontrivial
$p'$-subgroup which is conjugate (in $\N_G(T)$) to a subgroup of $C$.
So we may assume that $E \cap C \ne 1$, whence $E \le H$.

Now let $E \le H$  be such a subgroup.   Then $\N_G(E)$ must
normalize every subgroup of $E$, and so $\N_G(E)=H$.
Thus,  $b'_{T,G}= [\N_G(T):H]b'_{T,H}$.  If $T < D$, this implies
that $b'_{T,G}=0$ and (a) follows. 
If $T=D$,  then $[\N_G(T):H]=[\N_P(D):D]$, and again (a) follows.

 We now prove (b).   Note that $C$ acts on $S_G(T)$.
 Let $E  \in S_G(T)$.  Then $C$ centralizes $E$ if
 $E \subset H$.  Suppose $E$ is not contained in $H$ and $1 \ne c \in C$
 normalizes $E$.   Since $T \subset D$, $c$ centralizes $T$, so $c$ centralizes the
 cyclic Sylow $p$-subgroup of $E$.  Thus, the Sylow $p$-subgroup
 of $E$ is contained in $D$.  Thus, $E$ must contain a nontrivial
$p'$-subgroup $C'$ not contained in $C$ since we've assumed $E \not \subset H$.   Then $c$ normalizes
 $C'$ and so $c$ centralizes $C'$ because it does so mod $P$.  Therefore $C' \subset H$, so $C' \subset C$, which is 
 a contradiction.  Thus, $C$ acts freely by conjugation on
 the elements of $S_G(T)$ not contained in $H$.  Since
 $\mu([\Gamma:T])$ is constant on $C$-conjugates, it follows that:
 \begin{equation}
 \label{eq:SHT}
  \sum_{ \Gamma \in S_H(T)} \mu([\Gamma:T]) \equiv b''_{T,G}  \quad \mathrm{mod} 
\quad \# C\mathbb{Z}.
\end{equation}
Since $H$ is a cyclic group which properly contains $T$, the sum on the left in (\ref{eq:SHT}) is $0$, 
which  completes the proof.
\end{proof}

\begin{lemma}
\label{lem:liftx}  Let $Q$ be a $p$-group and $x$ an automorphism of $Q$
of order dividing $p-1$.  Let $\phi:Q \rightarrow R$ be an $x$-equivariant
surjection.  If $r \in R$ with $x(r)=r^e$, then there exists $s \in Q$
with $\phi(s)=r$ and $x(s) =s^e$.
\end{lemma}

\begin{proof}  There is no loss in assuming that $R$ is generated by $r$
(replace $R$ by this subgroup and $Q$ by the inverse image).  If $\phi$
factors through an intermediate group, the result follows by induction.
So we may assume that $K:=\mathrm{ker}(\phi)$ is a minimal normal $x$-invariant
subgroup of $Q$.  Thus, we may assume that $K$ is a central elementary
abelian $p$-subgroup of $Q$ with $x$ irreducible on $Q$.  Since $x$
has order dividing $p-1$, this forces $K$ to have order $p$. So either
$Q$ is cyclic of order $p^2$, in which case the result is clear, or else 
$Q$ is elementary abelian.  In that case, $Q$ is a completely reducible
$x$-module and so the result is also clear.
\end{proof}

\begin{lemma} \label{quotients}
If $G$ is a GM-group with respect to a character $\Theta$, then so is every subquotient.
\end{lemma}

\begin{proof}  Keep the usual notation.  We induct on the order
of $G$.   It is clear that subgroups of GM groups are GM groups.
So it suffices to consider quotients by minimal normal non-trivial
subgroups $N$. Since $P$ is normal in $G$, either $N \cap P$ is
trivial and $N$ is of order prime to $p$ or $N \subset P$.  If $N$
has order prime to $P$, then $N.P$ must be the product group $N \times P$,
so $N$ and $P$ commute.  Hence $N$ is a subgroup of the cyclic group $C$
since all subgroups of $G$ of order prime to $p$ are conjugate by an element of 
$P$ to a subgroup of $C$.  Definition \ref{def:groovy}(a) now implies $P$
must be cyclic by choosing $c$ to be a non-trivial element of $N$, and $C$
must commute with $P$.  Hence $G$ is cyclic, and it follows that $G/N$
is GM.

Suppose now that $N$ is a minimal normal subgroup of $G$
contained in $P$.  On taking the intersection of $N$ with the lower
central series of $P$, we see that there is a non-trivial subgroup $\N_0$
of $N$ which is normal in $G$ such that the commutator group $[P,\N_0]$ is trivial. 
Since $N$ is a minimal normal subgroup of $G$, this implies $N = \N_0 \subset \C(P)$.
Because $C$ has order prime to $p$,
we see that if $1 \ne c \in C$, then Lemma \ref{lem:liftx} implies that 
$\C_{P/N}(c)=\C_P(c)N/N=\C_P(C)N/N=\C_{P/N}(C)$ is cyclic. 
So condition (a) of Definition~\ref{def:groovy} holds in $G/N$.  In a similar way,
Lemma \ref{lem:liftx} implies that condition (b) of 
Definition~\ref{def:groovy} holds for $G/N$ because it holds for $G$.
\end{proof}

\medbreak
\noindent {\bf Completion of the proof of Theorem \ref{thm:antiB}}
\medbreak

Suppose first that there is an injection $\phi_G: \to \mathrm{Aut}_k(k[[t]]$
having vanishing Bertin obstruction.  Let $\Theta_C$ be the inverse
of the Teichm\"uller lift of the character $\theta:C \to k^*$ appearing in Theorem \ref{thm:reducetop}.
Parts (b) and (d) of Theorem \ref{thm:reducetop} then show that $G$ is a GM group for $k$ 
with respect to the restriction $\Theta$ of $\Theta_C$ to the maximal subgroup $B$ of order dividing $p-1$.  

Suppose now that $G$ is  GM for $k$ with respect to $\Theta$.  Pick a faithful extension
$\Theta_C:C \to W(k)^*$ of $\Theta$ from $B$ to $C$.  
Let $M$ be a  positive integer.  By induction on the length of a composition series for $G$, we
can use Lemma \ref{def:dumdumdum}
 and Proposition \ref{prop:subquots} of Appendix 1 to construct
an injection $\phi_G: G \to \mathrm{Aut}_k(k[[z]])$ which is GM
with respect to $\Theta_C$ and such that  
\begin{equation}
\label{eq:iotanice}
\iota(T) \ge \iota(\Gamma) + M\quad \mathrm{and}\quad 
\iota(T) \equiv 0\ \mathrm{mod} \ p^M
\end{equation}
if $T$ is a non-trivial proper subgroup of the cyclic $p$-subgroup $\Gamma$
of $G$.  We will show that if $M$ is chosen to be sufficiently large, then
$\phi_G$ will satisfy all the conditions of Theorem \ref{thm:reducetop}.
This theorem will then imply that $\phi_G$ has vanishing Bertin obstruction,
and this will complete the proof of Theorem \ref{thm:antiB}.

Consider first condition (a) of Theorem  \ref{thm:reducetop}.  
By Theorem \ref{thm:nontrivcase},
\begin{eqnarray}
\label{eq:lower}
b_{T,P} &=& \frac{1}{[\N_P(T):T]} \sum_{\Gamma \in S_P(T)} \mu([\Gamma:T]) \iota(\Gamma)\nonumber\\
&=& \frac{1}{[\N_P(T):T]} \left ( \iota(T) + \sum_{T \ne \Gamma \in S_P(T)} \mu([\Gamma:T]) \iota(\Gamma)\right )
\end{eqnarray}
where $S_P(T)$ is the set of cyclic subgroups $\Gamma \subset P$
which contain $T$.  So by making $M$  sufficiently large, (\ref{eq:iotanice})
will 
insure that each such $b_{T,P}$ will be larger than any specified
integer and will be integral.  Thus the Bertin obstruction of the restriction of $\phi$ to $P$
vanishes by Proposition \ref{prop:special}(ii)(b), so hypothesis (a) of Theorem \ref{thm:reducetop}
holds.  We see also from this that since the constant $b'_{T,G}$ in Notation \ref{notn:Bprime} depends only on $G$, we can insure that the inequality in condition (c)(ii) of 
Theorem \ref{thm:reducetop} holds by making $M$ sufficiently large.

It remains to check conditions (b), (c)(i) and (d) of Theorem \ref{thm:reducetop}.

Concerning condition (b), let $t$ be a non-trivial element of $C$.  Then $C$ commutes
with $t$; so since $G = P.C$ we conclude that $\C_G(t) = \C_P(t).C$.  However,
$\C_P(c) = \C_P(C)$ is a cyclic $p$-group since $G$ is a GM group (see 
Definition \ref{def:groovy}).  Hence $\C_G(t) = \C_P(C) \times C$  is cyclic 
and condition (b) of Theorem \ref{thm:reducetop} holds.

Suppose now that $T$ is a non-trivial cyclic subgroup of $P$ as in conditions
(c) and (d) of Theorem \ref{thm:reducetop}.  Condition (c)(i) of this theorem holds
by part (a) of Theorem \ref{groovythm}.  If $\C_{C}(T)$ is not trivial, condition d(i) 
of Theorem \ref{thm:reducetop} holds by part (b) of Theorem \ref{groovythm}.  
Suppose now that $\C_{C}(T)$ is trivial.  The statements about $\theta$ in  condition d(ii) of Theorem \ref{thm:reducetop}
hold because we constructed  $\phi_G: G \to \mathrm{Aut}_k(k[[z]])$ to be GM
with respect to $\Theta$ in the sense of Proposition \ref{prop:subquots} of Appendix 1.
It remains to prove the congruence
$$b'_{T,G} = \sum_{P \not \supset \Gamma \in S(T)} \mu([\Gamma:T])  \equiv 0 \quad \mathrm{mod}\quad \# \N_{C}(T) \mathbb{Z}$$
required in part d(ii) of Theorem \ref{thm:reducetop}.  Part (c) of Theorem \ref{groovythm} shows
that either $\# \N_{C}(T) = 1$ or the sum defining $b'_{T,G}$ is empty; so this congruence holds
and the proof is complete.

\section{Examples and characterizations of GM groups.}
\label{s:exGM}
We begin with some examples.

\begin{thm}
\label{thm:GMexamples}  Let $G$ be  
the semi-direct product of a normal $p$-group $P$
by cyclic subgroup $C$ of order prime to $p$.  Let $B$ be the maximal subgroup
of $C$ of order dividing $p-1$.
\begin{enumerate}
\item[a.]  If $G$ is cyclic or a $p$-group then $G$ is a GM-group.
\item[b.] If $\#B \le 2$, and $C$ acts freely on the nontrivial
elements of $P$, then $G$ is a GM-group.
\item[c.]  $G$ is not a $GM$ group if has any of the following properties:
\begin{enumerate}
\item[i.] {\rm (Green-Matignon)} $G$ contains an abelian subgroup that is neither
cyclic nor a $p$-group;
\item[ii.]$P$ is elementary abelian of order
$p^2$, $C$ has order dividing $p-1$ and $C$ acts with two distinct
nontrivial eigenvalues on $P$.  
\item[iii.] $P$ is cyclic of order $p$, and  $C$ neither acts faithfully or trivially on $P$.
\item[iv.] $P$ is extraspecial of order $p^3$ and exponent $p$,
$C$ is cyclic, $\#C$ does not divide $p+1$, and $\C_P(C) = \C(P)$.
\end{enumerate}
\end{enumerate}
\end{thm}

\begin{proof} Parts (a) and (b) follow directly from  \ref{def:groovy}.  If conditions (i) (resp. (ii), resp. (iii))
of part (c) hold, then condition (a) (resp. (b), resp. (a)) of Definition \ref{def:groovy} 
does not hold.  Suppose now that condition (iv) of part (c) holds but that $G$ is a GM group.   

Let us first show that $C$ must act faithfully on $P/\C(P)$.  Suppose to the contrary that
$c \in C$ is non-trivial and acts trivially on $P/\C(P)$.
Because $c$ commutes with $\C_P(C) = \C(P)$ and has  order prime to $p$,
$c$ must act trivially on $P$.  Then $P = \C_P(c) = \C_P(C)$ by part (a) of Definition \ref{def:groovy},
which contradicts the assumption that $\C_P(C) = \C(P)$ in part (iv).  Therefore $C$ must act faithfully on $P/\C(P)$.  

We have
$\mathrm{dim}_{\mathbb{Z}/p} P/\C(P) = 2$, and $\C(P) \cong \wedge^2(P/\C(P))$ as a $C$-module.
We have assumed in (iv) of part (c) that the action of $C$ on $\C(P)$ is trivial, so
the determinant of the action of $C$ on $P/\C(P)$ is trivial.  Let $c_0$ be a generator of $C$.  The characteristic polynomial
of the action of $c_0$  on the two-dimensional  $\mathbb{Z}/p$-vector space $P/\C(P)$
thus has the form $X^2 - aX + 1$ for some $a \in \mathbb{Z}/p$.  If this polynomial does not
split over $\mathbb{Z}/p$, its roots have multiplicative order dividing $p+1$.  Since the action of $C$ on 
$P/\C(P)$ is semi-simple, this would force the order of $C$ to divide  $p+1$, contradicting one of the assumptions in (iv).  
Therefore $X^2 - aX + 1$ splits over $\mathbb{Z}/p$. We conclude that 
as a representation of $C$ over $\mathbb{Z}/p$, 
$P/\C(P)$ must be isomorphic to the direct sum of two characters $\phi_1$ and $\phi_2$
over $\mathbb{Z}/p$.
Thus
$\# C$ divides $p-1$.  Now  Lemma \ref{quotients} and part (ii) imply that either $\phi_2 = \phi_1$ or
we can order $\phi_1$ and $\phi_2$ so that 
$\phi_2$ is trivial.    
The action of
$C$ on $\C(P) \cong \wedge^2(P/\C(P))$ is given by the character $\phi_1 \cdot \phi_2$, and we
assumed this action is trivial in part (iv).  Thus $\phi_1 = \phi_2^{-1}$.
If $\phi_1 = \phi_2$ then $\phi_1$ and $C$ have order $2$. However,
we assumed that $\# C$ does not divide $p+1$, so $\# C = 2$ would
force $p =2$, which is impossible since $\# C$ is prime to $p$.  
Thus $\phi_1 = \phi_2^{-1}$ and $\phi_2$ are distinct 
characters of $C$.  Part  (ii) now shows that $G/\C(P)$ is not a  GM group,
so $G$ is not a GM group by Lemma \ref{quotients}.
\end{proof}

In fact, we now show that GM groups can be characterized
as those groups of the form $PC$ which do not contain subgroups
of the form in Theorem \ref{thm:GMexamples}(c).

\begin{thm} \label{GM by subs}
Let $G=PC$ be a group with $P$ the normal Sylow $p$-subgroup
of $G$ with $C$ cyclic of order prime to $p$.  Then $G$ is
a GM group if and only if it has no subgroup of the following types:
\begin{enumerate}
\item $\mathbb{Z}/p \times \mathbb{Z}/p \times \mathbb{Z}/r$, with
$r$ a prime distinct from $p$;
\item $QE$ where $Q$ is of order $p$, $E$ is cyclic of order prime
to $p$ and $E$ acts neither faithfully nor trivially on $Q$; 
\item  $QE$ where $Q$ is elementary abelian of order $p^2$,
$E$ is cyclic of order dividing $p-1$, $\C_E(Q)=1$ and $E$ does not
act like a scalar on $Q$; or
\item $QE$ where $Q$ is extraspecial of exponent $p$ and order 
$p^3$, $E$ is cyclic of   order $e$  with $e$ not dividing 
$p+1$ and $\C_Q(E)=\C(Q)$.
\end{enumerate}
\end{thm}
We require the following lemmas, the first of which is an exercise beginning
with the definition $[x,y]=x^{-1}y^{-1}xy$. 

\begin{lemma} \label{commutator}  If $H$ is a group 
and $z=[x,y]$ commutes with $x$ for some $x, y, z \in H$, then $z^e=[x^e,y]$.
If $z$ commutes with both $x$ and $y$, then
$[x^e,y^f]=z^{ef}$.
\end{lemma}

\begin{lemma}  \label{filtration}
Let $Q$ be a $p$-group with $B = \langle b \rangle$ 
a  group of order
dividing $p-1$ acting on $Q$.  There is a filtration:
$$
1 = Q_0 < Q_1 < \cdots < Q_m = Q,
$$
such that:
\begin{enumerate}
\item[a.] each $Q_i$ is normal in $Q$ and $B$-invariant; 
\item[b.] each quotient $Q_i/Q_{i-1}$ is cyclic of order $p$;
\item[c.] There is a unique root of unity $e_i$ of order dividing $p-1$
in $\mathbb{Z}_p$ such that there is an element $x_i \in Q_i$ 
for which $x_i Q_{i-1}$ generates $Q_i/Q_{i-1}$ and  $bx_ib^{-1}=x_i^{e_i}$,
where $x_i^{e_i}$ is well defined because $x_i$ has non-trivial 
$p$-power order.  
\end{enumerate}
\end{lemma}

\begin{proof}  The representation of $B$ on the subgroup of $\C(Q)$ of elements of order $1$ or $p$ splits as a product of one-dimensional representations because the order of $B$ divides $p-1$.  So there is a central subgroup
$Q_1$ of $Q$ of order $p$, invariant under
$B$.  By induction, the result holds for $Q/Q_1$.  Now apply Lemma \ref{lem:liftx}.
\end{proof}

Note that one can modify the proof so that the filtration will pass
through any given normal subgroup of $Q$ that is $B$-invariant.

\medskip

\noindent {\bf Proof of Theorem \ref{GM by subs}:}  Having no
subgroup of the form (1) or (2) is equivalent to the condition 
that if $1 \ne c \in C$, then $\C_P(C)=\C_P(c)$ is cyclic.  
This is the condition (a) in the definition of GM groups (see Definition \ref{def:groovy}).

So it suffices to show that if $G$ satisfies condition (a) of Definition \ref{def:groovy}, then it is a 
GM group if and only if it does not contain a subgroup as in (3)  or (4).  
By condition (b) of Definition \ref{def:groovy}, a GM group
has no subgroups as in (3) or (4).  (For (4), see Theorem~\ref{thm:GMexamples}.c(iv) and Lemma \ref{quotients}.)
Thus it remains to show that if $G$ is not
a GM group, it contains such a subgroup.  

So assume that $G$ is not
a GM group but satisfies condition (a) of Definition \ref{def:groovy}.  Therefore condition (b) of Definition \ref{def:groovy} does not hold.  By passing to  counterexample of minimal order, 
we may assume that  $C$ has order at least $3$ and dividing $p-1$.  We will
use in what follows that fact that since $G$ is a minimal order counterexample, every subquotient
of $G$ which is not $G$ itself must be a GM group because of Lemma \ref{quotients}.
In particular condition (b) of Definition \ref{def:groovy} holds for all proper subquotients of $G$
but not for $G$ itself.

Let $Q$ be a $C$-stable subquotient of $P$, which may equal $P$.  
Let $B = C$ in Lemma \ref{filtration}, and let $1 = Q_0 < Q_1 < \cdots < Q_m = Q
$,  the $x_i$
and the $e_i$ be as in this lemma. If there are indices $i \ne j$ such that $e_i$,
$e_j$ and $1$ are distinct, then $C$ acts on $x_i$ and $x_j$ via distinct
non-trivial characters, so that $C.Q$ cannot be a GM-group because
this violates condition (b) of Definition \ref{def:groovy}.  Thus if $Q$ is
a proper subquotient of $P$, there is at most one $e_i$ different from $1$;
we let $e(Q)$ be this $e_i$ if it exists, and we let $e(Q) = 1$ otherwise. 
We claim:
\begin{equation}
\label{eq:ohbaby}
\mathrm{If}\quad Q = P\quad \mathrm{then \ at \ least \ two \ distinct}\quad e_i\quad \mathrm{are \ different \ from} \quad 1.
\end{equation}
 Suppose to the contrary that $Q = P$ and 
that there is at most one $e_i$ which is different
from $1$.  It will suffice to show that condition (b) of Definition
\ref{def:groovy} holds, since this will be a contradiction.  If all the $e_i$ equal $1$, then $C$ commutes
with all the $x_i$ and thus with $Q = P$, so $\C_P(C) = P$ and condition
(b) holds automatically.  If all the $e_i$ which are different from $1$ are equal and
there is at least one such $e_i$, we 
define $\Theta: B \to \mathbb{Z}_p^*$
by $\Theta(b) = b^{e_i}$ for any such $e_i$, where we have set $B = C$.  If $T$ is  cyclic subgroup of $P$ such
that $\C_{C}(T)$ is trivial, then by considering the smallest $i$ such that $T \subset Q_i$
and the image of $T$ in $Q_i/Q_{i-1}$ we see that condition (b) of Definition \ref{def:groovy}
holds for $T$.  This contradiction proves  the claim (\ref{eq:ohbaby}).

Let $P^p [P,P]$ be the Frattini subgroup of $P$, and let $P(p) = P/(P^p [P,P])$ be the $p$-Frattini
quotient.  If $P(p)$ is cyclic, then $P$ is cyclic, and the action of $C$ on $P$ must be through
a single character, contrary to the fact that we have shown there must be at least two distinct
$e_i$ different from $1$ when $Q = P$.  Thus $P(p)$ is a $\mathbb{Z}/p$-vector space
of dimension at least $2$, and the action of $C = B$ on $P(p)$ can be diagonalized
over $\mathbb{Z}/p$ since $\# C$ divides $p-1$. By pulling back two $C$-eigenspaces, we
conclude that there are $C$-stable normal subgroups $P_1$ and $P_2$ in $P$
such that $P/(P_1 \cap P_2)$ is elementary abelian of order $p^2$ and isomorphic
as a $C$-module to a sum of two characters $\phi_1$ and $\phi_2$ of $P$. 

 If
$P_1 \cap P_2$ is trivial, so that $P = P/(P_1 \cap P_2)$, we have seen that
$\phi_1$ and $\phi_2$ must be distinct and non-trivial (since there are at least two
distinct $e_i$ which are different from $1$).  In this case $G = PC$ satisfies
the conditions in part (3) of Theorem \ref{GM by subs}.  

Suppose now that
$P_1 \cap P_2$ is non-trivial.  If $C$ acts non-trivially on $P_1 \cap P_2$, then
we conclude that $e(P_1) = e(P_2) = e(P_1 \cap P_2)$ in the above notation,
since $P_1$, $P_2$ and $P_1 \cap P_2$ are proper subquotients of $P$.
We have a $C$-isomorphism $P_1/(P_1 \cap P_2) \to P/P_2$.  Since
$P_1 \cap P_2$ contains $P^p [P,P]$ we can now find a filtration
$$1 = Q_0 < Q_1 < \cdots < Q_{m-2} = P_1 \cap P_2 < Q_{m-1} = P_2 < Q_m = P$$
as in Lemma \ref{filtration}  such that at most one non-trivial character of $C$
arises from a $C$-module quotient $Q_i/Q_{i-1}$.  This contradicts that at least
two distinct $e_i$ different from $1$ must arise when we set $Q = P$.
We conclude that $C$ acts trivially on $P_1 \cap P_2$, so 
$P_1 \cap P_2 \subset \C_P(C)$, where $\C_P(C)$
is cyclic because we have assumed condition (a) of Definition
\ref{def:groovy} holds.  Thus $P/P_1$ and $P/P_2$ must define
distinct non-trivial characters $\phi_1$ and $\phi_2$ of $C$
in order for there to be two distinct non-trivial $e_1$ and $e_2$ associated
to setting $Q = P$.  We now use Lemma \ref{filtration} to find 
$x_i \in P_i$ such that $b x_i b^{-1} = x_i ^{e_i}$  if $b$ is a generator
of $B = C$ and for which $x_i(P_1 \cap P_2)$ generates the cyclic
group $P_i/(P_1 \cap P_2)$ of order $p$.  Then $x_i^p \in P_1 \cap P_2$
so $C$ acts trivially on $x_i^p$. Thus
$$x_i ^{p e_i} = (b x_i b^{-1})^p = b x_i^p b^{-1} = x_i^p$$
so $x_i^p = 1$ because $e_i$ is a non-trivial $(p-1)^{st}$ root of $1$ in $\mathbb{Z}_p$.
Let us check that $x_i$ must centralize each $\gamma \in P_1 \cap P_2$.
Since $x_i \gamma x_i^{-1} \in P_1 \cap P_2 \subset \C_P(C)$ we have 
$$x_i^{e_i} \gamma x_i^{-e_i} = (bx_i b^{-1})(b \gamma b^{-1}) (b x_i b^{-1})^{-1} = b(x_i \gamma x_i^{-1}) b^{-1} = x_i \gamma x_i^{-1}$$
so $x_i^{e_i - 1}$ centralizes $\gamma$, from which it follows that $x_i$ centralizes $\gamma$.
This implies that $P_1 \cap P_2$ is central in $G$, since we have shown that $C$ commutes with $P_1\cap P_2$
and because $G$ is generated by $C$,
$P_1 \cap P_2$, $x_1$ and $x_2$.
Thus $z = [x_1,x_2] \in P_1 \cap P_2$ is central in $G$, so Lemma 
\ref{commutator}  shows $z^p = [x_1^p,x_2] = 1$ because $x_1^p = 1$.
The group generated by $C$, $x_1$, $x_2$ and $z$ is now a subgroup of $G$
of the kind in part (4) of Theorem \ref{GM by subs}, which completes the proof.
$\square$

\medskip

If the subgroup $B$ (in the notation above) has order
bigger than $2$, the structure of GM groups is quite limited, as we show in the next theorem.

\begin{thm}  Let $G=PC$ be a GM-group with $\#B > 2$, where $B$ is the maximal subgroup of order dividing $p-1$.  Let $D=\C_P(C)$.
Then the derived subgroup $H = [G,G]$ of $G$ is abelian,  
$C$ acts freely on the nonidentity elements of $H$, $G$ is the
semi-direct product 
$H.(C \times D)$ and $B$ acts as a scalar
on $H$.   Conversely, any such group is a GM group. 
\end{thm}

\begin{proof}  We prove the first statement. 
By coprime action, we
have that  $P = D[C,P]$ and $[C,H]=[C,P]$ (see \cite[5.3.5]{gor}).
 Since $D$ normalizes both $C$ and $P$ and since $G=[C,P]DC$, 
it follows that $[C,P]$ is normal in $G$.  Clearly, $[C,P]$
is contained in the derived subgroup of $G$. On the other
hand $C$ and $P$ commute modulo $[C,P]$, whence $G/[C,P]$ is abelian.
Thus $[C,P]=H$.

We next claim that $H$ is abelian.  Suppose not.  Then
we can pass to the quotient $H/[H,[H,H]]$ and so assume
that $H$ is nilpotent of class $2$.  Let $T = H/\Phi(H)$, where $\Phi(H)$ is the Frattini subgroup of $X$.
Let $b$ be a generator for $B$.  Then $b$ is diagonalizable on
$T$ and we may choose a basis $y_1, \ldots, y_r$ of $T$
with each $y_i$ an eigenvector for $b$.  We may lift
$y_i$ to an element $x_i \in H$ with $b$ normalizing each
$\langle x_i \rangle$.  Since $H = [C,H]$, $T=[C,T]$ and
so $b$ centralizes none of the $y_i$.  Since $G$ is a GM-group,
this implies that $bx_ib^{-1}=x_i^{\Theta(b)}$ for each $i$.
Thus, by Lemma \ref{commutator},
 $b[x_i,x_j]b^{-1} = [x_i,x_j]^{\Theta(b)^2}$for each
$i,j$.   Since the order of $b$ is greater than $2$ and $\Theta$
is faithful, $\Theta(b)^2 \ne \Theta(b)$ and $\Theta(b) \ne 1$.
By definition, this implies that $[x_i,x_j] =1$, whence $H$ is abelian.

Again by coprime action on abelian groups, we have (see \cite[5.2.3]{gor})
that $H = \C_H(C) \times [C,H]$ and so $\C_H(C)=1$, using $H = [C,H]$.  Thus
$C$ acts fixed point freely on the nontrivial elements of
$H$.  We have already noted that $G=HDC=H.(C \times D)$.

Now assume that $G$ is as described. If $1 \ne c \in C$, 
then $\C_P(c)=D=\C_P(C)$ is cyclic.  Moreover, we see that
$c$ normalizes a cyclic subgroup if and only if it centralizes
it or acts via the character given by its action on $H$.  
Thus $G$ is a GM-group. 
\end{proof} 

\begin{example}
\label{ex:cycpsub} Suppose $P$ is cyclic.   Since $C$ is
cyclic of order prime to $p$, it acts faithfully on $P$ if
and only if it acts faithfully on the cyclic subgroup $Q$
of order $p$ in $P$.  We conclude from  Theorem
\ref{GM by subs} that when $P$ is cyclic, $G$ is a
 GM group if and only if it is cyclic or $C$ acts
 faithfully on $Q$;  the latter condition is equivalent
 to the statement that the center of $G$ is trivial.  
 If $P = Q$ has order $p$, it follows from \cite{OSS}
 (for cyclic $G$) and from  
 \cite[Theorem 2.1]{BWZ} (for non-cyclic $G$) that 
 if $G$ is a GM group then it is in fact a 
 a weak local Oort group.   
 \end{example}

 \section{Reducing the proofs of Theorems \ref{thm:nonexs} and 
\ref{thm:nonexsalmost} to particular groups.}
\label{s:reduction}
\setcounter{equation}{0}
 In this section we recall Propositions 3.1 and 4.2 of \cite{CGH}, which limit the
 possible cyclic by $p$-groups which has no quotients of certain kinds.
This will be used to limit the possible  isomorphism classes of Bertin and KGB groups.
 
\begin{thm} 
\label{thm:oddcasered} Let $p$ be an odd prime and let $G$ be a finite group 
with a normal Sylow $p$-subgroup $S$ such that $G/S=C$ is cyclic of order prime to $p$.   
Assume that $G$ has no homomorphic image of the following types:
\begin{enumerate}
\item   $C_p \times C_p$;
\item $E.C_m$, where $E$ is an elementary abelian $p$ group, $p {\not |} m \ge 3$, and 
$C_m$ acts faithfully and irreducibly on $E$;
\item $E.C_2$ where $E = C_p \times C_p$,  and $C_2$ acts on $E$
by inversion;
\item $D_{2p} \times C_\ell$ for some prime number $\ell  > 2$ 
 (including the possibility that $\ell = p$).
\item $E.C_4$ where $E = C_p$, and a generator of $C_4$ acts on $E$ by inversion.
\end{enumerate}
Then either $G$ is cyclic or is dihedral of order $2p^a$ for some 
$a$.
\end{thm}

We recall some notation and facts about $2$-groups.
A generalized quaternion group of order $2^a, a \ge 3$ is given by
$Q_a = \langle x, y | x^{2^{a-1}}=1, yxy^{-1}=x^{-1}, y^2=x^{2^{a-2}} \rangle$.
These are the only noncyclic $2$-groups that contain a unique involution.

The semidihedral group of order $2^a, a > 3$ is denoted $\mathrm{SD}_a$
and has presentation 
$\langle x, y | x^{2^{a-1}}=1, y^2=1, yxy= x^{-1 + 2^{a-2}} \rangle$.
Note that if $G$ is dihedral, semidihedral or generalized quaternion
then $G/[G,G]$ is elementary abelian of order $4$.

\begin{thm}
\label{thm:evenclass}  Let $G$ be a finite group with a normal Sylow 
$2$-subgroup  
$S$ such that
$G/S=C$ is cyclic (of odd order).   Assume that $G$ has no
homomorphic image of the following types:
\begin{enumerate}
\item $E.D$, where $E$ is a non-trivial elementary abelian $2$-group, $D$ is cyclic of odd order at
least $5$ and
$D$ acts irreducibly on $E$;
\item $E.D$, where $E$ is elementary abelian of order $16$, $D$ has 
order $3$ and acts
without fixed points on $E$;
\item $E.D$ where $E = \mathbb{Z}/4 \times \mathbb{Z}/4$, $D$ has order 
$3$ and acts
faithfully on $E$;
\item $E.D$, where $E$ is elementary abelian of order $8$ and $D$ acts 
faithfully on $E$
with $D$ of order $1$ or $3$
(note this is isomorphic to $A_4 \times \mathbb{Z}/2$ or $E$);
\item $E \times C$ where $E$ is elementary abelian of order $4$ and $C$ has 
prime order;
\item  $E.C$ where $C$ is cyclic of order $3p$ with $p$ an odd prime and $E$ 
is elementary
abelian of order $4$ with $C$ acting nontrivially on $E$; or
\item $\mathbb{Z}/4 \times \mathbb{Z}/2$.
\end{enumerate}
Then $G$ is cyclic, $A_4$, or $SL_2(3)$, or $G=S$ is a dihedral, semidihedral
or generalized quaternion $2$-group.
\end{thm}

\section{Some groups which are not almost Bertin groups.}
\label{s:notalmostBertin}
\setcounter{equation}{0}

We assume as before that $k$ is an algebraically 
closed field of characteristic $p > 0$.
\begin{prop}
\label{prop:thepoint}   Let $G$ be the semidirect product of an elementary abelian $p$-group $E$
of order $q  > 1$ with a cyclic group $C_m$ of order $m$ prime to $p$. 
\begin{enumerate}
\item[a.] Suppose $q > p$.   
Then $G$ is not an almost Bertin group for $k$ 
unless $p = 2$ and $q = 4$. If $(p,q) = (2,4)$, then $G$ is not a weak Bertin group for $k$  and not an
almost Bertin group for $k$  unless $m \in \{1,3\}$ and $C_m$ acts faithfully on $E$, in which case $G$ is isomorphic to either
$C_2^2$ or $A_4$.  
\item[b.]  Suppose $(p,q) = (2,4)$ and that $G$ is isomorphic to $A_4$.  Then for each 
integer $M \ge 0$, there is an integer $j \ge M-1$ such that $j \equiv 1$ mod $4$
and there is an injection $\phi:G \to \mathrm{Aut}_k(k[[t]])$ with  the property that $G_1 = E = G_j \ne G_{j+1}$.
\item[c.]  Suppose $q = p$, $m \ge 3$ and that $C_m$ acts faithfully on $E$.  Then
$G$ is not an almost Bertin group for $k$.
\end{enumerate}
\end{prop}

\begin{proof}Let $T$ be a subgroup of order $p$ in $E$, and recall that $S(T)$ is the set of cyclic subgroups of $G$ containing $T$.   Since $E$ is the unique $p$-Sylow subgroup of $G$, and $E$ is elementary abelian,
each $\Gamma \in S(T)$  which is different from $T$ has order divisible by some
prime different from $p$.  Thus $\iota(\Gamma) = 1$ for such $\Gamma$.  Therefore Theorem \ref{thm:nontrivcase} gives
\begin{equation}
\label{eq:killit2}
b_T =  \frac{1}{[\N_G(T):T]}  \sum_{ \Gamma \in S(T)} \mu([\Gamma:T]) \iota(\Gamma) = \frac{\iota(T) + c(T)}{[\N_G(T):T]} 
\end{equation}
where 
\begin{equation}
\label{eq:cdef}
c(T) = \sum_{ T \ne \Gamma \in S(T)} \mu([\Gamma:T]) 
\end{equation}
is independent of the ramification filtration of $G$.  

We now suppose $q > p$, so that $p^2|q$.  
Here $E \subset \N_G(T)$ since $T \subset E$ and $E$ is elementary abelian
of order $q$.  Thus $q/p$ is a positive power of $p$ dividing $[\N_G(T):T]$.
It follows that to show $G$ is not a  Bertin group for $k$, it will suffice to show that for each
integer $M \ge 1$, there is an embedding $\phi:G \to \mathrm{Aut}_k(k[[t]])$
such that $-a_\phi(\tau) \ge M$ for all non-trivial elements $\tau \in G$
of $p$-power order, and such that $\iota(T) \not \equiv -c(T)$ mod $q/p$ for 
some subgroup $T$ of order $p$ in $E$.  

The condition on $-a_\phi(\tau)$ is equivalent to requiring that if $j$ is the first jump
in the wild ramification filtration of $G$, so that $G_1 = G_j \ne G_{j+1}$,
then $j \ge M-1$.  Suppose that in addition we arrange that $T$ is
not contained in $G_{j+1}$.  Then $\iota(T) = j+1$, so we will be done
if we can also arrange that $j + 1 \not \equiv -c(T)$ mod $q/p$.
We may  assume that $k$ is the algebraic closure of $\mathbb{Z}/p$, since
if we can construct an extension of the required kind in this case we can
simply take its base change to an arbitrary algebraically closed field
of characteristic $p$.  

To construct a $\phi$ of the required kind, choose a power $q'$ of
$p$ such that $\mathbb{F}_{q'}$ contains a primitive $m^{th}$ root of unity, and let
$L= \mathbb{F}_{q'}((y))$ for an indeterminate $y$.  Letting $z = y^{1/m}$ we see
that $N = L(z)$ is a cyclic totally and tamely ramified extension of $L$, and the integral
closure of $O_L = \mathbb{F}_{q'}[[y]]$ in $N$ is $O_N = \mathbb{F}_{q'}[[z]]$.  
We fix an identification of $H = \mathrm{Gal}(N/L)$ with $C_m$.  

The group
ring $(\mathbb{Z}/p)[C_m]$ is semi-simple and acts on $N^*/(N^*)^p$.
For each integer $i \ge 1$ the natural map
\begin{equation}
\label{eq:wi}
W_i = \frac{1 + z^i O_N}{1 + z^{i+1}O_N} \to \frac{N^*}{(N^*)^p(1 + z^{i+1}O_N)}
\end{equation}
is injective if $i \not \equiv 0$ mod $p$ and is the trivial homomorphism
otherwise.  The group $H = \mathrm{Gal}(N/L) \cong C_m$
acts on the one-dimensional $\mathbb{F}_{q'}$ vector space $\mathbb{F}_{q'}\cdot z$ via a faithful character $\chi:H \to \mathbb{F}_{q'}^*$.  Thus $H$
acts on the one-dimensional $\mathbb{F}_{q'}$ vector space 
$$W_i = \frac{1 + z^i O_N}{1 + z^{i+1}O_N} \cong \frac{z^i O_N}{z^{i+1} O_N}$$
via the character $\chi^i$.  As a $(\mathbb{Z}/p)[C_m]$-module, $W_i$ is the direct sum
of finitely many copies of the unique simple $(\mathbb{Z}/p)[C_m]$-module $V_i$ whose character is the sum
of the conjugates of $\chi^i$ over $\mathbb{Z}/p$. Each simple $(\mathbb{Z}/p)[C_m]$-module
is isomorphic to $V_i$ for some $i \not \equiv 0$ mod $p$.  Finally $W_i$ and $W_{i+m}$
are isomorphic, so $V_i$ and $V_{i+m}$ are isomorphic.   
 
 Suppose first that $q/p > 2$, and recall that we have assumed $p^2|q$.  There is a direct sum decomposition 
 $E = T_0 \oplus T_1$ of $E$ as a $(\mathbb{Z}/p)[C_m]$-module in which $T_0$ is a simple $(\mathbb{Z}/p)[C_m]$-module.
 Let $T$ be an order $p$ subgroup of $T_0$.  We claim there is an integer $j$ such that $ j \not \equiv 0 $ mod $p$,
 $1+j \not \equiv -c(T)$ mod $q/p$, $T_0$ is isomorphic to $V_j$ and $j \ge M-1$.
The condition
 that $j \not \equiv 0$ mod $p$ removes $q/p^2$ residue classes mod $q/p$, while $1 + j \not \equiv -c(T)$ mod $q/p$
 removes at most one more residue class  mod $q/p$.  Since $q/p > 2$ by assumption, we have
 $q/p - q/p^2 - 1 = (q/p)(1 - 1/p) - 1  > 0 $ so both
 of these congruences may be satisfied.   The condition that $T_0$ is isomorphic to $V_j$ is a condition on $j$ mod $m$.  Since
 $m$ is prime to $p$ we can find arbitrarily large $j$ satisfying all  three congruences, as claimed.
 
The group $N^*/(N^*)^p$ is a semi-simple $(\mathbb{Z}/p)[C_m]$-module with a descending
filtration whose terms are given by the image of $1 + z^i O_N$ for $i \ge 1$ prime to $p$. The 
successive quotients in this filtration are the $W_i$ above.  Since $V_j$ is by construction isomorphic to $T_0$, it follows from the semi-simplicity of $N^*/(N^*)^p$
and of $E$ that we can find an $H$-stable subgroup $U$ of $N^*$ containing $1 + z^h O_N$
for some $h \ge 1$ with the following properties.
There is an $H$-equivariant isomorphism $ N^*/(U \cdot (N^*)^p) \to E$
which gives rise to surjections $1 + z^j O_N \to E = T_0 \oplus T_1$ and  $1 + z^{j+1} O_N \to T_1$.
Let $F$ be the extension of $N$ corresponding
to $U \cdot (N^*)^p$ by local class field theory.  Then $F/L$ is
a Galois extension, and there is an isomorphism $\mathrm{Gal}(F/L) = G$
such that $G_1 = G_j$ and $G_{j+1} \ne G_j$ does not contain $T_0 \supset T$.
The existence of this $G$-extension shows that $G$ is not an
almost Bertin group for $k$ if $q/p > 2$.  

Suppose now that $q/p = 2$, so that $p = 2$ and $q = 4$.
Let $C_{m'} \subset C_m$ be the kernel of the action of $C_m$
on $E = (C_2)^2$.  Then $C_{m'}$ is in the center of $G$, and $C_m/C_m'$
is a cyclic group of odd order acting faithfully on $E = (C_2)^2$.  It follows that $G/C_{m'}$
is either isomorphic to $E = (C_2)^2$ or to $A_4$.  

Suppose first that $m' = 1$ and $(p,q) = (2,4)$.  In this case, $G$ is isomorphic to either $E = (C_2)^2$
or to $A_4$.   All that we must prove for such $G$ is that part (b) of Proposition \ref{prop:thepoint}
holds when $G$ is isomorphic to $A_4$. Thus we now assume $m =3$.  With the above notation, we can find an integer $j \ge M-1$ such that
$j \equiv 1$ mod $4$,  and $V_j$ is faithful as a module for $C_m =C_3 = H$.  (The 
last condition is equivalent to $j \not \equiv 0$ mod $3$.)  The above
construction now produces an example in which $E = G_1 = G_j \ne G_{j+1}$,
which is all that is required when $m' = 1$.    

We now suppose $m' > 1$, $(p,q) = (2,4)$ and that  $k$ is an arbitrary algebraically closed field of characteristic $p$.  Then $\C_G(C_{m'}) = G \ne C_{m'}$.  If $G/C_{m'}$ is isomorphic to $E = (C_2)^2$, then with the notation of Definition \ref{def:invariantgroup}, 
$$\psi(\{e\},\C_G(C_{m'})/C_{m'}) = 1 + 3 \mu(2) = -2.$$   
Otherwise $G/C_{m'}$ is isomorphic to $A_4$ and $$\psi(\{e\},\C_G(C_{m'})/C_{m'}) = 1 + 3 \mu(2) + 4 \mu(3) = -6.$$  It now follows from Corollary \ref{cor:flipit} that there is no local $G$-cover
for which the Bertin obstruction vanishes, so that $G$ is not a weak Bertin group for $k$.  We can construct 
examples of such covers in which the first jump in the wild ramification
is arbitrarily large by the same arguments used earlier,
so this completes the proof of case (a) of Proposition \ref{prop:thepoint}.

We now suppose that $q = p$, $m \ge 3$ and that $C_m$ acts faithfully
on $E = C_p$.  Let $T = E$.  Then $S(T) = \{T\}$, since the image of $T$
in $G/E \cong C_m$ has to act trivially on $T = E$.  We have $\N_G(T) = G$.
So Theorem \ref{thm:nontrivcase} gives
\begin{equation}
\label{eq:killitnow}
b_T = \frac{\iota(T)}{m}.
\end{equation}
Therefore we just have to produce a $\phi:G \to \mathrm{Aut}_k(k[[t]])$ such
that when $j$ is the first (and only) jump in the wild ramification 
of $G$, $j$ is arbitrarily large and $j+1 = \iota(T)$ is not congruent to $0$
mod $m$.  The action of $C_m$ on $E = C_p$ is via some faithful
character of $C_m$, and $\mathrm{Aut}(C_m)$ acts transitively on these
faithful characters.  Thus by varying the identification of  $\mathrm{Gal}(N/L) = \mathrm{Gal}(\mathbb{F}((y^{1/m}))/\mathbb{F}(y))$ with $C_m$ in our earlier construction of $G$ extensions,
we can produce an example in which $j$ is any positive integer such
that $j \not \equiv 0$ mod $p$ and $j$ is relatively prime to $m$.
Since $m$ is prime to $p$, we can find arbitrarily large $j$
such that $j \not \equiv 0$ mod $p$ and $j \equiv 1$ mod $m$.
Since $m \ge 3$, such $j$ will have $j + 1 \not \equiv 0$ mod $m$,
so we are done.
\end{proof}

\begin{example}
\label{ex:sep}
Suppose $p > 2$ and that $G = C_p \times C_p$, so that $q = p^2$ and
$m = 1$ in Proposition \ref{prop:thepoint}.  Thus $G$ is not an almost Bertin group.  Nevertheless, there exists a $\phi$ for which the Bertin obstruction vanishes.  Namely, each non-trivial $T \in \mathcal{C}$ 
has order $p$, and (\ref{eq:killit2}) shows $b_T = (1 + \iota(T))/p$.
For all positive integers $a \equiv -1$ mod~$p$, we can construct
an injection $\phi:G \to \mathrm{Aut}_k(k[[t]])$ such that $\iota(T) = a$
for all non-trivial subgroups $T$ of $G$.  Thus the Bertin obstruction
such a $\phi$ vanishes.  Moreover, Pagot proves in \cite{P} that when $a = p-1$,
one cannot lift $\phi$ to characteristic $0$.  
\end{example}

\begin{lemma}
 \label{lem:duhduh}
 Suppose that $p$ is odd and that $G$ is the semidirect product of a normal cyclic subgroup
 $E$ of order $p$ with a cyclic group $C_{2\ell}$ of order $2\ell$, where $\ell$ is a prime different
 from $p$, with a generator of $C_{2\ell}$ acting on $E$ by inversion.  
 Then there is a non-trivial cyclic subgroup $T$ of $G$ such that the constant $b_T$ in
 Proposition \ref{prop:special}  is not integral.  Therefore $G$ is an not a weak Bertin group for $k$, not  
an almost Bertin group for $k$, and not a local Oort group for $k$.
 \end{lemma}
 
 \begin{proof}Suppose first that $\ell = 2$ and that $\sigma$ is a generator for $C_{2\ell} = C_4$.  Let $T = \{e,\sigma^2\}$, so that $T$ is in the center of $G$ and $\N_G(T) = \C_G(T) = G$.  The group
 $\C_G(T)/T$ is isomorphic to the dihedral group $D_{2p}$, and $\psi(\{e\},\C_G(T)/T) = 1 + \mu(p) + p\mu(2) = 1 - 1 - p = -p$.  Therefore \ref{cor:flipit} implies no local $G$ cover has vanishing Bertin obstruction.
 To show that $G$ is not a local almost Bertin group for $k$, it will now
be enough to prove that for each integer $M \ge 0$, there is an injection $\phi:G \to \mathrm{Aut}_k(k[[t]]))$ such that $-a_\phi(\tau) \ge M$ for all non-trivial elements $\tau \in G$ of $p$-power order. 
Let $q = p^2$ and let $L/K$ be the cyclic quartic extension $\mathbb{F}_q((z))/\mathbb{F}_q((t))$
for which $z^4 = t$.  One can construct a $\phi$ with the above properties by considering 
$L^*/(L^*)^p$ as a module for $(\mathbb{Z}/p)[\mathrm{Gal}(L/K)]$ and by applying
the class field theory arguments used in the proof of Proposition \ref{prop:thepoint};  we will leave the details to the reader.

In the other case of Lemma \ref{lem:duhduh}, $G = (E.C_2) \times C_\ell$ where $\ell > 2$ is prime, $p \ne \ell$ and $C_2$ acts on $E = C_p$ by inversion.   Let $T$ be the cyclic sugroup
$E \times C_\ell \cong C_{p\ell}$.  Then $T$ has index $2$ in $G$, so $\N_G(T) = G$ while $\C_G(T) = T$.
Hence $\psi(T,\C_G(T)) = 1$, so Corollary \ref{cor:flipit} shows that no local $G$ cover has vanishing
Bertin obstruction.  We can construct injections $\phi:G \to \mathrm{Aut}_k(k[[t]])$ leading
to such covers such that $-a_\phi(\tau)$ is arbitrarily large for all non-trivial elements $\tau \in G$
of $p$-power order by the same local class field theory arguments used in previous cases.  
This completes the proof.
\end{proof}

\begin{cor}
 \label{cor:theanswerhere}
Suppose that $p > 2$, that $H$ is a semi-direct product of a non-trivial $p$-group with a cyclic prime-to-$p$ group, and that $H$ is an almost Bertin group for $k$.  Then $H$ must be either
a cyclic $p$-group or a dihedral group $D_{2p^z}
$ of order $2p^z$ for some $z \ge 1$.
 \end{cor}
 
 \begin{proof}   Suppose the corollary is false for some group $H$.
 By Theorem \ref{thm:oddcasered}, $H$ has a quotient $G$ having one of the forms (1)-(5) there.  Forms (1)-(3) are not almost Bertin groups by 
Proposition \ref{prop:thepoint}, and similarly for (4) and (5) by
Lemma \ref{lem:duhduh}.  So there is a quotient
 $G$ of $H$ that is not an almost  Bertin group for $k$.  Thus $H$ is not an almost Bertin
 group by Corollary \ref{cor:subquots}, and this is a contradiction.
 \end{proof}
 
 \begin{cor}
 \label{cor:twoanswers}
 Suppose that $p = 2$.  Let $G$ be a group which is not cyclic, and which is one
 of the groups described in items (1), (2), (4), (5) or (6) of
 Theorem \ref{thm:evenclass}.  Then $G$ is not an almost-Bertin group for $k$.
 \end{cor}
 
 \begin{proof}  The only $G$ described in items (1), (2), (4), (5) or (6) of Theorem \ref{thm:evenclass}
 which are not covered by Proposition \ref{prop:thepoint}(a) are those described in item
 (1) of Theorem \ref{thm:evenclass} for which the elementary abelian $p =2$ group $E$
 is of order $2$.  However, these $G$ are cyclic, so Corollary \ref{cor:twoanswers} follows.
 \end{proof}

 \begin{prop}
 \label{prop:no4}
 Suppose $p = 2$ and that as in item (3) of Theorem \ref{thm:evenclass}, $G$ is isomorphic to the semi-direct product $E.C_3$ where
the normal subgroup $E $ is isomorphic to $(\mathbb{Z}/4)^2$ and the cyclic group
$C_3$ of order $3$ acts faithfully on $E$.     Then $G$ is not an
 almost Bertin group for $k$.
 \end{prop}
 
 \begin{proof}  Since the ramification groups $G_i$ are normal in $G$, there is an integer $r \ge 1$
such that $G = G_0 \supset G_1 = E = \cdots = G_{r} \ne G_{r+1}$ and $G_{r+1} \subset E^2 = (2\mathbb{Z}/4\mathbb{Z})^2$.
Let $T$ be a cyclic subgroup of order $4$ in $E$.  Then $\N_G(T) = E$ and there are no cyclic subgroups
of $G$ which properly contain $T$.  Now   Theorem \ref{thm:nontrivcase} gives
\begin{equation}
\label{eq:btnice}
 b_T = \frac{1}{[\N_G(T):T]}  \iota(T) = \frac{r+1}{4}.
 \end{equation}

Following \cite[\S IV.3]{corps} we let $G_u$ for $u \ge 0$ be $G_i$ when $i$
is the smallest integer $\ge u$, and we define
$$\varphi(u) = \int_{0}^u \frac{dt}{[G_0:G_t]}.$$
The upper ramification group $G^{\varphi(u)}$ then equals $G_u$.  Since $G_0$ contains $E$ with
index $3$, we find that $\varphi(u) = u/3$ for $0 \le u \le r$.  Thus $G^{r/3} = E$ and $G^{r/3+\epsilon}$ is contained in $E^2$ if $\epsilon > 0$. The group $E^2$ is normal in $G$, and $H = G/E^2$
is isomorphic to $A_4$.  By \cite[Prop. IV.14]{corps}, the image of $G^\nu$ in $H$ is
$H^\nu$ for all $\nu \ge 0$.  Thus $H^{r/3} = E/E^2$ and $H^{r/3 + \epsilon} = \{e\}$ for $\epsilon > 0$.
By comparing the lower and upper ramification groups of $H$, we find that $H_r = E/E^2$
while $H_{r + 1} = \{e\}$.  

Suppose now that $M \ge 0$ is given.  To show that $G$ is not an almost Bertin group for $k$,
it will suffice to show that there is $\phi:G \to \mathrm{Aut}_k(k[[t]])$ such that when 
$r$ is defined as above, $r \ge M- 1$ and $r+1 \not \equiv 0$ mod $4$.  This
is because (\ref{eq:btnice}) will then show $b_T$ is not integral, so the Bertin
obstruction of $\phi$ does not vanish by Proposition \ref{prop:special}.  

To construct
such a $\phi$, we apply Proposition \ref{prop:subquots} of Appendix 1 to the surjection
$G \to H = A_4$.  This produces an integer $M'$ depending on $M$ for which we may use the
following argument.  Replace $M$ by $M'$ in Proposition \ref{prop:thepoint} and  let $r$ be the integer $j$ in part (b) of Proposition \ref{prop:thepoint}.  Proposition \ref{prop:thepoint}
then produces an injection $\psi:H = A_4 \to \mathrm{Aut}_k(k[[z]])$ such that  $H_1 = H_r \ne H_{r+1}$
for some integer $r \ge M'-1$  such that $r + 1 \equiv 2$ mod $4$.  
Proposition \ref{prop:subquots}  now produces an injection $\phi:G \to \mathrm{Aut}_k(k[[t]])$
as follows.  The $H$-cover associated to $\psi$ is the quotient of the local $G$-cover associated to $\phi$, and  $-a_\phi(\tau) \ge M$ if $\tau$ is
a $p = 2$-torsion element of $G$.  We now see from the above computation of
upper and lower ramification groups that $r$ is the first jump in the wild lower
ramification filtration of $G$, so we are done.
\end{proof}

\begin{prop}
 \label{prop:no9}
 Suppose $p = 2$ and that as in item (7) of Theorem \ref{thm:evenclass}, $G$ is isomorphic to
 $(\mathbb{Z}/4) \times (\mathbb{Z}/2)$.     Then $G$ is not an
 almost Bertin group for $k$.
 \end{prop}
 
 \begin{proof}  Suppose $M \ge 1$.  Let $H = \mathbb{Z}/2$ be the second factor in $G$, so that there
 is a split surjection $\pi:G \to H$.  Let $M'$ be an integer having the properties for $M$ and $G \to H$
 described in Proposition \ref{prop:subquots} of Appendix 1.  We can construct an $H$-extension $N/L$ of $L = k((y))$
 such that the first (and only) jump in the lower numbering ramification filtration of $H$
 occurs at an integer $r \ge M'-1$ such that $r \equiv 1$ mod $4$.  By Proposition 
 \ref{prop:subquots}(i), there is an injection $\phi:G \to \mathrm{Aut}_k(k[[t]])$ which defines
 a local $G$-cover of $L = k((y))$ having $N/L$ as the quotient cover associated to 
 $\pi:G \to H$.  By Proposition 
 \ref{prop:subquots}(ii), we can furthermore require that $a_\phi(\tau) \ge r+1$ for all
 non-trivial elements $\tau \in \mathrm{ker}(\pi)$, since all such $\tau$ have order a power of $p = 2$.  
 This means that $\mathrm{ker}(\pi) \subset G_r$.  
 Since $\pi(G^\nu) = H^\nu$ for all $\nu$, we conclude that $\pi(G^r) = H$
 while $\pi(G^{r+\epsilon}) = \{e\}$ for $\epsilon > 0$.  Now $\mathrm{ker}(\pi) \subset G_r \subset G^r$,
 so we conclude that $G^r = G$, while $G^{r+\epsilon} \subset \mathrm{ker}(\pi)$ for $\epsilon > 0$.
 The first jumps in the lower and upper ramification filtrations of $G$ are equal, so we deduce
 from this that $G_r = G$ while $G_{r+1} \subset \mathrm{ker}(\pi)$.  When we now view $H = \mathbb{Z}/2$
 as a subgroup of $G = (\mathbb{Z}/4) \times H$, we see that $\iota(H) = r+1$.  Furthermore,
 $\N_G(H) = G$, while there are no cyclic subgroups of $G$ which property contain $H$.
 Thus
 $$b_H = \frac{1}{[\N_G(H):H]} \iota(H) = \frac{r+1}{4}.$$
 Since we arranged that $r \equiv 1$ mod $4$, this proves $b_H$ is not integral, 
 so the Bertin obstruction of $\phi:G \to \mathrm{Aut}_k(k[[t]])$ is non-trivial by
Proposition \ref{prop:special}.  Because $G_r = G$ and $r \ge M-1$, this completes
the proof that $G$ is not an almost Bertin group for $k$.
\end{proof}

 \begin{cor}
 \label{cor:summary}
 To complete the proof of Theorem \ref{thm:nonexs}, it will suffice to show the following:
 \begin{enumerate}
 \item[a.]  The groups listed in items (1) - (4) of Theorem \ref{thm:nonexs} are KGB groups for $k$.
 \item[b.]  When $p = 2$, neither the quaternion group $Q_8$ nor the group $\mathrm{SL}_2(3)$
is a Bertin group for $k$.
 \item[c.]  When $p = 2$, no semi-dihedral group of order at least $16$ is a Bertin group for $k$.
 \end{enumerate}
 \end{cor}
 \begin{proof}  If $p$ is odd, this follows from Corollary \ref{cor:theanswerhere}, since KGB groups are Bertin and hence almost Bertin. Suppose
 now that $p = 2$. If we grant the results stated in parts (b) and (c) of Corollary \ref{cor:summary},
 then Corollary \ref{cor:twoanswers} together with Propositions \ref{prop:no4} and \ref{prop:no9}
 show that none of the groups listed in items (1)  - (7) of Theorem \ref{thm:evenclass}   are
 Bertin groups for $k$, and that $Q_8$, $\mathrm{SL}_2(3)$ and semi-dihedral groups of order $\ge 16$ are not Bertin groups for $k$.  By Corollary \ref{cor:subquots}, no group $G$ that has one of these
 groups as a quotient can be a Bertin group for $k$.  Thus Theorem \ref{thm:evenclass} shows that
 if $G$ is a cyclic-by-$p$ group which is a Bertin group for $k$ for $p = 2$, it must be cyclic, dihedral, generalized quaternion of order at least $16$, or $A_4$.  Thus a proof 
 that all of these groups are in fact KGB groups for $k$ will complete the proof of Theorem \ref{thm:nonexs}.
 \end{proof}
 
   \section{Reduction to quasi-finite residue fields.}
   \label{s:quasired}
   \setcounter{equation}{0}

To further apply classfield theory to study  Artin characters, it is useful to be able to replace the algebraically closed field $k$ by a quasi-finite field.  We first recall the definition of such fields from  \cite[\S XIII.2]{corps}.

\begin{dfn}
A field $L$ of characteristic $p > 0$ is \textit{quasi-finite} if it has the following properties:
\begin{enumerate}
\item[a.]  $L$ is perfect;
\item[b.]  There is an automorphism
$F \in \mathrm{Gal}(L^{\mathrm{sep}}/L)$ of the separable closure $L^{\mathrm{sep}}$ of $L$
such that the map $\hat {\mathbb{Z}} \to \mathrm{Gal}(L^{\mathrm{sep}}/L)$ defined by $\nu \mapsto F^\nu$ is an isomorphism of profinite groups.
\end{enumerate}
\end{dfn}

\begin{prop}
\label{prop:reducetoquasi}
Suppose that $G$ is a finite group, $k$ is an algebraically closed field of characteristic $p$
and that $\phi:G \to \mathrm{Aut}_k(k[[t]])$ is an injection.  There is a subfield $k'$ of $k$
of finite type over the prime field $\mathbb{F}_p$ such that $\phi$ is the base change from $k'$
to $k$ of a unique injection $\phi':G \to \mathrm{Aut}_{k'}(k'[[t]])$.  There is a quasi-finite field $L$
containing $k'$ such that $\phi'$ induces an injection $\phi'_L:G \to \mathrm{Aut}_{L}(L[[t]])$
with the following properties.  Let $\overline{L}$ be an algebraic closure
of $L$.  Then $\phi'_L$ induces an injection $\phi'_{\overline{L}}:G \to \mathrm{Aut}_{\overline{L}}(\overline{L}[[t]])$, and the Artin characters of $\phi$, $\phi'$, $\phi'_L$ and $\phi'_{\overline{L}}$ are equal.
\end{prop} 

\begin{proof}
The existence of $k'$ and $\phi'$ is clear from the fact that a Katz-Gabber $G$-cover
associated to $\phi$, together with the action of $G$ on this cover, is defined over a field of finite type
over the prime field $\mathbb{F}_p$.  Since $\phi$ defines a totally ramified action of $G$,
so does $\phi'$.  If $k'$ is finite, we can therefore take $L$ to be  $k'$.
Suppose now that $k'$ has positive transcendence degree over $\mathbb{F}_p$.  By the 
Noether normalization theorem, $k'$ is a finite extension of a rational subfield $\mathbb{F}_p(t_1,\dots,t_n)$ for some algebraically independent indeterminates $t_1,\ldots,t_n$, where $n \ge 1$.  
Let $k_1 = \overline{\mathbb{F}_p(t_1,\dots,t_{n - 1})}$ be an algebraic closure of the
subfield $\mathbb{F}_p(t_1,\dots,t_{n-1})$, and let $N$ be the compositum of $k'$ and
$k_1$ in an extension field of $k$.  Then $N$ has transcendence degree $1$ over $k_1$.
Since $k_1$ is algebraically closed, $N$ is the  function field a smooth projective curve $V$ over $k_1$,
and $k_1$ is the field of constants of $V$.  By \cite{Harbater} and \cite{Pop},  $\mathrm{Gal}(\overline{N}/N) = \mathrm{Gal}(\overline{k_1(V)}/k_1(V))$ is a free profinite group of countable rank since $k_1$ is countable.  Let $F$ be one element of a set of topological
generators for $\mathrm{Gal}(\overline{N}/N)$, and let $L = \overline{N}^{\langle F \rangle}$
be the fixed field of $F$ acting on $\overline{N}$.  Then $L$ is a quasi-finite
field, with algebraic closure $\overline{L} = \overline{N}$ and an isomorphism $\hat{\mathbb{Z}} \to \mathrm{Gal}(\overline{L}/L)$ defined by $\nu \mapsto F^\nu$.  Since $k' \subset N \subset L$, we can let 
$\phi_L:G \to \mathrm{Aut}_{L}(L[[t]])$ be the base change of $\phi'$ from $k'$ to $L$.
Since $\phi$, $\phi_L$ and $\phi_{\overline{L}}$ are base changes of $\phi'$, all of the associated
Artin characters are equal.
\end{proof}

\begin{cor}
\label{cor:quasidone}
Fix an algebraically closed field $k$ of characteristic $p > 0$, and let $a $ be a complex
character of $G$.  There is an injection
$\phi:G \to \mathrm{Aut}_k(k[[t]])$ for which $a_\phi = a$ if and only if there is a
quasi-finite field $L$ of characteristic $p$ together with an injection $\phi_L: G\to \mathrm{Aut}_L(L[[t]])$
such that  $a_{\phi_L} = a$. 
\end{cor}

\begin{proof} Given $\phi$, we can take $L$ and $a_{\phi_L}$ to be as in 
Proposition \ref{prop:reducetoquasi}.  Given $L$ and $\phi_L$, we take $k = \overline{L}$,
and we let $\phi$ be the base change of $\phi_L$ from $L$ to $k$.
\end{proof}

\section{Dihedral, Quaternion and Semi-dihedral groups: Ramification filtrations.}  
\label{s:exts2}
\setcounter{equation}{0}

The object of this section is to begin the analysis of the Bertin obstruction for certain dihedral,
generalized quaternion and semi-dihedral groups.  

The following lemma is an example from the end of \S IV.3 of \cite{corps}.  Recall that a real number $\mu$ is a \textit{jump} in the upper (resp.\ lower) ramification filtration of a 
subgroup $J \subset G$ if $J^{\nu} \ne J^{\nu+\epsilon}$ (resp.\ $J_{\nu} \ne J_{\nu + \epsilon}$)
for all $\epsilon > 0$.  
 
\begin{lemma}
\label{lem:jumps} \cite{corps}
Let $k$ be a field of characteristic $p>0$, let $H$ be a cyclic group of order $p^n$, and assume we are given a Galois extension of $k((t))$ with group $H$. Then there are positive integers $i_0,i_1,\ldots,i_{n-1}$ such that
the jumps in the upper numbering of the ramification filtration of $H$ occur
at $i_0, i_0+i_1,\ldots,i_0 + i_1 + \cdots + i_{n-1}$.  We have ramification groups
\begin{eqnarray}
\label{eq:jumplist}
H_0 &=& \cdots = H_{i_0} = H = H^0 = \cdots = H^{i_0}\nonumber\\
H_{i_0 + 1}  &=& \cdots = H_{i_0+pi_1} = pH = H^{i_0 + 1} = \cdots = H^{i_0+i_1}\nonumber\\
H_{i_0 + pi_1 + 1}  &=& \cdots = H_{i_0+pi_1 + p^2 i_2} = p^2 H = H^{i_0 + i_1 + 1} = \cdots = H^{i_0+i_1+ i_2}\\
&\cdots&\nonumber\\
H_{i_0+p i_2 + \cdots + p^{n-1} i_{n-1} + 1} &=& p^{n} H = \{e\} = H^{i_0 + \cdots i_{n-1} + 1}.\nonumber
\end{eqnarray}
Thus the jumps in the lower ramification filtration are at 
$\sum_{j = 0}^\ell p^j i_j$ for $0\le \ell  \le n-1$.
\end{lemma}

For the remainder of this section we make the following standing hypothesis:

\begin{hyp}
\label{hyp:good}
Let $k$ be an algebraically closed field of characteristic $p>0$ and let $n \ge 1$ be an integer.  The group $G$ is of order $2p^n$,
is generated by a cyclic subgroup $H = \langle \tau \rangle$ of order $p^n$ and an 
element $\sigma$.  In addition to the relation $\tau^{p^n} = e$, $G$ is specified by the following relations:
\begin{enumerate}
\item[a.] (Dihedral case) $\sigma^2 = e$ and $\sigma \tau \sigma^{-1} = \tau^{-1}$.
\item[b.] (Generalized quaternion case) $p = 2$, $n \ge 2$, $\sigma^2 = \tau^{p^{n-1}}$, $\sigma \tau \sigma^{-1} = \tau^{-1}$.
\item[c.](Semi-dihedral case) $p = 2$, $n \ge 3$, $\sigma^2 = e$, $\sigma \tau \sigma^{-1} = \tau^{-1 + p^{n-1}}$.
\end{enumerate}
Let $\phi:G \to \mathrm{Aut}_k(k[[t]])$ be an injection. For $\Gamma$ a subgroup of $G$,
let $\Gamma_\nu$ and $\Gamma^\nu$ be the lower and upper ramification subgroups of $\Gamma$
associated to $\nu \in \mathbb{R}$.
\end{hyp}

Under Hypothesis~\ref{hyp:good}, Lemma~\ref{lem:jumps} yields:

\begin{cor}
\label{cor:computeiota}
Suppose that $\Gamma = p^j H$ is a non-trivial subgroup of $H$, so that $0 \le j \le n-1$.  Then
\begin{equation}
\label{eq:iotacal1}
\iota(\Gamma) = 1+ i_0 + p i_1 + \cdots + p^j i_j
\end{equation}
\end{cor}

It is straightforward to verify the following lemma and corollary.

\begin{lemma}
\label{lem:subgrps}
A set $\mathcal{C}$ of representatives for the cyclic subgroups $T$ of $G$
may be given as follows.  For $0 \le j \le n$ let $p^j H= \langle \tau^{p^j}\rangle $ be the subgroup of index $p^j$ in $H$, and let $\mathcal{H} = \{p^j H: 0 \le j \le n\}$.    One has $\N_G(T) = T$ for $T \in \mathcal{H}$.
\begin{enumerate}
\item[a.] (Dihedral case when $p > 2$) $\mathcal{C} = \mathcal{H} \cup \{D_1\}$,
where $D_1 = \langle \sigma \rangle$ has order $2$ and  
 $\N_G(D_1)  = D_1$. 
 \item[a$'$.] (Dihedral case when $p = 2$) $\mathcal{C} = \mathcal{H} \cup \{D_1, D_2\}$
 where $D_1 = \langle \sigma \rangle$ and $D_2 = \langle \tau \sigma \rangle$ have
 order $2$.  The group $\N_G(D_i) = \langle D_i, 2^{n-1}H \rangle$ contains $D_i$
with index $2$. 
\item[b.] (Generalized quaternion case) $\mathcal{C} = \mathcal{H} \cup \{D_1,D_2\}$
where $D_1 = \langle \sigma \rangle$ and $D_2 = \langle \tau \sigma \rangle$
have order $4$.  The group $\N_G(D_i) = \{2^{n-2}H,D_i\}$ contains $D_i$ with index $2$. 
\item[c.] (Semi-dihedral case)  $\mathcal{C} = \mathcal{H} \cup \{D_1,D_2\}$
where $D_1 = \langle \sigma \rangle$ has order $2$ and $D_2 = \langle \sigma \tau \rangle$
has order $4$.  One has $\N_G(D_1) = \langle 2^{n-1}H,D_1\rangle$ and $\N_G(D_2) = \langle 2^{n-2}H,D_2\rangle$.  For $i = 1, 2$, the index $[\N_G(D_i):D_i]$ equals $2$.
\end{enumerate}
\end{lemma}

\begin{cor}
\label{cor:Sdescrip}
Suppose $T \in \mathcal{C}$ is non-trivial.  Recall that $S(T) $ is the set of non-trivial cyclic subgroups
$\Gamma \subset G$ which contain $T$.  Define $S'(T)$ to be the set of $\Gamma \in S(T)$
such that $\mu([\Gamma:T])$ is non-zero, i.e.\  for which $[\Gamma:T]$ is square-free.  Then
$S'(T)$ has the following description.
\begin{enumerate}
\item[a.]  If $T = D_i$ for some $i$ as in Lemma \ref{lem:subgrps}, then $S'(T) = \{T\}$.
\item[b.] Suppose $T = p^j H$ for some $0 \le j \le n-1$ and that either $G$ is dihedral or $j \ne n-1$.  Then
$S'(T) = \{T\}$ if $j = 0$ and $S'(T) =  \{T,p^{j-1}T\}$ if $0 < j$.
\item[c.] Suppose $G$ is quaternionic and $T = 2^{n-1}H$. Then $S'(T)$
is the union of $\{T,2^{n-2}H\}$ with the set of $\# (G/\N_G(D_1)) = 2^{n-2}$ distinct conjugates
of $D_1$ and the set of $\# (G/\N_G(D_2)) = 2^{n-2}$ distinct conjugates of $D_2$.
\item[d.]  Suppose $G$ is semi-dihedral and $T = 2^{n-1}H$.  Then $S'(T)$
is the union of $\{T,2^{n-2}H\}$ with the set of $\# (G/\N_G(D_2)) = 2^{n-2}$ distinct conjugates of $D_2$.
\end{enumerate}
\end{cor}

\begin{cor}
\label{cor:nicesubgroups}
Suppose that $T = p^j H$ is a non-trivial subgroup of $H$ satisfying the conditions
of part (b) of Corollary \ref{cor:Sdescrip}.  Thus either $G$ is dihedral or $0 \le j < n-1$.  
Then
\begin{equation}
\label{eq:answer}
b_T  = \frac{1 + i_0}{2}\quad \mathrm{ if}\quad  j = 0\quad \mathrm{ and}\quad b_T =  
\frac{i_j}{2}\quad \mathrm{ if}\quad  j > 0. 
\end{equation}
\end{cor}

\begin{proof}
Recall that 
\begin{equation}
\label{eq:bigt}
b_T  = \frac{1}{[\N_G(T):T]} \sum_{\Gamma \in S(T)} \mu([\Gamma:T]) \iota(\Gamma).
\end{equation}
Since $T$ is normal in $G$, $[\N_G(T):T] = \# G/\# T = 2p^n/p^{n-j} = 2p^j$.  The only
$\Gamma$ which contribute to the sum for $b_T$ are those $\Gamma$ in $S'(T)$.
Hence Corollary \ref{cor:Sdescrip} gives $b_T = \frac{1}{2}\iota(H)$ if $j = 0$ 
while $b_T = \frac{1}{2p^j}(\iota(p^jH) - \iota(p^{j-1}H))$ if $j > 0$.  Corollary
\ref{cor:computeiota} now gives
the stated formulas for $b_T$.
\end{proof}

\begin{cor}
\label{cor:di}
Suppose $p > 2$, $G$ is dihedral, and that $T = D_1$ is as in Lemma \ref{lem:subgrps}(a).
Then $b_{T} = \iota(T) = 1$.
\end{cor}

\begin{proof}
This is clear from the general formula (\ref{eq:bigt}) for $b_T$ and the fact that $S'(T) = \{T\}$, $\N_G(T) = T$ and $T$ has order $2$, which is prime to $p$.
\end{proof}

\begin{lemma}
\label{lem:di2}
Suppose that $G$, $D_1$ and $D_2$ are as in parts (a$'$), (b) or (c) of Lemma \ref{lem:subgrps}.  Let
$\phi:G \to \mathrm{Aut}_k(k[[t]])$ be an injection.  Define $D_0 = H$.  Let $N = k((t))$,
$K = N^G$,   and $L_i = N^{\langle 2H,D_i\rangle }$ for $i = 0,1,2$, where $\langle 2H, D_i\rangle$
is the subgroup generated by $2H$ and $D_i$.  Then $L_i/K$ is quadratic for $i = 0, 1, 2$, with relative discriminant ideal
$(m_K)^{d_i}$ for some even integer $d_i$, where $m_K$ is the maximal ideal
of the integers $O_K$ of $K$.  
Moreover for $i = 0, 1, 2$,
\begin{equation}
\label{eq:bdeq}
b_{D_i} = \frac{\left ( \sum_{j \in \{0,1,2\}, j \ne i} d_j \right ) - d_i}{2}
\end{equation}
is a positive integer.
\end{lemma}
\begin{proof}
Let $\psi_i$ be the quadratic one-dimensional character of
$G= \mathrm{Gal}(N/K)$ which is the inflation of the non-trivial one dimensional
character of $\mathrm{Gal}(L_i/K)$. Then $\psi_i$ is trivial on $D_i$.  
Write 
\begin{equation}
\label{eq:aphispecial}
-a_\phi = \sum_{T \in \mathcal{C}} b_T \ 1_T^G = \sum_{2H \supset T \in \mathcal{C}} b_T \ 1_T^G + 
\sum_{i = 0}^2 b_{D_i} \ 1_{D_i}^G .
\end{equation}
Take the inner product of this expression with $\chi_0 - \psi_i$ when $\chi_0$ is
the one-dimensional trivial character of $G$.  If $T \subset 2H$ then $\langle  1_T^G, \chi_0 - \psi_i \rangle
= 0$, and 
\begin{equation}
\label{eq:innerp}
\langle -a_\phi, \chi_0 - \psi_i \rangle = 
- \langle a_\phi, \chi_0 \rangle + \langle a_\phi, \psi_i \rangle = 0 +  \langle a_\phi,\psi_i \rangle = d_i
\end{equation} by
\cite[\S VI.2]{corps}.  For $i,j \in \{0,1,2\}$, one has
\begin{equation}
\label{eq:uppity}
\langle 1_{D_j}^G,  \chi_0 - \psi_i \rangle = \langle 1_{D_j}^G,  \chi_0\rangle  - \langle 1_{D_j}^G,  \psi_i \rangle = 1 - \delta(i,j)
\end{equation}
since
the restriction of $\psi_i$ to $D_j$ is trivial if $i = j$ and non-trivial otherwise.  Combining
(\ref{eq:aphispecial}), (\ref{eq:innerp}) and (\ref{eq:uppity}) gives
the system of equations
\begin{eqnarray}
\label{eq:deqs}
d_0 &=& b_{D_1} + b_{D_2} \nonumber\\
d_1 &=& b_{D_0} + b_{D_2} \\
d_2 &=& b_{D_0} + b_{D_1} \nonumber
\end{eqnarray}
The formula (\ref{eq:bdeq}) is clear from this.  The exponents $d_0, d_1, d_2$ are even and positive since $p = 2$, so
that all the $b_{D_i}$ are integral.  The compositum of $L_0, L_1$ and $L_2$ over $K$
is a biquadratic extension of $K$.  Thus either all the $d_i$ are equal, or two are equal 
and the third is smaller than these two.  This implies that all the $b_{D_i}$ are positive, which
completes the proof.    
\end{proof}

\begin{cor}
\label{cor:hunch}  Assume the hypotheses of Lemma~\ref{lem:di2}.  Then 
$$\iota(D_i) = 2 b_{D_i} = \left ( \sum_{j \in \{0,1,2\}, j \ne i} d_j \right ) - d_i$$
for $i = 0, 1 , 2$.
\end{cor}

\begin{proof}  By Lemma \ref{lem:subgrps} and Corollary \ref{cor:Sdescrip}, $[\N_G(D_i):D_i] = 2$
and $S'(D_i) = \{D_i\}$, so the result follows from (\ref{eq:bigt}) and Lemma~\ref{lem:di2}.
\end{proof}

\begin{cor}
\label{cor:hardpart}
With the hypotheses of Lemma~\ref{lem:di2}, suppose $G$ is quaternionic or semi-dihedral,
and that $T = 2^{n-1}H$.  
\begin{enumerate}
\item[a.]  If $G$ is quaternionic then $$b_T = \frac{i_{n-1}}{2} - \frac{d_0}{2}.$$

\item[b.] If $G$ is semi-dihedral then $$b_T =  \frac{i_{n-1}}{2} - \frac{d_0 + d_1 - d_2}{4}.$$
\end{enumerate}
\end{cor}

\begin{proof}  If $G$ is quaternionic, then  Corollaries \ref{cor:Sdescrip}, 
\ref{cor:computeiota} and \ref{cor:hunch} give
\begin{eqnarray}
b_T &=& \frac{1}{[\N_G(T):T]} \sum_{\Gamma \in S(T)} \mu([\Gamma:T]) \iota(\Gamma) \nonumber\\
&=&\frac{1}{2^n}\left (\iota(2^{n-1}H) - \iota(2^{n-2}H) - 2^{n-2}(\iota(D_1) + \iota(D_2))\right )\\
&=& \frac{1}{2^n}\left ( 2^{n-1} i_{n-1} - 2^{n-1}(b_{D_1} + b_{D_2}) \right ) \nonumber\\
&=& \frac{i_{n-1}}{2} - \frac{d_0}{2}.
\end{eqnarray}
If $G$ is semi-dihedral, the same arguments show
\begin{eqnarray}
b_T &=& \frac{1}{[\N_G(T):T]} \sum_{\Gamma \in S(T)} \mu([\Gamma:T]) \iota(\Gamma) \nonumber\\
&=&\frac{1}{2^n}\left (\iota(2^{n-1}H) - \iota(2^{n-2}H) - 2^{n-2} \iota(D_2))\right )\\
&=& \frac{1}{2^n}\left ( 2^{n-1} i_{n-1} - 2^{n-1} b_{D_2} \right ) \nonumber\\
&=& \frac{i_{n-1}}{2} - \frac{d_0 + d_1 - d_2}{4}
\end{eqnarray}
\end{proof}

\begin{cor}
\label{cor:closer}The Bertin obstruction of an injection $\phi:G \to \mathrm{Aut}_k(k[[t]])$
 vanishes if and only
if the following conditions hold:
\begin{enumerate}
\item[a.] $i_0$ is odd, and $i_j$ is even for $0 < j < n-1$.
\item[b.]  If $G$ is dihedral, $i_{n-1}$ is even.
\item[c.]  If $G$ is quaternionic, $i_{n-1}$ is even and $i_{n-1} \ge d_0$. 
\item[d.] If $G$ is semidihedral, $\frac{i_{n-1}}{2} - \frac{d_0 + d_1 - d_2}{4}$
is a non-negative integer.
\end{enumerate}
\end{cor}

\begin{proof}  By Proposition \ref{prop:special}, the Bertin obstruction of $\phi$ vanishes
if and only if $b_T$ is a non-negative integer for $T$ a non-trivial subgroup contained in the set $\mathcal{C}$ described in Lemma \ref{lem:subgrps}.  Those non-trivial $T \in \mathcal{C}$ contained in $H$ are treated in Corollaries \ref{cor:nicesubgroups} and \ref{cor:hardpart} since the $d_i$ in Lemma~\ref{lem:di2} are even.  The $T \in \mathcal{C}$ 
which are not contained in $H$ are treated in Corollary \ref{cor:di} and Lemma~\ref{lem:di2}.
Conditions (a) - (d) are equivalent to the statement that the $b_T$ in Corollaries \ref{cor:nicesubgroups}, \ref{cor:hardpart} and \ref{cor:di}
and in Lemma \ref{lem:di2} are non-negative integers.  
\end{proof}

\begin{cor}
\label{cor:closertwo}  Suppose that the Bertin obstruction of $\phi$
vanishes, so that (a) - (d) of Corollary \ref{cor:closer} hold.   
Then the KGB obstruction of $\phi$ vanishes. 
\end{cor}

\begin{proof} We claim that the KGB obstruction vanishes if  we can find for each non-trivial $T \in \mathcal{C}$ a sequence of elements $\{g_{T,i}\}_{i=1}^{b_T}$ of $G$ with the following properties:
\begin{enumerate}
\item[i.] Each 
$g_{T,i}$ is in a conjugate of $T$;
\item[ii.]  There is an ordering $\{g_t\}_{t\in \Omega}$ of the doubly indexed set $\{g_{T,i}\}_{T,i}$, counting
multiplicities, such that $\prod_{t \in \Omega} g_t$ has order $[G:G_1]$.
\end{enumerate}

To prove this claim, suppose we can find $\{g_t\}_{t\in \Omega}$ as above.  
In Theorem \ref{thm:KGBtwo}(b) we can then take $S$ to be $\coprod_{t \in \Omega} G/\langle g_t \rangle$
provided we show that $\{g_t\}_{t \in \Omega}$ generates $G$.   Suppose
first that $p > 2$, so $G = D_{2p^n}$.  By Corollary \ref{cor:di}, $b_{D_1} = 1$, so there is one $g_t$
which has order $2$.  By Corollary \ref{cor:nicesubgroups}, $b_H  = \frac{1 + i_0}{2} > 0$.  Thus
the subgroup of $G$ generated by $\{g_t\}_{t \in \Omega}$ contains $H$ and $D_1$, so must
be all of $G$.  Suppose now that $p = 2$.  By Lemma~\ref{lem:di2}, $b_H$, $b_{D_1}$ and $b_{D_2}$
are positive.  Hence the subgroup of $G$ generated by $\{g_t\}_{t\in \Omega}$ surjects onto 
the Klein four quotient $G/2H$ of $G$. This implies this subgroup must be all of $G$.   

We now have to show that we can choose the $g_{T,i}$ so that (i) and (ii) hold.  

We consider first the case $p > 2$.  Then $G$ is isomorphic to $D_{2p^n}$
for some $n \ge 1$, and (\ref{eq:congruencenew}) holds vacuously.  By Corollary \ref{cor:di},
$b_{D_1} = 1$.  It follows that if we pick the $g_{T,i}$ to be any generators of
$T$, and pick any ordering $\{g_t\}_{t\in \Omega}$ of all these $g_{T,i}$, then
the product $\prod_{t\in \Omega} g_t$ projects to the non-trivial element of $G/H$.
Hence this product has order $2 = [G:G_1]$, since every element of $G = D_{2p^n}$
not in $H$ has order $2$.    Hence (i) and (ii) hold.

Suppose now that $p = 2$ and that $G$ is a $2$-group and is either dihedral, quaternionic
or semi-dihedral.  The quotient $G/(2H)$ is then isomorphic to the Klein
four group, and $G = G_1$. We claim that
\begin{equation}
\label{eq:congruencenew}
b_H \equiv b_{D_1} \equiv b_{D_2} \ \mathrm{mod} \ 2\mathbb{Z}.
\end{equation}
Here $b_H = b_{D_0}$ in the terminology of 
Lemma~\ref{lem:di2}, where it was shown that
\begin{equation}
\label{eq:bdeqmore}
b_{D_i} = \frac{\left ( \sum_{j \in \{0,1,2\}, j \ne i} d_j \right ) - d_i}{2}
\end{equation}
for $i = 0, 1, 2$. Since all the $d_i$ are even, we have $d_i \cong -d_i$ 
mod $4\mathbb{Z}$.  Thus
$$2b_{D_i} \equiv \sum_{j \in \{0,1,2\}} d_j \quad \mathrm{mod}\quad 4 \mathbb{Z}$$
from which (\ref{eq:congruencenew}) is clear.
Let the $g_{T,i}$ be any choice of generators for $T$ as $T$ ranges
over $\mathcal{C}$ and $i = 1, \ldots,b_T$.  By Corollary \ref{cor:nicesubgroups},
$b_H > 0$.  Hence we can choose an ordering  $\{g_t\}_{t\in \Omega}$ of these $g_{T,i}$ such that the first $g_t$ is a generator of $H$.  The congruence (\ref{eq:congruencenew})
implies that 
$\prod_{t \in \Omega} g_t$ lies in $2H$.  We can now multiply the first $g_t$
by an element of $2H$ to produce a new generator of $H$ such that
when we use this element as the first $g_t$ we have $\prod_{t \in \Omega} g_t = e$.
This shows that (i) and (ii) hold and completes the proof.
\end{proof}

\begin{cor}
\label{cor:closerthree}
With $G$ a finite group and $k$ a field of characteristic $p>0$, let $P(G,k)$ be the assertion that every injection $\phi:G \to \mathrm{Aut}_k(k[[t]])$ satisfies
conditions (a) - (d) of  Corollary \ref{cor:closer}.   Then fixing $G$, the assertion $P(G,k)$ holds for all algebraically closed fields $k$ of characteristic
$p$ if and only if $P(G,k)$ holds for all quasi-finite fields $k$ of characteristic $p$.  The same is true if we add condition (\ref{eq:congruencenew}) to $P(G,k)$. 
\end{cor}

\begin{proof}
This a consequence of Corollary \ref{cor:quasidone}.
\end{proof}

\section{Dihedral, Quaternion and Semi-dihedral groups: Class field theory.}
\label{s:cft}
\setcounter{equation}{0}

In the section we will assume the hypotheses and notation of the previous
section, with the following modifications:

\begin{hyp}
\label{hyp:setuplocal}  The field $k$ is quasi-finite (rather than algebraically closed)
of characteristic $p$.  Fix an injection $\phi:G \to \mathrm{Aut}_k(k[[t]])$.  Let $N = k((t))$,
$L = N^H$ and $K = N^G$.   Define $\chi:L^* \to H = \mathrm{Gal}(N/L)$ to be the
Artin isomorphism.  Let $\overline{\sigma}$ be the image of $\sigma \in G$ in $G/H = \mathrm{Gal}(L/K)$,
so that $\overline{\sigma}$ has order $2$ and is a generator of $G/H$. We choose a uniformizer
$\pi_L$ in $L$ such that if $p > 2$, $\overline{\sigma}(\pi_L) = -\pi_L$.  Let $\mathrm{Norm}_{L/K}:L \to K$
be the norm.  Define $D_0 = H$, so that $L = N^{D_0}$.   If $p=2$, we also have
the quadratic extensions $L_1$ and $L_2$ of $K$ defined in Lemma~\ref{lem:di2};  in this
case, $L_0 = L, L_1$ and $L_2$ are the three quadratic subfields of $N$ containing $K$.
\end{hyp}

\begin{dfn}
\label{def:types}
Let $\chi:L^* \to \mathbb{C}^*$ be a character of order $p^n$, and let $\chi|_{K^*}$
be the restriction of $\chi$ to $K^*$.   
\begin{enumerate}
\item[a.]  Say $\chi$ is of \textit{dihedral type} if $\chi|_{K^*}$ is trivial.  
\item[b.]  Say $\chi$ is of \textit{quaternionic type} if $p = 2$ and $\chi|_{K^*}$ is the non-trivial quadratic character  $\epsilon_0:K^* \to \{\pm 1\}$ associated to $L/K$.  This is the character with kernel $\mathrm{Norm}_{L/K}(L^*)$.
\item[c.] Say $\chi$ is of \textit{semi-dihedral type} if $p = 2$, $n \ge 3$, $\chi|_{K^*}$ is the non-trivial
quadratic character $\epsilon_1:K^* \to \{\pm 1\}$ associated to $L_1/K$,
and $\chi \circ \mathrm{Norm}_{L/K} = \chi^{2^{n-1}}$.
\end{enumerate}
\end{dfn}

\begin{lemma}
\label{lem:thetypes}
If $G$ is a dihedral (resp.\ quaternion, resp.\ semi-dihedral) group then $\chi$
is of dihedral (resp.\ quaternion, resp.\ semi-dihedral) type.  Conversely,
suspend Hypothesis \ref{hyp:setuplocal} for the moment, and suppose 
$L/K$ is a specified quadratic extension of $K = k((z))$ for some indeterminate $z$.
Let $\chi$ be a character of order $p^n$ of $L^*$, and let $N$ be the 
cyclic extension of $L$ of degree $p^n$ over $L$ which corresponds to the kernel of $\chi$ by local classfield theory.
\begin{enumerate}
\item[i.] Suppose $\chi$ is of dihedral (resp.\ quaternionic) type, in the sense of parts (a) and (b)
of Definition \ref{def:types}.   Then $N$ is a Galois extension 
of $N$ and $\mathrm{Gal}(N/K)$ is a dihedral group (resp.\ generalized quaternion group)
of order $2p^n$.  
\item[ii.]  Suppose that $\chi|_{K^*}$ is an order $2$ character of $K^*$ which
corresponds to a quadratic extension $L_1/K$ different from $L/K$, and that
$\chi \circ \mathrm{N}_{L/K} = \chi^{2^{n-1}}$.  Then $N$
is a semi-dihedral extension of $K$ of degree $2p^n$, with biquadratic subfield 
over $K$ the compositum of $L$ and $L_1$ over $K$.  The character $\chi$ is of
semi-dihedral type in the sense of Definition \ref{def:types}(c).
\end{enumerate}
\end{lemma}

\begin{proof}
Suppose first that Hypothesis \ref{hyp:setuplocal} holds and that $G$ is dihedral, quaternionic or semi-dihedral.  By local
classfield theory, $\chi|_{K^*}:K^* \to \mathbb{C}^*$ corresponds to the
character $G^{ab} \to \mathbb{C}^*$ which is the composition of the
transfer map $\mathrm{ver}:G^{ab} \to H^{ab} = H$ with $\chi:H \to \mathbb{C}^*$.
The action of $\mathrm{Gal}(L/K) = G/H$ on $L^*$ corresponds via $\xi:L^* \to \mathrm{Gal}(N/L) = H$
with the conjugation action of $\mathrm{Gal}(L/K)$ on $H$. The assertions in the lemma
now follow from the properties of this conjugation action and of $\mathrm{ver}$ 
when $G$ is a dihedral, quaternion or semi-dihedral group (see \cite[\S VII.8]{corps}). 
The converse implications of the lemma are proved similarly.
\end{proof} 

\begin{lemma}
\label{lem:jumpform} Assume Hypothesis \ref{hyp:setuplocal} holds. If $0 \le \ell \le n-1$ then
\begin{equation}
\label{eq:cldef}
c(\ell) = i_0 + \cdots + i_\ell
\end{equation}
is a jump in the upper ramification
filtration of $H$.  Let $\chi:L^* \to \mathbb{C}^*$ be a character having the properties in the converse direction of 
Lemma~\ref{lem:thetypes}.  
Then the kernel of $\chi^{p^{n-(\ell+1)}}$ corresponds via class field theory to the extension $N^{p^{\ell+1} H}/L$, which
has Galois group $H/(p^{\ell+1} H)$.    The integer $c(\ell)$ is the largest positive integer such 
that $\chi^{p^{n-(\ell+1)}}$ is non-trivial on $1 + \pi_L^{c(\ell)} O_L$. 
\end{lemma}

\begin{proof} By \cite[\S XV.2]{corps}, if $h \ge 1$
is integral then the image $\xi(1 + \pi_L^h O_L)$ of the multiplicative subgroup $1 + \pi_L^j O_L$ 
under the local Artin map $\xi:L^* \to H$ equals the upper ramification subgroup $H^h$ of $H$.
Since $\chi^{p^{n-(\ell+1)}}$ has kernel $p^{\ell + 1}H$ when we view it as a character of
$H$ via the local Artin map, this leads to the interpretation $c(\ell)$ in the lemma.
\end{proof}

\begin{cor}
\label{cor:firstjump}
The jump $i_0$ is odd.  If $G$ is dihedral, then $i_j$ is even for $0 < j \le n-1$.
Suppose that $G$ is quaternionic or semi-dihedral. Then $c(\ell)$ is odd for $0 \le \ell \le n-2$, $i_j$ is even for $0 < j < n-1$,
and $i_{n-1} = c(n-1) - c(n-2) $ is even if and only if $c(n-1)$ is odd.  
\end{cor}

\begin{proof}By \cite[\S XV.2]{corps}, the local Artin map $\xi$ induces an isomorphism 
\begin{equation}
\label{eq:isotime}
\frac{1+ \pi_L^{c(\ell)} O_L}{1 + \pi_L^{c(\ell)+1} O_L} = H^{c(\ell)}/H^{c(\ell) + 1} = p^\ell H/ (p^{\ell + 1} H)\cong \mathbb{Z}/p \quad \mathrm{for} \quad 0 \le \ell \le n-1.
\end{equation}
This isomorphism is equivariant with respect to the action of $\mathrm{Gal}(L/K)$.
 Recall that we chose the uniformizer $\pi_L$ so that that $\gamma(\pi_L) = - \pi_L$
if $p > 2$, and there is an isomorphism of the left hand side of (\ref{eq:isotime}) with the
one dimensional $k$-vector space $\pi_L^{c(\ell)} k$.   Hence if $p > 2$, then $\gamma$ acts on the left hand side of (\ref{eq:isotime}) by
$(-1)^{c(\ell)}$. We conclude that $c(\ell)$ is odd because $\gamma$ acts by inversion
on the right hand side of (\ref{eq:isotime}).  In view of (\ref{eq:cldef}), this shows that $i_0$ must be odd
and $i_j$ is even for $j > 0$, so we are done in case $p > 2$. 

Suppose now that
$p = 2$.  To complete
the proof, it will suffice by Lemma \ref{lem:jumpform} to show that $c(\ell)$ is odd if either
\begin{enumerate}
\item[i.] $0 \le \ell \le n-1$ and $G$
is dihedral, or 
\item[ii.] $0 \le \ell \le n-2$ and $G$ is either quaternionic or semi-dihedral.
\end{enumerate}

By Lemma \ref{lem:jumpform},  $c(\ell)$ is the largest positive integer such 
that $\chi^{p^{n-(\ell+1)}}$ is non-trivial on $1 + \pi_L^{c(\ell)} O_L$.  
Therefore $\chi^{p^{n-(\ell+1)}}$ is trivial on $1 + \pi_L^{c(\ell)+1} O_L$.   Suppose $c(\ell)$
is even.  Then $\pi_L^{c(\ell)}$ is equal to $\pi_K^{c(\ell)/2}\cdot u$, where $\pi_K = \mathrm{Norm}_{L/K} \pi_L$ 
is a uniformizer  in $K$ and $u$ is a unit in $O_L^*$.  Since $O_L$ and $O_K$ have the same residue field $k$, we 
would then  have
\begin{equation}
\label{eq:easystuffpi}
1 + \pi_L^{c(\ell)} O_L = (1 + \pi_K ^{c(\ell)/2} O_K) \cdot (1 + \pi_L^{c(\ell)+1} O_L).
\end{equation}
It follows that  $\chi^{p^{n-(\ell+1)}}$ must be non-trivial on $1 + \pi_K ^{c(\ell)/2} O_K$
By Lemma \ref{lem:thetypes}, the restriction of $\chi$ to $K^*$ is trivial if $G$ is
dihedral, so we have a contradiction in this case.  Suppose now that $G$ is quaternionic or semi-dihedral.
By Lemma \ref{lem:thetypes}, the restriction of $\chi$ to $K^*$ then has order $2$. We have supposed
$0 \le \ell \le n-2$ if $G$ is quaternionic
or semi-dihedral, so $n - (\ell + 1) \ge 1$ and $\chi^{p^{n-(\ell+1)}}$ is an integral power of $\chi^p = \chi^2$.  
Thus $\chi^{p^{n-(\ell+1)}}$ is trivial on $K^*$ in this case, and we again have a contradition.
This shows that $c(\ell)$ must have been odd, and completes the proof.
\end{proof} 

\begin{prop}
\label{prop:evenqsemi}  Suppose $G$ is quaternionic or semi-dihedral.  
\begin{enumerate}
\item[i.] The integer $i_{n-1}$ is
even unless $G$ is the quaternion group of order $8$ and $G_4 = \{e\}$.  In this case $i_{n-1}$ is odd and the following is true:
\begin{enumerate}
\item[a.] The lower ramification filtration of $G$ is 
$G = G_0 = G_1 \ne G_2 = G_3 \ne G_4 = \{e\}$, where $G_2$ is the order $2$ center of $G$.
\item[b.] The Bertin obstruction of $\phi:G \to \mathrm{Aut}_k(k[[t]])$ does not vanish.
\end{enumerate}
\item[ii.]  Suppose $G$ is a generalized quaternion group and $i_{n-1}$ is
even.  Then the KGB obstruction (and hence the Bertin obstruction) associated to $\phi$ vanishes.
\item[iii.]  Suppose $G$ is semi-dihedral. The Bertin obstruction of $\phi$
does not vanish if $d_0+d_1+d_2$ is not divisible by $4$,
in the notation of 
Lemma~\ref{lem:di2}.
\end{enumerate}
\end{prop}

The proof is an argument by contradiction, and requires a series of lemmas.  
Before beginning this we note that this result gives a new proof of a result of Serre
\cite[\S 5]{serre2} and Fontaine \cite{JMF} concerning 
Artin representations which cannot be realized over $\mathbb{Q}$.

\begin{cor}
\label{cor:Serrecor} {\rm (Serre, Fontaine)} Suppose $G$ is a generalized quaternion group.
Then the Artin character $-a_\phi$ is not realizable over $\mathbb{Q}$ 
if and only if $G$ has order $8$ and the lower ramification filtration
filtration of $G$ is as in Proposition \ref{prop:evenqsemi}.i.a.  
\end{cor}

\begin{proof} If the Bertin obstruction vanishes, then $-a_\phi$
is the character of a  permutation representation by 
Proposition \ref{prop:special}   so it is realizable over $\mathbb{Q}$.  Otherwise,
$G$ must have order $8$ and must have the ramification filtration in part
(i.a) of Proposition \ref{prop:evenqsemi}.
Suppose now that $G$ is as in part (i.a) of Proposition \ref{prop:evenqsemi}.  Serre proved in 
\cite[\S 5]{serre2} that  $-a_\phi$ is not realizable
over $\mathbb{Q}$ by proving that the multiplicity of the
two-dimensional irreducible representation of $G$ in $-a_\phi$
is $5$.  (This also follows from Proposition \ref{prop:serrethm}.)
\end{proof}

\begin{lemma}
\label{lem:firststep}
Suppose $G$ is quaternionic or semi-dihedral. Then $i_{n-1}$ is odd if and only 
if $c(n-1)$ in (\ref{eq:cldef}) is even.  Suppose $i_{n-1}$ is odd, and let $d_0$
and $d_1$ be as in Lemma~\ref{lem:di2}.  Then
\begin{equation}
\label{eq:yippie}
\frac{c(n-1)}{2} = d_0 -1 \ (\mathrm{resp. } \ d_1-1 )\ \quad \mathrm{if} \quad G \quad \mathrm{is\ quaternionic \ (resp.\ \ semi-dihedral)}.
\end{equation}
\end{lemma}

\begin{proof}
The first statement is clear from (\ref{eq:cldef}) and Corollary \ref{cor:firstjump}.  
Now suppose $i_{n-1}$ is odd, so that $c(n-1)$ is even.  By Lemma \ref{lem:jumpform},
$c(n-1)$ is the largest positive integer such that $\chi|(1 + \pi_L^{c(n-1)}O_L)$
is non-trivial.  Now $c(n-1)$ is even, $\pi_K O_L = \pi_L^2 O_L$, and 
the residue fields of $L$ and $K$ are both $k$; so we have
$$(1 + \pi_L^{c(n-1)}O_L) = (1 + \pi_K^{c(n-1)/2}O_K)\cdot (1 + \pi_L^{c(n-1)+1}O_L).$$
This implies that $c(n-1)/2$ is the largest positive integer $j$ such that
$\chi|(1 + \pi_K^j O_K)$ is non-trivial.   By Lemma \ref{lem:thetypes},
the restriction $\chi|K^*$ is the non-trivial quadratic character associated
to the quadratic extension $L_i/K$, where $i = 0$ and $L_0 = L$ if $G$
is quaternionic, and $i =1$ and $L_1/K$ is described in Lemma \ref{lem:thetypes}
if $G$ is semi-dihedral.  The ramification groups $\mathrm{Gal}(L_i/K)^\nu = \mathrm{Gal}(L_i/K)_\nu$
equal $\mathrm{Gal}(L_i/K)$ if $\nu = 0,\ldots, d_i -1$ and  these groups are trivial
for $\nu > d_i  - 1$ since $[L_i:K] = 2$.  Thus $c(n-1)/2 = d_i -1$ as claimed.  
\end{proof}

\begin{lemma}
\label{lem:whatup}
Recall that $\sigma \in G = \mathrm{Gal}(N/K)$ is an element not in $H = \mathrm{Gal}(N/L)$.  One 
has 
\begin{equation}
\label{eq:sigy}
\sigma(\pi_L) = \pi_L \cdot (1 + \beta \pi_L^{d_0-1} + \pi_L^{d_0}\gamma)
\end{equation}
for some $\gamma \in O_L$ and some $\beta \in k^*$. We may
define a uniformizer $\pi_K$ in $K$ by 
\begin{equation}
\label{eq:xacts}
\pi_K = \pi_L \sigma(\pi_L) = \pi_L^2 \cdot (1 + \beta \pi_L^{d_0-1} + \pi_L^{d_0} \gamma)
\end{equation}
\end{lemma}

\begin{proof}The first and only jump in the lower
ramification numbering of $\mathrm{Gal}(L/K)$ occurs at $d_0 - 1$.   By the definition
of the lower numbering, this gives (\ref{eq:sigy}),
and (\ref{eq:xacts}) follows from the fact that $L/K$ is quadratic and totally
ramified.
\end{proof}

\begin{lemma}
\label{lem:zapper}
Let $G$ be quaternionic or semi-dihedral.  The integer $i_0 = c(0)$ is the largest integer for which there exists an
element $\delta \in k^*$ such that
$\chi(1 + \delta \pi_L^{i_0}) = \zeta$ is a primitive $2^{n}$-th root of unity.  The value of $i_0$ is odd and given by 
\begin{equation}
\label{eq:jequal}
i_0 = d_1 + d_2  - d_0 - 1.
\end{equation}
\end{lemma}

\begin{proof}  The first statement follows from \cite[Cor. 3, \S XV.2]{corps} since $\chi$ has order $2^n$.  In the quaternionic or semi-dihedral case, $p = 2$.  For
$\gamma \in L^*$, the value $\chi(\gamma)$ is a primitive $2^{n}$-th root of
unity if and only if $\chi^{p^{n-1}}(\gamma)$ is non-trivial.  Hence
Lemma \ref{lem:jumpform} shows the first statement about $i_0$,
since $i_0$ is the largest integer such that $\chi^{2^{n-1}}$ is not trivial on $1 + \pi_L^{i_0} O_L$.
 The character $\chi^{2^{n-1}}$ corresponds to the order two
 character of the Galois group $\mathrm{Gal}(L'/L)$, where $L = L_0$ and $L'$ is the compositum
 $L \cdot L_1 = L_1 \cdot L_2$ over $K$ (with notation as in Hypothesis~\ref{hyp:setuplocal}). Thus $i_0$ the first (and only) jump in the 
 upper (and lower) ramification filtration of $\mathrm{Gal}(L'/L)$.  It follows that
 the relative discriminant $d_{L'/L}$ equals $\pi_L^{i_0+1}O_L$. The relative discriminant $d_{L'/K}$
 is given by
 $$d_{L/K}^2 \cdot \mathrm{Norm}_{L/K} d_{L'/L} = d_{L'/K} = d_{L/K} \cdot d_{L_1/K} \cdot d_{L_2/K}.$$
 Since $L = L_0$ is totally and quadratically ramified over $K$, this gives 
 $$2 d_0 + i_0 + 1 = d_0 + d_1 + d_2$$
 which is equivalent to (\ref{eq:jequal}).  Since all of $d_0$, $d_1$ and $d_2$
 are even, $i_0$ is odd.
 \end{proof}

 \begin{lemma}
 \label{lem:identity}
 Suppose $\beta$, $\pi_L$ are as in Lemma \ref{lem:zapper} and that $\pi_K$ is a uniformizer in $K$.   For all $a \in k$, all odd integers
 $h \ge 1$, and all 
 sufficiently large positive integers $M > 1$, we have
\begin{equation}
\label{eq:zowie}
(1 + a \pi_L^h)^{2^M - 2} \cdot (1 + a^2 \pi_K^h) = 1 + a^2 \beta \pi_L^{2h + d_0 - 1} + \pi_L^{2h + d_0}\eta
\end{equation}
for some $\eta \in k[[{\pi_L}]] = O_L$.
\end{lemma}

\begin{proof} We will show by induction on the integer $\mu \ge 1$
that 
\begin{equation}
\label{eq:muinduct}
(1 + a \pi_L^h)^{2^\mu - 2} \cdot (1 + a^2 \pi_K^h) = 1 + a^{2^\mu} \pi_L^{h2^\mu} +  a^2 \beta \pi_L^{2h + d_0 - 1} + 
\pi_L^{2h + d_0} \cdot \eta_\mu 
\end{equation}
for some $\eta_\mu \in k[[{\pi_L}]] = O_L$, using that the characteristic is $2$. We then get (\ref{eq:zowie}) by 
setting $M = \mu$ and 
$$\eta = a^{2^\mu} \pi_L^{h 2^\mu - (2h + d_0)} + \eta_\mu$$
when $\mu$ is large enough so that $h 2^\mu \ge 2h + d_0$.

To prove (\ref{eq:muinduct}), first consider the case $\mu = 1$.
We have from (\ref{eq:xacts}) that since $h$ is odd and $d_0 \ge 2$, 
\begin{eqnarray}
1 + a^2 \pi_K^h &=& 1 + a^2 ({\pi_L} \sigma({\pi_L}))^h \nonumber \\
&=& 1 + a^2 (\pi_L^2 \cdot (1 + \beta \pi_L^{d_0-1} + \pi_L^{d_0} \gamma))^h\\
&=& 1 + a^2 \pi_L^{2h} + a^2 \beta \pi_L^{2h + d_0-1} + \pi_L^{2h + d_0} \eta_1\nonumber
\end{eqnarray}
for some $\eta_1 \in k[[{\pi_L}]] = O_L$.  This is exactly
the assertion in (\ref{eq:muinduct}) when $\mu = 1$.

Now assume that (\ref{eq:muinduct}) holds for some $\mu \ge 1$.
We multiply both sides by 
\hbox{$(1 + a \pi_L^h)^{2^\mu} = 1 + a^{2^\mu} \pi_L^{h 2^\mu}$}.
The left side becomes
\begin{equation}
\label{eq:leftie}
(1 + a \pi_L^h)^{2^\mu} \cdot (1 + a \pi_L^h)^{2^\mu - 2} \cdot (1 + a^2 \pi_K^h) 
= (1 + a \pi_L^h)^{2^{\mu+1} - 2} \cdot (1 + a^2 \pi_K^h) 
\end{equation}
The right hand side becomes
\begin{eqnarray}
\label{eq:rightie}
&\left (1 + a^{2^\mu} \pi_L^{h 2^\mu}\right )
\left ( 1 + a^{2^\mu} \pi_L^{h2^\mu} +  a^2 \beta \pi_L^{2h + d_0 - 1} +  
\pi_L^{2h + d_0} \cdot \eta_\mu  \right )\nonumber \\
& = 1 + a^{2^{\mu+1}} \pi_L^{h2^{\mu+1}}  +   a^2 \beta \pi_L^{2h + d_0 - 1} 
 + \pi_L^{2h + d_0} \cdot \eta_{\mu + 1}
\end{eqnarray}
where 
$$\eta_{\mu + 1} = (1 + a^{2^\mu} \pi_L^{h 2^\mu}) \eta_\mu + a^{2^\mu + 2} \pi_L^{h2^\mu-1} \beta  \in k[[{\pi_L}]] = O_L.$$
Equating the right hand sides of (\ref{eq:leftie}) and (\ref{eq:rightie}) shows (\ref{eq:muinduct}) when $\mu$ is replaced by $\mu + 1$, 
so the induction is complete.  
\end{proof}

\begin{cor}
\label{cor:almost}
With the notations of Lemma \ref{lem:identity}, let $a = \delta$, $h = i_0$ and 
\begin{equation}
\label{eq:duh}
z = 1 + \delta^2 \beta \pi_L^{2i_0 + d_0 - 1} + \pi_L^{2i_0 + d_0}\eta
\end{equation}
Then $\chi(z)$ is a root of unity of order exactly $2^{n - 1}$
unless $n = 2$, $G$ is a quaternion group of order $8$ and
$\chi(1 + \delta^2 \pi_K^{i_0}) = -1$;  in this case, $\chi(z) = 1$.  
\end{cor}

\begin{proof}
{From} (\ref{eq:zowie}), we have
\begin{equation}
\label{eq:zero}
\chi(z) = \chi(1 + \delta \pi_L^{i_0})^{2^M - 2} \cdot \chi(1 + \delta^2 \pi_K^{i_0})
\end{equation}
Now from Lemma \ref{lem:zapper}, $\chi(1 + \delta \pi_L^{i_0}) = \zeta$ is
a primitive root of unity of order $2^{n}$, while
$\chi(1 + \delta^2 \pi_K^{i_0}) = \pm 1$ by
Lemma \ref{lem:thetypes} since $1 + \delta^2 \pi_K^{i_0} \in K^*$.  Thus
$\chi(1 + \delta \pi_L^{i_0})^{2^M - 2} = \zeta^{2^M - 2}$
is a root of unity of order $2^{n-1}$ since $M > 1$.
If $n \ge 3$, the product of a root of unity of order $
2^{n-1}$ with $\pm 1$ is a root of unity of order $2^{n-1}$,
so (\ref{eq:zero})
shows $\chi(z)$ has order $2^{n-1}$.  Since
$G$ is quaternionic or semi-dihedral of order $2^{n+1}$, the only 
way in which one can have $n < 3$ is for $n = 2$ 
and for $G$ to be a quaternion group of order $8$.
In this case, $\chi(1 + \delta \pi_L^{i_0})^{2^M - 2} = -1$, so $\chi(z) = - \chi(1 + \delta^2 \pi_K^{i_0})$
is a root of unity of order $2^{n-1} = 2$ if and only if $\chi(1 + \delta^2 \pi_K^{i_0}) = 1$.\end{proof}

\begin{lemma}
\label{lem:evenqsemi2}  Suppose $G$ is quaternionic or semi-dihedral.  Then $i_{n-1}$ is
even unless all of  the following hypotheses hold:
\begin{enumerate}
\item[i.] $n = 2$ and $G$ is the quaternion group of order $8$;
\item[ii.]  The constants $d_0, d_1$ and $d_2$ are all equal to $i_0        +1$
in the notation of  Lemma \ref{lem:zapper};
\item[iii.]  The largest integer $c(1) = c(n-1)$ such that $\chi$ is 
non-trivial on $1 + \pi_L^{c(1)} O_L$ is $c(1) = 2i_0       $.
\end{enumerate}
\end{lemma}

\begin{proof}We assume throughout the proof that $i_{n-1}$ is odd, so that 
$c(n-1)$ is even by  Lemma \ref{lem:firststep}.  Suppose first that $n > 2$.
By  Corollary \ref{cor:almost}, $\chi(z)$ is a root of unity
of order exactly $2^{n - 1}$.  Then since $L$ has characteristic $2$, 
\begin{eqnarray}
\label{eq:zpower}
\chi(z^{2^{n-2}}) 
&=& \chi(1 + \delta^{2^{n-1}} \beta^{2^{n-2}} \pi_L^{(2i_0         + d_0 -1)2^{n-2}} + \pi_L^{(2i_0        + d_0)2^{n-2}}\eta^{2^{n-2}})
\end{eqnarray}
is a primitive root of unity of order $2$, so it is equal to $-1$.  Since
$$\delta^{2^{n-1}} \beta^{2^{n-2}}  \ne 0$$
in $k$, this shows
that
\begin{equation}
\label{eq:finish}
c(n-1) \ge (2i_0 + d_0 -1) 2^{n-2}. 
\end{equation}
On the other hand, Lemma \ref{lem:firststep} shows 
\begin{equation}
\label{eq:money}
d_i -1 = \frac{c(n-1)}{2},
\end{equation} 
where $i = 0$ if $G$ is quaternionic and $i = 1$ if $G$ is semi-dihedral.
We conclude from (\ref{eq:money}) and (\ref{eq:finish})  that 
\begin{equation}
\label{eq:tada2}
(2i_0        + d_0 -1) 2^{n-2}   = 2^{n-2} (d_0  - 1) + 2^{n-1}i_0         \le c(n-1) = 2 (d_i - 1).
\end{equation}
Now $i_0        \ge 1$ and $d_0 - 1 \ge 1$ because $d_0$ is an even positive integer.  Since
we have assumed $n \ge 3$ in deducing (\ref{eq:tada2}), we conclude
from (\ref{eq:tada2}) that $d_i > d_0$. Then $i = 1$ and $G$ must be semi-dihedral.  In this
case, $d_1 = d_2 > d_0$, since $\pi_K^{d_i} O_K$ is the relative discriminant
of the quadratic extension $L_i/K$, and the compositum of $L_0$, $L_1$
and $L_2$ is the biquadratic extension $L'/K$.  Here 
\begin{equation}
\label{eq:oof}
d_2  \ge d_0 + 2
\end{equation} since $d_2 = d_1 > d_0$ and each of $d_0, d_1$ and
$d_2$ are even. 
By Lemma \ref{lem:zapper}, 
$\chi(1 + \delta \pi_L^{i_0})$ is a root of unity of order $2^n$ for  some $\delta \in k^*$
and $i_0        = d_1 + d_2 - d_0 - 1 $ and some $\delta \in k^*$.    Hence (\ref{eq:oof}) gives
\begin{equation}
\label{eq:jlower}
i_0        = d_1 + d_2 - d_0 - 1\ge d_1 + 1 
\end{equation} 
Thus (\ref{eq:money}), (\ref{eq:tada2}) and (\ref{eq:jlower}) give 
\begin{equation}
\label{eq:tadada}
2^{n-1} (d_1 + 1)\le 2^{n-1} i_0        < 2^{n-2} (d_0  - 1) + 2^{n-1}i_0         \le c(n-1) = 2(d_1 - 1).
\end{equation}
Since $n \ge 3$, this would imply $d_1 < 0$,
which is impossible.  Thus $d_i > d_0$ is impossible, and we conclude 
our assumption that $n \ge 3$ is also impossible.  

What we have shown thus far is that if $i_{n-1}$ is odd
then $n < 3$.  So $n = 2$ and 
$G$ must be the quaternion group of order $8$, which
we will assume for the rest of the proof.  In view of 
Lemma \ref{lem:firststep} we have $c(n-1) = c(1) = d_0 - 1$.
Suppose that the disciminant exponents $d_0, d_1$ and $d_2$
are not all equal.  Since $G$ is the quaternion group of order $8$,
we can switch the roles of $L_0$, $L_1$ and $L_2$ to be
able to assume that $d_0 < d_1 = d_2$.  We now have the lower bound
\begin{equation}
\label{eq:jlowernew}
i_0        = d_1 + d_2 - d_0 - 1 = 2 d_1 - d_0 - 1 \ge 2(d_0 + 2) - d_0 - 1 = d_0 + 3
\end{equation} provided that $i_{n-1}$ is
odd.  Since $\chi(1 + \delta \pi_L^{i_0})$ is a root of unity of order $2^n = 4$,
 $\chi(1 + \delta^2 \pi_L^{2i_0       }) = \chi(1 + \delta \pi_L)^2  =-1$, so 
$\chi$ is non-trivial on $1 + \pi^{2(d_0 + 3)} O_L$ by (\ref{eq:jlowernew}).  This implies 
$2(d_0 + 3) \le c(1) = 2(d_0 - 1)$ which is impossible. 
Hence all of $d_0, d_1$ and $d_2$ must  be equal, and we find
from Lemma \ref{lem:zapper} that they equal $i_0        + 1$. This 
and (\ref{eq:money}) (in which $i = 0$, since $G$ is a quaternion group)
complete the proof.
\end{proof}

\begin{lemma}
\label{lem:boundquat}
Suppose $G$ is quaternionic.   Then $i_{n-1} \ge d_0$ unless
conditions (i), (ii) and (iii) of Lemma \ref{lem:evenqsemi2}  hold.
\end{lemma}

\begin{proof}
By Lemma \ref{lem:jumpform},
\begin{equation}
\label{eq:okeyi}
i_{n-1} = c(n-1) - c(n-2)
\end{equation}
where $j = c(n-1)$ (resp.\ $j = c(n-2)$) is the largest positive
integer such that $\chi$ (resp.\ $\chi^2$) is non-trivial on
$1 + \pi_L^j O_L$.  Thus $c(n-2)$ is the largest positive 
integer such that there is a constant $a \in k^*$ such that
$\chi(1 + a \pi_L^{c(n-2)}) = \zeta$ is a primitive fourth root of
unity.  It follows from Corollary \ref{cor:firstjump} and (\ref{eq:cldef}) that $c(n-2)$
is odd, so we can let $h = c(n-2)$ in Lemma \ref{lem:identity}.
With the notations of Lemma \ref{lem:identity},
\begin{equation}
\label{eq:zowie2}
(1 + a \pi_L^{c(n-2)})^{2^M - 2} \cdot (1 + a^2 \pi_K^{c(n-2)}) = 1 + a^2 \beta \pi_L^{2 c(n-2) + d_0 - 1} + \pi_L^{2 c(n-2) + d_0}\eta
\end{equation}
for some $\eta \in k[[{\pi_L}]] = O_L$.  Since $M$ is very large,
$\chi(1 + a \pi_L^{c(n-2)})^{2^M - 2} = \zeta^{2^M -2} = -1$
because $\zeta$ is a root of unity of order $4$.  Hence
(\ref{eq:zowie2}) shows 
$$\chi(1 + a^2 \beta \pi_L^{2 c(n-2) + d_0 - 1} + \pi_L^{2 c(n-2) + d_0}\eta) \ne 1 \quad \mathrm{if} \quad \chi(1 + a^2 \pi_K^{c(n-2)}) = 1.$$
This shows that if $\chi(1 + a^2 \pi_K^{c(n-2)}) = 1$ then
$\chi$ is non-trivial on $1 + \pi_L^{2 c(n-2) + d_0 -1}O_L$, so
(\ref{eq:okeyi}) gives
$$i_{n-1} = c(n-1) - c(n-2) \ge 2c(n-2) + d_0 - 1 - c(n-2) \ge d_0$$
as required.  

We now must consider the case in which
$\chi(1 + a^2 \pi_K^{c(n-2)}) \ne 1$.  For quaternionic $G$,
$\chi|_{K^*}$ is the character associated to $L/K$. Hence 
if $\chi(1 + a^2 \pi_K^{c(n-2)}) \ne 1$ then $c(n-2) \le d_0 -1$
since the first jump in the ramification filtration of $\mathrm{Gal}(L/K)$
occurs at $d_0 -1$.  However, Lemmas \ref{lem:jumpform} and
\ref{lem:zapper} show that
\begin{equation}
\label{eq:strict}
c(n-2) = i_0 + i_1 + \cdots + i_{n-2} \ge i_0 = d_1 + d_2  - d_0 - 1 
\end{equation}
where all of the $i_j$ are positive.  Since $d_0, d_1$ and $d_2$
are the exponents of the discriminants of quadratic subextensions
of a Klein four extension of $K$, either all of these integers
are the same or two of them are equal and larger than the third.
Hence we see from (\ref{eq:strict}) that $c(n-2) \ge d_0 -1$,
with strict inequality unless $n = 2$, $d_0 = d_1 = d_2$ and 
$c(n-2) = i_0 = d_0 - 1$.  Suppose now that all of these conditions
hold.  Then $\chi(1 + \pi_L^{c(n-2)}O_L) = \chi(1 + \pi_L^{i_0}O_L)$ contains a primitive fourth root of unity, so
$\chi(1 + \pi_L^{2c(n-2)}O_L) \ne \{1\}$.  If follows
that $c(1) = c(n-1) \ge 2c(n-2) = 2 i_0$, and if $c(1) > 2i_0$
then  (\ref{eq:okeyi}) implies
$$i_{n-1} = c(n-1) - c(n-2) \ge 2i_0 + 1 - i_0 = i_0 + 1 = d_0.$$
Thus the only way in which we could have $i_{n-1} < d_0$
is for $c(1) = 2 i_0$, which shows that all of the conditions
of Lemma \ref{lem:evenqsemi2} hold.
\end{proof}

\begin{lemma}
\label{lem:quatcase}  Suppose hypotheses (i), (ii) and (iii) of Lemma
\ref{lem:evenqsemi2} hold and that $i_0 > 1$.  There is an inclusion of multiplicative groups
\begin{equation}
\label{eq:sortitout}
1 + \pi_L^{2i_0       } O_L \subset ( \mathrm{Norm}_{L/K} L^*) \cdot (L^*)^4 \cdot(1 + \pi_L^{2i_0 +1} O_L)
\end{equation}
\end{lemma}

\begin{proof}  By Lemma \ref{lem:di2}, $d_0$ is even.  We have
$$i_0 = d_1 + d_2 - d_0 - 1 = d_0 - 1$$
by Lemma \ref{lem:zapper} and hypothesis (ii) of Lemma \ref{lem:evenqsemi2}.
By \cite[Lemme 5.1.1]{MezardThesis} and the paragraph following that lemma, we can choose the uniformizer $\pi_L$ of $L$ such
that 
\begin{equation}
\label{eq:sigmaact}
\sigma(\pi_L) = \frac{\pi_L}{(1 + \pi_L^{i_0})^{1/i_0       }}.
\end{equation}
where $\sigma \in G$ projects to the non-trivial element of $\mathrm{Gal}(L/K)$.

For $i \ge 1$ the binomial theorem for fractional exponents now gives
\begin{equation}
\label{eq:urp2}
\sigma(\pi_L^i) \equiv \pi_L^i \ (\mathrm{resp.}\ \pi_L^i (1 + \pi_L^{i_0})\,)\ \mathrm{mod} \ \pi_L^{2i_0   +1} O_L \quad 
\mathrm{if}\quad 2\, |\, i \ (\mathrm{resp.} \ 2 {\not |}\, i\, ).
\end{equation}

Suppose now that $\zeta \in k$.  
Define
\begin{equation}
\label{eq:hdef}
h(\zeta) =  (1 + \pi_L + \zeta^2 \pi_L^{i_0   -1}) \cdot \sigma(1 + \pi_L + \zeta^2 \pi_L^{i_0   -1}) = \mathrm{Norm}_{L/K}(1 + \pi_L + \zeta^2 \pi_L^{i_0   -1}). 
\end{equation}
Using  $i_0    > 1$  and the fact that $i_0    - 1$ is even, we have from (\ref{eq:urp2}) the following congruences mod $\pi_L^{2i_0   +1} O_L$:
\begin{eqnarray}
\label{eq:lineup}
h(\zeta)  &\equiv&  (1 + \pi_L + \zeta^2 \pi_L^{i_0   -1}) (1 + \pi_L + \pi_L^{i_0   +1} + \zeta^2 \pi_L^{i_0   -1})\nonumber \\
&\equiv & 1 + \pi_L^2 + \zeta^4 \pi_L^{2(i_0   -1)} + \pi_L^{i_0   +1} + \pi_L^{i_0   +2} + \zeta^2  \pi_L^{2i_0   }\\
&\equiv& h(0) + \zeta^4 \pi_L^{2(i_0   -1)} + \zeta^2 \pi_L^{2i_0   }\nonumber 
\end{eqnarray}
Here 
\begin{equation}
\label{eq:inversecong}
h(0)^{-1} = \left ( 1 + \pi_L^2 +  \pi_L^{i_0   +1} + \pi_L^{i_0   +2} \right )^{-1} \equiv 1 - \pi_L^2 \quad \mathrm{mod} \quad 
\pi_L^3 O_L.
\end{equation}
since $i_0    \ge 3$.  Thus (\ref{eq:lineup}) gives congruences
\begin{eqnarray}
\label{eq:whatthef}
h(0)^{-1} h(\zeta) &\equiv& 1 + h(0)^{-1} (\zeta^4 \pi_L^{2(i_0   -1)} + \zeta^2 \pi_L^{2i_0   })\quad \mathrm{mod} \quad 
\pi_L^{2i_0   +1} O_L\nonumber\\
&\equiv& 1 +  (1 - \pi_L^2) (\zeta^4 \pi_L^{2(i_0   -1)} + \zeta^2 \pi_L^{2i_0   })\quad \mathrm{mod} \quad 
\pi_L^{2i_0   +1} O_L\\
&\equiv& (1 + \zeta^4 \pi_L^{2(i_0   -1)}) \cdot (1 + (\zeta^2 - \zeta^4) \pi_L^{2i_0   })\quad \mathrm{mod} \quad 
\pi_L^{2i_0   +1} O_L\nonumber
\end{eqnarray}
 where the last congruence holds because $$2(i_0   -1) + 2i_0    = 4i_0    - 2 \ge 2i_0    + 1$$
since $i_0    \ge 3$.  Because $i_0     \ge 3$ is odd, $(i_0   -1)/2 \ge 1$ is an integer.  Hence (\ref{eq:whatthef}) gives
\begin{equation}
\label{eq:zipup}
(1 + \zeta \pi_L^{(i_0   -1)/2})^{-4} \cdot h(0)^{-1} \cdot h(\zeta) \equiv 1 + (\zeta^2 - \zeta^4) \pi_L^{2i_0   } \quad \mathrm{mod} \quad \pi_L^{2i_0    + 1} O_L
\end{equation} 
Now for each $\lambda \in k$, there is a $\zeta \in K$ such that $\zeta^2 - \zeta^4 = \lambda$.   By (\ref{eq:hdef}), $h(0)^{-1} \in \mathrm{Norm}_{L/K}(L^*)$.  We
conclude from (\ref{eq:zipup}) that 
$$1 + \pi_L^{2i_0   }O_L \subset (L^*)^4 \cdot \mathrm{Norm}_{L/K}(L^*) \cdot (1 + \pi_L^{2i_0   +1} O_L)$$
which proves Lemma \ref{lem:quatcase}.  
\end{proof}

\begin{cor}
\label{cor:quatcasecor}  Hypotheses (i), (ii) and (iii) of Lemma
\ref{lem:evenqsemi2} imply $i_0 = 1$.  
\end{cor} 

\begin{proof}Suppose $i_0 > 1$.  
By part (iii) of Lemma \ref{lem:evenqsemi2},
$\chi$ is not trivial on $1 + \pi^{2i_0   }O_L$  but trivial on $1 + \pi^{2i_0   +1}O_L$.  Since
$G$ is the quaternion group of order $8$, the character $\chi$ has order $2^n = 4$.  By 
Lemma \ref{lem:thetypes}, $\chi$ is trivial on $\mathrm{Norm}_{L/K} (L^*)$.  Hence
$\chi$ is trivial on $( \mathrm{Norm}_{L/K} L^*) \cdot (L^*)^4 \cdot(1 + \pi_L^{2i_0   +1} O_L)$
so $\chi$ is trivial on $1 + \pi^{2i_0   }O_L$ by  Lemma \ref{lem:quatcase}, which is a contradiction.  This proves that we must have
$i_0    = 1$.   
\end{proof}

\medbreak
\noindent {\bf Completion of the proof of Proposition \ref{prop:evenqsemi}.}
\medbreak
We begin with statement (i) of the proposition.  

Suppose first that $i_{n-1}$ is odd.  By Lemma \ref{lem:evenqsemi2}, we can reduce the case in which $G$ satisfies the
hypotheses of Lemma \ref{lem:evenqsemi2}.   By Corollary \ref{cor:quatcasecor},  $i_0    = 1$.  
For $i = 0, 1, 2$, let $H_i = \mathrm{Gal}(N/L_i)$ where $L_i$ is the
quadratic extension of $K$ described in Lemma~\ref{lem:di2}, so that $H = H_0$.  The first 
(and only) jump in the upper (and lower) ramification filtration on $G/H_i$ occurs at $d_i - 1$.  
By Lemma \ref{lem:evenqsemi2}, $d_0 = d_1 = d_2$.  Therefore if $\nu \in \mathbb{R}$
is a jump in the ramification filtration of $G/H_i$ for one $i$, it is a jump
in this filtration for all $i$. 
The image of the higher ramification group $G^\nu$
in $G/H_i$  is equal to the ramification group
$(G/H_i)^\nu$.   It follows that $G^\nu \cap H_i$ has order independent of $i \in \{0,1,2\}$.
Hence $G^\nu$ is either $G$, $\{e\}$ or the center $\C(G) = \{e,\tau^2\}$ of $G$.  By Lemma \ref{lem:evenqsemi2}
and the definition of $i_0   $ in Lemma  \ref{lem:zapper}, the jumps in the ramification
filtration of $H_0 = H = \mathrm{Gal}(N/L)$ occur at the integers $i_0    = 1$ and at
$2i_0    = 2$.  Hence by the Hasse-Arf Theorem (see Lemma \ref{lem:jumps}), the jumps in the lower numbering
of the ramification filtration of $H$ occur at $1$ and at $1 + 2 = 3$.  Since
each ramification group is either $G$, $\{e\}$ or $\C(G)$, we conclude
that the lower numbering of the ramification filtration of $G$ is
\begin{equation}
\label{eq:ohby}
G = G_0 = G_1 \supset \C(G) = G_2 = G_3 \supset G_4 = \{e\}.
\end{equation}

Suppose now that $G$ is the quaternion group of order $8$
and  $G_4 = \{e\}$.  
From $H_4 = \{e\}$ and Lemma \ref{lem:jumps}
we must have $H = H_0 = H_1 \supset H_2 = 2H = H_3 \supset H_4 = \{e\}$.
Since this holds true for each of the cyclic subgroups $H$ of index $2$ in
$G$, the lower ramification filtration of $G$ must be given by (\ref{eq:ohby}). 
Since $\# G = 8$, we have $n = 2$.  Lemmas \ref{lem:jumpform} and
\ref{lem:jumps} that $i_0 = i_1 = 1$, so that $i_{n-1} = i_1$ is odd.

We now check that if $G$ has order $8$ and ramification filtration
(\ref{eq:ohby}) then the Bertin obstruction does not vanish.  Let $T = \C(G)$.  The set $S(T)$ of cyclic subgroups
of $G$ which contain $T$ is $\{T,H_0,H_1,H_2\}$.  The constant
$b_T$ appearing in Theorem \ref{thm:nontrivcase}
is thus
\begin{eqnarray}
\label{eq:theformlong}
b_T &=& \frac{1}{[\N_G(T):T]} \left( -\delta(T,\{e\}) a_\phi(1) + \sum_{\Gamma \in S(T)} \mu([\Gamma:T]) \iota(\Gamma) \right ) 
= -\frac{1}{2}.
\end{eqnarray}
Proposition \ref{prop:special}   now shows that the Bertin obstruction associated
to the given action of $G$ on $N$ does not vanish.  This completes
the proof of part (i) of Proposition \ref{prop:evenqsemi}.

We now suppose that as in part (ii) of Proposition \ref{prop:evenqsemi},
$G$ is a generalized quaternion group and $i_{n-1}$ is even.  We claim that
not all of hypotheses (i), (ii) and (iii) Lemma \ref{lem:evenqsemi2} can hold.  Suppose
to the contrary that all of these hypotheses do hold.  Thus $n = 2$, $c(1) = 2i_0$, 
and by  Corollary \ref{cor:quatcasecor},   $i_0 = 1$.   Thus Lemma \ref{lem:jumpform}
gives $i_{n-1} = i_1 = c(1) - c(0) = c(1) - i_0 = i_0 = 1$.  This contradicts our assumption that
$i_{n-1}$ is even, so not all of Hypotheses (i),(ii)
and (iii) of Lemma \ref{lem:evenqsemi2} hold.  Therefore Lemma \ref{lem:boundquat} proves $i_{n-1} \ge d_0$. Hence  Corollaries
\ref{cor:closer}(c), \ref{cor:closertwo} and \ref{cor:closerthree} show that the KGB obstruction vanishes.

To prove the final statement (iii) in Proposition \ref{prop:evenqsemi},
we know by Lemma \ref{lem:evenqsemi2}  that $i_{n-1}$
is even if $G$ is semi-dihedral.  Since the $d_i$ are all even, we conclude from Corollary \ref{cor:closer}(d) that the Bertin
obstruction does not vanish if 
\begin{equation}
\label{eq:dumbcon}
\quad d_0 + d_1 + d_2 \equiv d_0 + d_1 - d_2 \not \equiv 0 \quad \mathrm{mod} \quad 4\mathbb{Z}. \quad \hbox{\ \ \ \ \ \ \ \ \ \ \ \ \ \ \ \ \ \ \ \ \ \ \ \ \ \ \ \ } \square 
\end{equation}


\section{The group $\mathrm{SL}_2(3)$ when $p = 2$.}
\label{s:sl23}
\setcounter{equation}{0}

\begin{prop}
 \label{prop:qual1}
 Suppose $p = 2$ and that $G$ is isomorphic to
 $\mathrm{SL}_2(3)$.     A $2$-Sylow subgroup $P$ of $G$ is normal and isomorphic to a quaternion
 group of order $8$.  The Bertin obstruction associated to an injection $\phi:G \to \mathrm{Aut}_k(k[[t]])$
 vanishes if and only if the Bertin obstruction associated to the restriction $\phi_P$ of $\phi=\phi_G$ from $G$ to $P$ vanishes.  These two
 equivalent conditions hold if and only if the KGB obstructions of both $\phi$ and $\phi_P$ vanish.
 \end{prop}
 
 \begin{proof} Because of Theorem \ref{thm:reducetop}, the Bertin obstruction of $\phi_P$ 
 vanishes if that of $\phi$ does, and we now prove the converse.
 We suppose for the rest of the proof that Bertin obstruction of $\phi_P$ vanishes.  To show
 that the Bertin obstruction of $\phi_G$ vanishes, it will
 be enough to show that conditions (b), (c) and (d) of Theorem \ref{thm:reducetop} hold.


Let $t$ be a non-trivial element of the cyclic group $C$ of order $3$.  Then $\C_G(t) = \C_P(t) \times C$ where $\C_P(t)$
is the center $\C(G)$ of $G$, which has order $2$.  This shows condition (b) of
Theorem \ref{thm:reducetop}.

The cyclic non-trivial subgroups $T$ of $P$ are $\C(G)$ together with the three
cyclic subgroups $\Gamma(1)$, $\Gamma(2)$ and $\Gamma(3)$ of order $4$ which
are conjugate under the action of $C$.  There are four conjugates $C(1) = C$, $C(2)$, $C(3)$
and $C(4)$ of $C$ in $G$.  Let $J(j) = \langle \C(G),C(j)\rangle$ be the cyclic group of
order $6$ generated by $C(j)$ and $\C(G)$.  One has
\begin{equation}
\label{eq:layout}
S_G(\C(G)) = \{\C(G),\Gamma(1),\Gamma(2),\Gamma(3),J(1),J(2),J(3),J(4)\}
\end{equation}
and
\begin{equation}
\label{eq:layout2}
S_G(\Gamma(j)) = \{\Gamma(j)\}.
\end{equation}

In the notation of Theorem \ref{thm:reducetop}, we have
\begin{equation}
\label{eq:oyvey}
b'_{\C(G),G} = \sum_{P \not \supset \Gamma \in S_G(\C(G))} \mu([\Gamma:\C(G)]) = \sum_{j =1}^4 \mu([J(j):\C(G)]) = -4.
\end{equation}
Since every $\Gamma \in S_G(\Gamma(j))$ is contained in $P$, we have
\begin{equation}
\label{eq:oyoyoy}
b'_{\Gamma(j),G} = \sum_{P \not \supset \Gamma \in S_G(\Gamma(j))} \mu([\Gamma:\Gamma(j)]) = 0
\end{equation}
Condition (c.i) of Theorem \ref{thm:reducetop} is that $b'_{T,G} \equiv 0 $ mod $[\N_P(T):T] \mathbb{Z}$, which we
see follows from (\ref{eq:oyvey}) and (\ref{eq:oyoyoy}) since $[\N_P(\C(P)):P] = 4$.

Since we have supposed that the Bertin obstruction of $\phi_P$ vanishes, we have $b_{T,P} \ge 0$
for $T = \C(G)$ and $T = \Gamma(j)$.  Condition (c.ii) of Theorem \ref{thm:reducetop}
 is that
\begin{equation}
\label{eq:youch}
[\N_G(P):T] b_{T,P} = \sum_{\Gamma \in S_P(T)} \mu([\Gamma:T]) \iota(T) \ge -b_{T,G}.
\end{equation}
When $T = \Gamma(j)$, this follows from $b_{\Gamma(j),P} \ge 0 = - b_{\Gamma(j),G}$.
We now assume that $T = \C(G)$, so that $-b_{T,G} = -b_{\C(G),G} = 4$. 
It remains to show prove the inequality (\ref{eq:youch}) in this case.

Let $H= \Gamma(1)$ be one of the three cyclic subgroups of order $4$
in $P$.  Let $i_0$ and $i_1$ be the integers associated to $\phi_P$ and 
to $H$ in Lemma \ref{lem:jumps}.  
  Let $\chi:H \to \mu_4$ be a faithful character of $H$.  We let $L = N^H$
be the fixed field of $H$ acting on $N$, where $N/F$ is the $G$ extension associated
to $\phi_G$.  By classfield theory, we can view $\chi$ as a character of $L^*$, after
reducing to the case of quasi-finite residue fields via Proposition \ref{prop:reducetoquasi}. 
By Lemma \ref{lem:jumpform}, $i_0$ is the largest integer such that $\chi(1 + \pi_L^{i_0} O_L) = \mu_4$, while $i_1$ is the largest integer
such that $\chi(1 + \pi_L^{i_0 + i_1} O_L) = \{\pm 1\}$.  Since $(1 + \pi_L^{i_0} O_L)^2 \subset
1 + \pi_L^{2 i_0} O_L$ we conclude that $i_0 + i_1 \ge 2i_0$, so $i_1 \ge i_0$.  
By Corollary \ref{cor:closer}, $i_0$ is odd and $i_1$ is even since the Bertin
obsrtruction to $\phi_P$ is trivial. Hence $i_1 \ge i_0$ implies $i_1 = i_0 + 1 + 2h$ for some $h \ge 0$. 
By the definition of $i_0$ and $i_1$, the jumps in the upper ramification filtration
of $H$ occur at $i_0$ and $i_0 + i_1$.  By Herbrand's theorem \cite[Chap.~IV.3, Lemma~5]{corps}, the jumps in the 
lower ramification filtration occur at $i_0$ and $i_0 + 2i_1 = i_0 + 2(i_0 + 1 + 2h) = 3i_0 + 2+4h$.  Now $\C(G)$
is the order $2$ subgroup of $H$, so the jumps in the lower and the upper
ramification filtation of $\C(G)$ both occur at $3i_0 + 2+4h$.  Recall that $J(1)$
is a cyclic group of order $6$ which contains $\C(G)$ (see (\ref{eq:layout})).  By the Hasse-Arf Theorem,
the jumps
in the upper ramification of $J(1)$ occur at integers $j_0 = 0$ and $j_1 \ge 0$ since
$J(1)$ is abelian and the wild ramification subgroup of $J(1)$ is $\C(G)$. Herbrand's theorem now
shows that the jumps in the lower ramification of $J(1)$ occur at $0$ and at
$3j_1$.  Therefore the (unique) jump in the lower ramification of
$\C(G)$ occurs at $3j_1 = 3i_0 + 2 + 4h$.  This and $h \ge 0$ force $h = 1 + 3h'$ for some $0 \le h' \in \mathbb{Z}$.
Thus
\begin{equation}
\label{eq:hbound}
i_1 = i_0 + 1 + 2h = i_0 + 1 + 2(1+3h') = i_0 + 3 + 6h'\quad \mathrm{with}\quad 0 \le h' \in \mathbb{Z}.
\end{equation}

Since $\Gamma(1)$ is
conjugate to $\Gamma_j$ for $j = 2, 3$, we have $\iota(\Gamma(1)) = \iota(\Gamma(j))$
for $j \in \{1,2,3\}$.  Considering that the jumps in the lower numbering of $H = \Gamma(1)$
occur at $i_0$ and $i_0 + 2i_1$, we conclude that 
\begin{equation}
\label{eq:putit}
\iota(\C(G)) = i_0 + 2i_1 + 1 = i_0 + 2(i_0 + 3 + 6h') + 1 = 3i_0 + 7 + 12h' 
\end{equation}
and
\begin{equation}
\label{eq:putit2}
 \iota(\Gamma(1))  = i_0 + 1.
\end{equation}
Now
\begin{eqnarray}
\label{eq:fin}
[\N_P(\C(G)):\C(G)] b_{\C(G),P} &=& \sum_{\C(G) \subset \Gamma \in S_P(\C(G))} \mu([\Gamma:\C(G)]) \iota(\Gamma)\nonumber\\
& = &
\iota(\C(G)) - 3\iota(\Gamma(1))\nonumber\\
& = &3i_0 + 7 + 12h' - 3(i_0 + 1)\nonumber \\
&=& 4 + 12h' \ge 4 = -b'_{\C(G),G}
\end{eqnarray}
because of (\ref{eq:putit}), (\ref{eq:putit2}) and (\ref{eq:oyvey}). This proves (\ref{eq:youch}) for $T  = \C(G)$, which completes the proof of condition (c.ii) of Theorem \ref{thm:reducetop}.

We finally consider condition (d) of Theorem \ref{thm:reducetop}.  If $T = \C(G)$ then 
$\C_{C}(T) = C = \N_{C}(T)$ and
$$b''_{\C(G),G} = \sum_{\Gamma \in S_G(\C(G))} \mu([\Gamma:\C(G)]) = 1 - 3 - 4 = -6 \equiv 0 \quad \mathrm{mod}\quad \# \N_{C}(T)\mathbb{Z}$$
since $\# \N_{C}(T) = 3$. Thus condition (d.ii) of Theorem \ref{thm:reducetop}
 holds for $T = \C(G)$.  The other non-trivial cyclic subgroups $T$ of $p$-power order
in $G$ are the $\Gamma(j)$ for $j = 1, 2, 3$.  We have $\N_{C}(\Gamma(j)) = \{e\} = \C_{C}(\Gamma(j))$,
so condition (d.ii) of Theorem \ref{thm:reducetop}
 holds trivially for $T = \Gamma(j)$.  This completes the proof
 that the Bertin obstruction of $\phi_G$ vanishes if and only if that of $\phi_P$ does.
 
It remains to prove the last assertion of the proposition. 
It follows from Proposition \ref{prop:evenqsemi} that the Bertin obstruction of $\phi_P$
vanishes if and only if the KGB obstruction of $\phi_P$ vanishes.   By Theorem \ref{thm:KGBtwo},
if the KGB obstruction of $\phi_G$ vanishes then the Bertin obstruction of $\phi_G$
does also.  So to finish the proof of
of Proposition \ref{prop:qual1}, it will suffice to show that if the Bertin obstructions
of $\phi_G$ and $\phi_P$ vanish then the KGB obstruction of $\phi_G$ vanishes.

We can choose a set $\mathcal{C}_G$ of representatives
for the conjugacy classes of non-trivial cyclic subgroups of $G$ 
in the following
way:
\begin{equation}
\label{eq:CGrep}
\mathcal{C}_G = \{\C(G), \Gamma(1),C(1),J(1)\}.
\end{equation}
We have
\begin{equation}
\label{eq:okeydokey}
b_{\Gamma(1),G} = \frac{1}{[\N_G(\Gamma(1)):\Gamma(1)]} \sum_{\Gamma \in S_G(\Gamma(1))} \mu([\Gamma:\Gamma(1)]) \iota_G(\Gamma) = \frac{\iota(\Gamma(1))}{2} = \frac{i_0+1}{2}.
\end{equation}
Thus $b_{\Gamma(1),G} > 0$, and this is integral  
by Proposition~\ref{prop:special}
because the Bertin obstruction
of $\phi_G$ vanishes.  One has $\N_G(C(1)) = J(1) \ne C(1)$ and $\N_G(J(1)) = J(1)$, so 
Proposition \ref{prop:subtle} gives
\begin{equation}
\label{eq:arrgh!}
b_{C(1),G} = 0 \quad \mathrm{and}\quad b_{J(1),G} = 1
\end{equation}

{From} Theorem \ref{thm:KGBtwo} and Proposition \ref{prop:special}, to show that
the KGB obstruction vanishes, it will be enough to construct for each $T \in \mathcal{C}_G$
an sequence $B_T $ of $ b_{T,B}$ elements of $G$ such that each $b \in B_T$ generates a conjugate of $T$
and $\prod_{T \in \mathcal{C}_G} \prod_{b \in B_T} b$ has order $[G:G_1] = \# C = 3$
after choosing some ordering for $\coprod_{T \in \mathcal{C}_G} B_T$.   Here 
$b_{\Gamma(1),G} > 0$, so since we can choose the elements of $B_{\Gamma(1)}$ to be
generators of any of the three (conjugate) order $4$ subgroups of $P$, we can arrange that
$\prod_{b \in B_{\Gamma(1)}} b $ has order $4$.  We know that $B_{C(1)} = \emptyset$
and that $B_{J(1)}$ has one element by (\ref{eq:arrgh!}).  Now the product of 
an element of order $4$ and and element of order $3$ in $\mathrm{PSL}_2(3) = \mathrm{SL}_2(3)/\C(G)$ is an element of order $3$. Thus for any choice of $B_{\C(G)}$, the resulting product
$$\prod_{T \in \mathcal{C}_G} \prod_{b \in B_T} b$$
has order $3$ image in $\mathrm{PSL}_2(3)$.  We can make this element have order
$3$ in $\mathrm{SL}_2(3) = G$ by adjusting one of the elements of $B_{\Gamma(1)}$
by multiplying it by either the trivial or the non-triivial element of $\C(G)$.  This completes
the proof of Proposition \ref{prop:qual1}.
\end{proof}

\section{Proof of Theorem \ref{thm:nonexs}.}
\label{s:yipes}
\setcounter{equation}{0}

By Corollary \ref{cor:summary}, the proof of Theorem \ref{thm:nonexs}
is reduced to showing the following results:
 \begin{enumerate}
 \item[a.]  The groups listed in items (1) - (4) of Theorem \ref{thm:nonexs} are KGB groups for $k$.
 \item[b.]  When $p = 2$, the quaternion group $Q_8$ and the group $\mathrm{SL}_2(3)$
 are not Bertin groups for $k$.
 \item[c.]  When $p = 2$, no semi-dihedral group of order at least $16$ is a Bertin group for $k$.
 \end{enumerate}
 
\begin{lemma}
\label{lem:cyclic}
For all $k$, every cyclic group $G$ is a KGB group for $k$.
\end{lemma}

\begin{proof}Let  $G$ be a cyclic group with $p$-Sylow
group $H$ of order $p^n$.  The lower numbering of the ramification
groups of $G$ has $G_0 = G$ and $G_i = H_i$ if $i > 0.$  By Lemma \ref{lem:jumps}, there are positive integers $i_0,i_1,\ldots,i_{n-1}$ such 
such that the jumps in the upper numbering of the ramification filtration of $H$ occur
at $i_0, i_0+i_1,\ldots,i_0 + i_1 + \cdots + i_{n-1}$.  Write $\# G = m p^n$ for
some integer $m$ prime to $p$.  By the Hasse-Arf Theorem, the jumps in
the lower ramification filtration are at $q(-1) := 0$ if $m > 1$ together with the integers
$q(\ell) := m\sum_{j = 0}^\ell p^{j} i_j$ for $\ell = 0,\cdots,n-1$.
 We find by Corollary
\ref{cor:computeiota} that 
\begin{equation}
\label{eq:iotacal5}
\iota(p^\ell H) = 1 + q(\ell) \quad \mathrm{for}\quad 0 \le \ell \le n-1\quad;\quad \iota(\Gamma)= 1\quad \mathrm{if }\quad \Gamma \not \subset H.
\end{equation}
 
Suppose that $T$ is a non-trivial cyclic subgroup
of $G$.  By  Proposition \ref{prop:subtle}, 
\begin{equation}
\label{eq:easynon}
b_T = 0 \ \mathrm{if}\ G \ne T \not \subset H\quad \mathrm{and}\quad  b_G = 1 \ \mathrm{if}\ G \ne H.
\end{equation}
Otherwise, $T \subset H$ and we may write  
$T = p^\ell H$ for some $0 \le \ell \le n-1$.  By 
Theorem~\ref{thm:nontrivcase}, 
\begin{eqnarray}
\label{eq:theformagainduh}
b_T &=& \frac{1}{[\N_G(T):T]} \left(  \sum_{\Gamma \in S(T)} \mu([\Gamma:T]) \iota(\Gamma) \right ) \nonumber\\
&=& \frac{1}{mp^{\ell}} \left (\sum_{\Gamma \in S(T)} \mu([\Gamma:T]) + \sum_{\Gamma \in S(T), \Gamma \subset H} \mu([\Gamma:T]) (\iota(\Gamma) - 1)\right ) \label{eq:theformagain}\nonumber \\
&=& \frac{1}{mp^{\ell}} \left (\delta(p^{\ell}H,G) + m p^\ell i_\ell \right )
= \delta(p^\ell   H,G) + i_\ell
\end{eqnarray}
(In the second sum in (\ref{eq:theformagain}), only $\Gamma = H$ contributes if $\ell = 0$, while $\Gamma = p^\ell H$ and $\Gamma = p^{\ell -1}H$ contribute if $\ell > 0$.  The last equality  in (\ref{eq:theformagainduh}) is a consequence of the fact that $p^\ell H = G$ if and only if $m = 1$ and $\ell = 0$.)  This shows that $b_T$ is always a non-negative integer, so the Bertin obstruction of $G$ vanishes by
Proposition \ref{prop:special}.  

To show that
the KGB obstruction vanishes, we start
by picking an ordered set $\{g_t\}_{t \in \Omega}$ of elements of $G$ such
that each $g_t$ is non-trivial, and the number of $g_t$ which generate a  given non-trivial $T \in \mathcal{C}$ is $b_T$.  As in
Theorem \ref{thm:KGBtwo}(b), we have to show that
we can adjust these $g_t$ so that they collectively generate $G$ and so that
$\prod_{t \in \Omega} g_t$ has order $[G:G_1] = m$.

Suppose first that $G = H$, so that $m = 1$.  By (\ref{eq:theformagainduh}), $b_H = 1 + i_0 \ge 2$.  Hence there must at least
two distinct elements $u, v \in \Omega$ such that $g_u$ and $g_v$ generate $G = H$.  
Consider the product $g = \prod_{t \in \Omega - \{u,v\}} g_t$.  It will suffice to show that there
are generators $g'_{u}$ and $g'_{v}$ of $G$ such that $g'_u g'_v = g^{-1}$, since then we can replace
$g_u$ by $g'_u$ and $g_v$ by $g'_v$ to have a set with the required properties.  We claim
that for all primes $p$, the elements of a cyclic group $G = H$ of order $p^n$ which are the product of two 
generators are exactly the set of squares in $G$ (which equals $G$ unless $p = 2$).
This is clear for $n = 1$, and it follows by induction for all $n$.  Thus to construct
the required $g'_u$ and $g'_v$, it will suffice to show that $g$ above is a square if
$p = 2$.  So we now suppose $p = 2$.  Then  $1 + i_0$ is the valuation of the discriminant
of the quadratic extension $k((t))^{2G}$ of $k((t))^G$ inside $k((t))$, and this must be even.  Hence $i_0$ is odd so $b_G = 
 1 + i_0 > 0$ is even.  Every
$T \in \mathcal{C}$ except for $T = G$ is contained in 
$2G = 2H$.  Hence  the product $g =\prod_{t \in \Omega - \{u,v\}} g_t$
 lies in $2G = 2H$, and this completes the analysis of the case $G = H$.

We now suppose that 
$G \ne H$.  By
(\ref{eq:easynon}),
the unique $T \in \mathcal{C}$ which is not a $p$-group and
for which $b_T$ is not $0$ is $T = G$, and $b_G = 1$.  We can
therefore pick the first element $g_u$ of $\{g_t:t \in \Omega\}$ to
be a generator of $G$, and all the other elements will be in $H$.
We are therefore done if $H$ is trivial, so suppose from now
on that $H$ is non-trivial.  
Consider the product $\prod_{u \ne t \in \Omega} g_t \in H$.
If $p = 2$ then all terms of this product are in $2H$ except
for $b_H$ terms in which $g_t$ is a generator of $H$.  Here
$b_H = i_0 > 0$ by (\ref{eq:theformagainduh}), and this is odd if $p = 2$.
Thus we can pick an element $v \in \Omega - \{u\}$ such that $g_v$
is a generator of $H$, and the number of $v' \in \Omega - \{u,v\}$ for which 
$g_{v'}$ generates $H$ is even.  
Taking into account that $g_u$ is a generator of $G$, we find that
$g = \prod_{t \in \Omega} g_t$ lies in $2G$ if $p = 2$; and whether or not $p=2$,
this is the product of a generator $g_u$ of $G$ with an element of the
$p$-Sylow subgroup $H$ of $G$.  This implies $g$ has order divisible by $m$.
It will suffice to show that we can 
pick elements $h_u, h_v \in H$ such that $g'_u = g_u h_u$ is a generator of
$G$, $g'_v = g_v h_v$ is a generator of $H$, and
$$g'_u g'_v \prod_{t \in \Omega - \{u,v\}} g_t =  h_u h_v g$$
has order $m$, since then we can simply replace $g_u$ by $g'_u$
and $g_v$ by $g'_v$.  

We first observe that $h_u h_v g$ always has order divisible by $m$
since $h_u$ and $h_v$ are elements of $p$-power order and $g$
has order divisible by $m$.  Hence $h_u h_v g$ has order  $m$ if and only if $h_u^m h_v^m g^m = e$, where $g^m \in H$
and $g^m \in 2H$ if $p = 2$.  Since $h_u \in H$ and $g_u$ is a generator of $G$,
the element $g'_u = g_u h_u$ will be a generator of $G$ if and only if
$g_u^m h_u^m$ is a generator of the $p$-Sylow subgroup $H$ of $G$.
This will be the case if $h_u^m$ is not congruent mod $pH$ to $g_u^{-m} \in H$.
Similarly, $h_v g_v$ will be a generator of $H$ if and only if $h_v^m g_v^m$
is such a generator, and this will be so if and only if $h_v^m$ is not 
congruent mod $pH$ to $g_v^{-m}$.  We thus see that $h_u^m$ and $h_v^m$
are to be elements of $H$ which each avoid a particular congruence
class mod $pH$ which generates $H$ mod $pH$, and for which $h_u^m h_v^m$ is equal to $g^m$,
where $g^m \in H$ and $g^m \in 2H$ if $p = 2$.  Since $H$ is cyclic
of order $p^n$, one sees by induction on $n$ that such $h_u^m$
and $h_v^m$ always exist.  Since $m$ is prime to $p$, we can then
find $h_u$ and $h_v$ in $H$ with the required properties, which completes
the proof.
\end{proof}

\begin{prop}
\label{prop:trumps}  For all primes $p$ and integers $n \ge 1$,  the dihedral group $D_{2p^n}$ is a KGB group for $k$.  If $p = 2$, then the following hold:  
\begin{enumerate}
\item[i.] If $G$ is a generalized quaternion group of order $2^n \ge 16$, then $G$ is a KGB group for $k$.  
\item[ii.] If $G$ is the quaternion group of order $8$, $G$ is an almost KGB group for $k$.
\item[iii.] If $G$ is a semi-dihedral group then $G$ is not an
almost Bertin group for $k$.
\end{enumerate}

\end{prop}

\begin{proof}By Corollaries \ref{cor:closer}, \ref{cor:closertwo}, \ref{cor:closerthree} and \ref{cor:firstjump},
$D_{2p^n}$ is a KGB group for $k$, for all $p$.  Now assume that $p=2$.  These corollaries
together with Proposition \ref{prop:evenqsemi}(iii) imply
statements (i) and (ii) concerning generalized quaternion groups.
Suppose now that $G$ is a semi-dihedral group.  We can 
construct a Klein four extension $L'/K$ such that the exponents
$d_0$, $d_1$ and $d_2$ of the discriminants of the three quadratic
subfields 
(as in Lemma~\ref{lem:di2}) 
are larger than a specified number and for which $d_0 + d_1 + d_2$ is not divisible
by $4$.  By parts (i) and (ii) of Proposition  
 \ref{prop:subquots} of Appendix 1, we can realize this $L'/K$ as a subfield
of a $G$-extension of $K$ such that the first jump in the
ramification filtration of $G$ occurs above a specified number.   
Proposition \ref{prop:evenqsemi}(iii) now shows $G$ is
not an almost Bertin group for $k$.
\end{proof}

\begin{cor}
\label{cor:enditall}
When $p = 2$, the group $\mathrm{SL}_2(3)$ is an almost KGB group for $k$.
\end{cor}

\begin{proof}
This follows from Propositions \ref{prop:trumps}  and  \ref{prop:qual1}.
\end{proof}

\begin{lemma}
\label{lem:a4}
When $p = 2$, the alternating group $ A_4$ is a KGB group for $k$.
\end{lemma}
\begin{proof}
We can take the set $\mathcal{C}$ to consist of a group $H_2$
of order $2$ and a group $H_3$ of order $3$.  Since
$\N_G(H_3) = H_3$, Proposition \ref{prop:subtle} shows
$b_{H_3} = 1$.  There is no cyclic subgroup of $A_4$ which
properly contains $H_2$, and $\N_{A_4}(H_2)$ has order $4$,
so 
$$b_{H_2} = \frac{1}{[\N_{A_4}(H_2):H_2]} \iota(H_2) = \frac{\iota(H_2)}{2}.$$
Since $H_2$ has order $p = 2$, the first (and only ) jump $i_0$ in the 
upper (and lower)
ramification filtration of $H_2$ occurs at an odd integer,
so $\iota(H_2) = 1 + i_0 \ge 2$ is even.  Thus $b_{H_2} \ge 1$
is an integer.  This shows that Bertin obstruction associated to local
$A_4$ covers in characteristic $2$ is trivial.  Since $b_{H_3} = 1$
and the $2$-Sylow of $A_4$ is normal, if we choose
any set of generators for the stabilizers appearing in the
product described in the KGB condition of Theorem \ref{thm:KGBtwo}(b),
this product will not lie in the $2$-Sylow of $A_4$.  Therefore
this product has to be an element of order $3 = [A_4:(A_4)_1]$,
so the KGB condition holds.
\end{proof}

In view of the remarks at the beginning of this section, the following
result completes the proof of Theorem \ref{thm:nonexs}.

\begin{lemma}
\label{lem:trumpsit}  Suppose $p = 2$. The quaternion group $Q_8$ of order $8$
and the group $\mathrm{SL}_2(3)$ are not Bertin groups for $k$. 
\end{lemma}

\begin{proof}
By Proposition \ref{prop:qual1},  it will be enough
to construct an injection $\phi:G = \mathrm{SL}_2(3) \to \mathrm{Aut}_k(k[[t]])$
such that the restriction of $G$ to the (unique) $2$-Sylow subgroup
$P$ of $G$ has non-trivial Bertin obstruction, where $P$ is isomorphic
to $Q_8$.  By Proposition \ref{prop:evenqsemi}(ii), it will be enough to 
construct an example of this kind in which the lower ramification filtration
of $P$ has the form $P = P_0 = P_1$, $\C(P) = P_2 = P_3$ and $P_4 = \{e\}$.  By \cite[Ex. A.1.b]{Silverman},
there is an elliptic curve $E$ over $k$ whose automorphism
group $G = \mathrm{Aut}(E)$ is isomorphic to $\mathrm{SL}_2(3)$.
Every element of $G$ fixes the origin $\underline{0}$
of $E$, so $\underline{0}$ is totally ramified over its image $c$
in the quotient cover $E \to E/G$.  Let $P_{\underline{0},i}$ be the
$i^{th}$ lower ramification subgroup of $P$ acting on the completion of
the local ring of $\underline{0}$ on $E$.  Then $P_{\underline{0},0} = 
P_{\underline{0},1} = P$.  By applying Lemma \ref{lem:jumps}
to the action of a cyclic subgroup $H$ of order $4$ in $P$, we
see that $P_{\underline{0},2}$ and $P_{\underline{0},3}$ must be non-trivial.
The Hurwitz formula gives
\begin{eqnarray}
\label{eq:hurwitz}
0 &=& 2g(E) - 2 = 8\cdot (2g(E/P) - 2) + \sum_{i = 0}^{\infty} (\#P_{\underline{0},i} - 1) + r_{\ne 0} \nonumber\\
&=& 8 \cdot(2g(E/P) - 2) + 16 + \sum_{i=2}^3 (\#P_{\underline{0},i} - 2) + \sum_{i = 4}^{\infty} (\#P_{\underline{0},i} - 1)  + r_{\ne 0}
\end{eqnarray}
where $r_{\ne 0}$ is the contribution of ramification points of the cover $E \to E/P$
which are not equal to $\underline{0}$.  This implies $g(E/P) = 0$, $r_{\ne 0} = 0$  and that
the ramification filtration of $P = P_{\underline{0}}$ has the required form,
in the sense that $P = P_0 = P_1$, $\C(P) = P_2 = P_3$ and $P_4 = \{e\}$.
Hence the action of $G$ on the completion of the local ring of $E$ at $\underline{0}$ defines an injection $\phi:G = \mathrm{SL}_2(3) \to \mathrm{Aut}_k(k[[t]])$ for which the Bertin obstruction does not vanish.
\end{proof}

\section{Proof of Theorem \ref{thm:nonexsalmost}.}
\label{s:nonexall}
\setcounter{equation}{0}

By Corollary \ref{cor:subquots}, if a quotient of a group $G$ is not an almost Bertin
group then $G$ is not an almost Bertin group for $k$.  The  Bertin groups for $k$ have been determined
in Theorem \ref{thm:nonexs}, and each of these is a KGB group for $k$ and hence an almost KGB group for $k$.
Hence by Theorems~\ref{thm:oddcasered} and~\ref{thm:evenclass}, 
Theorem \ref{thm:nonexsalmost} follows from the following assertions, which have already been shown:
\begin{enumerate}
\item[i.] The groups listed in items (1) - (5) of Theorem \ref{thm:oddcasered} 
are not almost Bertin groups for $k$ if $p = \mathrm{char}\, k \ne 2$.  This follows from Corollary~\ref{cor:theanswerhere}, since each of the groups 1)-(5) in Theorem~\ref{thm:oddcasered} are cyclic-by-$p$.
\item[ii.] The groups listed in items (1) - (7) of Theorem \ref{thm:evenclass} 
are not almost Bertin groups for $k$ if $\mathrm{char}\, k = 2$. This follows from Corollary \ref{cor:twoanswers},
Proposition \ref{prop:no4} and Proposition \ref{prop:no9}.
\item[iii.]  The groups $H_8$ and $\mathrm{SL}_2(3)$ are almost KGB groups $k$ if $\mathrm{char}\, k = 2$.
This was shown in Proposition \ref{prop:trumps}   and Corollary \ref{cor:enditall}.
\item[iv.]  Semi-dihedral groups are not almost Bertin groups in characteristic $2$.  This was shown in Proposition~\ref{prop:trumps}.  
\end{enumerate}

\section{Appendix 1:  Constructing extensions with prescribed ramification.}
\label{s:constructions}
\setcounter{equation}{0}

In this section we suppose $G$ is a finite group which is the semi-direct
product of a normal $p$-group $P$ with a finite cyclic group $C$.

\begin{dfn}
\label{def:uncle}
Suppose that $G$ is a GM group for $k$ with respect to a faithful
character $\Theta:B \to \mathbb{Z}_p^*$ as in Definition \ref{def:groovy}.  Let $\Theta_{C}:C \to W(k)^*$
be an extension of $\Theta$ to a faithful character of $C$. An injection $\phi_G:G \to \mathrm{Aut}_k(k[[z]])$ will be said to be GM for $\Theta_{C}$ if 
\begin{equation}
\label{eq:groovtime}
\phi_G(c)(u) / u \equiv \overline{\Theta}_{C}(c)^{-1} \quad \mathrm{mod}\quad uk[[u]]
\end{equation} 
for some uniformizer $u$ in $k[[z]]^{\phi_G(P)}$, where $\overline{\Theta}_{C}:C \to k^*$
is the reduction of $\Theta_{C}$ mod $pW(k)$.  
\end{dfn}

\begin{lemma}
\label{def:dumdumdum}
Suppose $G = C$.  Then $G$ is GM with respect to 
any given faithful character $\Theta:B \to \mathbb{Z}_p^*$.  Let $\Theta_{C}:C \to W(k)^*$ be a faithful extension of $\Theta$.  There is an injection $\phi_G:G \to \mathrm{Aut}_k(k[[z]])$
which is GM with respect to $\Theta_{C}$.  
\end{lemma}

\begin{proof}
The first statement is clear from Definition \ref{def:groovy}.  For the second statement, 
pick a root of unity $\zeta \in k^*$
of order $\# C$ and a generator $c$ of $C = G$.  Let
$\phi'_G:G \to \mathrm{Aut}_k(k[[z]])$ the injection for which $\phi'_G(c)(z) = \zeta c$.   
Since $\mathrm{Aut}(C)$ acts transitively on the faithful characters of $C$,
there will be a unique $\alpha \in \mathrm{Aut}(C)$ such that $\phi_G = \phi'_G \circ \alpha$
will be GM with respect to $\Theta_{C}$.
\end{proof}

The main result of this section is:

\begin{prop}  
\label{prop:subquots}  Suppose $H$ is a quotient group
of $G$, and let $H(p)$ be the $p$-Sylow subgroup of $H$.  Let $M$ be  a positive integer.  There is an integer $M' \ge 1$, with 
$M' = 1$ if $M = 1$, for which the following is true.  Suppose $\phi_H:H \to \mathrm{Aut}_k(k[[t]])$ is an injection such that that the lower ramification group $H_{M'-1}$ contains $H(p)$.  Let $J$ be the kernel of the surjection $\pi:G \to H$.  Then there is an injection $\phi_G:G \to \mathrm{Aut}_k(k[[z]])$ 
with the following properties:
\begin{enumerate}
\item[i.]  There is a $k$-isomorphism between $k[[z]]^J$ and $k[[t]]$
such that the induced action of $G/J = H$ on $k[[t]]$ is given by $\phi_H$.  
\item[ii.]  The lower ramification group $G_{M-1}$ contains $P$.  

\item[iii.]  Suppose $J \subset P$, $G$ is GM with respect
to $\Theta:B \to \mathbb{Z}_p^*$, and $\Theta_{C}:C \to W(k)^*$ is a faithful 
extension of $\Theta$ to $C$.  Suppose $\phi_H$
is GM with respect to $\Theta_{C}$.  Then $\phi_G$
is GM 
with respect to $\Theta_{C}$.
\item[iv.] Suppose $J \subset P$ and $T$ is a proper non-trivial cyclic subgroup
of $J$.  Then $\iota_G(T) \equiv 0$ mod $p^{M-1}$, where as before $\iota_G(T)$ is $i+1$
if $i$ is the largest integer such that $T$ lies in the ramification group $G_i$.  Let $\Gamma$ be cyclic subgroup of $P$ containing $T$ such that  $J \not \supset \Gamma$. Then $\iota_G(T) > \iota_G(\Gamma) + M$,  
\item[v.]Suppose $M > 1$, $J \subset P$ and $\iota_H(T') \equiv 0$ mod $p^{M'-1}$ for all non-trivial
cyclic $p$-subgroups $T'$ of $H$.  Then $\iota_G(T) \equiv 0$ mod $p^{M-1}$
for all non-trivial cyclic $p$-subgroups $T$ of $G$. 
\end{enumerate}
\end{prop}

The remainder of this section is devoted to proving Proposition \ref{prop:subquots}.
For a related result, see the work of Pries in \cite[Prop.~2.7]{Pries}.
Since $J$ is solvable we have the following result by induction on $\# J$.

\begin{lemma}
\label{lem:reduce}
To prove Proposition \ref{prop:subquots}, it will suffice to consider the case in which $J$ 
is abelian and the conjugation action
of $H$ on $J$ makes $J$ into a simple $\mathbb{Z}[H]$-module.   We will assume
$J$ to be such a module for the rest of this section.
\end{lemma}

\begin{lemma}
\label{lem:step1}
Given
any $\phi_H:H \to \mathrm{Aut}_k(k[[t]])$ as in Proposition \ref{prop:subquots}, there is always an injection $\phi_G:G \to \mathrm{Aut}_k(k[[z]])$
for which condition (i) of the proposition holds.   To complete the 
proof of the proposition, it will suffice to show that there is a $\phi_G$
for which (i) and (iv) hold. 
\end{lemma}
\begin{proof}
It is shown in \cite[Lemma 2.10]{CGH} that there is always a $\phi_G$ as in part (i). We  now show part (iii) of Proposition \ref{prop:subquots}.  
(See Lemma~\ref{def:dumdumdum} for the case $G=C$.)
Let
$J$, $G$, $\Theta_{C}$ and $\phi_H$ be as in part (iii), so that 
 $\phi_H$ is GM for $\Theta_{C}$.  
Then the identification of $k[[t]]$ with $k[[z]]^{\phi_G(J)}$ identifies 
$k[[z]]^{\phi_G(P)}$ with $k[[t]]^{\phi_H(P/J)} = k[[u]]$.  We have identified
$C$ as a subgroup of both $G$ and $H$, and the actions of $\phi_G(c)$
and $\phi_H(c)$ on $k[[u]]$ for $c \in C$ must be the same since
$\phi_G$ induces $\phi_H$.  Since $\phi_H$ is GM with
respect to $\Theta_{C}$,
$$\phi_H(c)(u) / u \equiv \overline{\Theta_{C}}(c)^{-1} \quad \mathrm{mod}\quad uk[[u]].$$
Thus this congruence holds when $\phi_H$ is replaced by $\phi_G$;
so $\phi_G$ is GM with respect to $\Theta_{C}$.

To complete the proof of Lemma \ref{lem:step1} now amounts to showing
that if we can always construct an $M'$ for which parts (i) and (iv) 
of Proposition \ref{prop:subquots} hold, then we can construct an $M'$
for which parts (ii) and (v) also hold.   By increasing $M'$, we can
assume $M' \ge M$.

Let $\sigma$ be a non-trivial element of $G$ of $p$-power order, and define $\sigma' = \pi(\sigma) \in H$. To show $G_{M-1}$ contains $P$ as in part (ii), it will suffice to show 
\begin{equation}
\label{eq:targetineq}
i_G(\sigma) \ge M.
\end{equation}
To show part (v), it will be enough to prove 
\begin{equation}
\label{eq:target2}
\iota_G(\langle \sigma \rangle) \equiv 0\quad \mathrm{ mod}\quad 
p^{M -1}\mathbb{Z}.
\end{equation}

Suppose first $J$ is a $p$-group.  If $\sigma \in J$, then
condition (iv) applied to the subgroup $\langle \sigma \rangle$ generated by $\sigma$ shows
$0 < i_G(\sigma) = \iota_G(\langle \sigma \rangle) \equiv 0$ mod $p^{M-1}$, which shows (\ref{eq:target2}).  We also have $i_G(\sigma) \ge p^{M-1} \ge M$ which proves (\ref{eq:targetineq}).  Suppose now that $\sigma \not \in J$, so that $\sigma' = \pi(\sigma) \in H$
is not trivial.  By \cite[Chap. IV.1, Prop. 3]{corps},
\begin{equation}
\label{eq:relateit}
i_H(\sigma') = \frac{1}{\# J} \sum_{\nu \in G, \pi(\nu) = t'} i_G(\nu) = 
\frac{1}{\# J} \sum_{j \in J} i_G(\sigma j)
\end{equation}
{From} \cite[Chap. IV.1]{corps} we have
$$i_G(\sigma j) \ge \mathrm{Inf}(i_G(\sigma),i_G(j)).$$
Since $\sigma \not \in J$ and we have supposed that (iv) holds, we have $i_G(j) > i_G(\sigma)$ for all
$j \in J$, where $i_G(e) = \infty$ by definition if $e$ is the identity element of $J$.
It follows that $i_G(\sigma j) \ge i_G(\sigma)$ for all $j \in J$, and similarly $i_G(\sigma) = i_G(\sigma jj^{-1}) \ge i_G(\sigma j^{-1})$, so $i_G(\sigma j) = i_G(\sigma)$ for $j \in J$.  Thus (\ref{eq:relateit})
becomes $i_H(\sigma') = i_G(\sigma)$, so $\iota_G(\langle \sigma \rangle ) = i_G(\sigma ) = i_H(\sigma') = \iota_H(\langle \sigma'\rangle )$.
Because we chose $M' \ge M$, part (iv) now gives $0 < i_G(\sigma ) = \iota_H(\langle \sigma'\rangle ) \equiv 0$ mod $p^{M-1}$, so $i_G(\sigma ) \ge p^{M-1} \ge M$ as above, which completes
the proof of (\ref{eq:targetineq}) and (\ref{eq:target2}) when $J$ is a $p$-group. 

  Suppose now that $J$ is not a $p$-group.  We only need to
  show (\ref{eq:targetineq}), since statement (v) of Proposition \ref{prop:subquots}
  holds vacuously.  Since $J$ is a simple $\mathbb{Z}[H]$-module, it has order prime to $p$.  Therefore $\sigma'= \pi(\sigma)$ has the same order as $\sigma$, and in particular is not trivial.   The group
$P.J$ generated by $P$ and $J$ has normal subgroups $P$ and $J$,
and these groups have coprime order and the product of their orders is $\# P.J$.
Hence $P.J$ is isomorphic to $P \times J$ and we conclude that $P$ and $J$
commute.  Thus if $e \ne j \in J$ then $tj$ is not of $p$-power order and
so $i_G(\sigma j) = 1$. In this way, (\ref{eq:relateit}) becomes
$$i_H(\sigma') = \frac{1}{\# J} \sum_{j \in J} i_G(\sigma j) = \frac{i_{G}(\sigma ) + \#J - 1}{\# J}.$$
This shows
$$i_G(\sigma)  - i_H(\sigma') = (\#J -1)(i_H(\sigma')  -1) \ge 0.$$
Thus $M' \ge M$ and the assumption that
$H_{M' - 1}$ contains $H(p)$ in  Proposition \ref{prop:subquots} implies
$$i_G(\sigma) \ge i_H(\sigma') \ge M' \ge M.$$
This establishes (\ref{eq:targetineq}) and completes the proof.
\end{proof}

The following corollary is now clear from Lemma \ref{lem:step1}
because condition (iv) of Proposition \ref{prop:subquots}
holds vacuously if $J$ has order prime to $p$. 

\begin{cor}
\label{cor:easy}
Suppose $J$ is abelian and is a simple $\mathbb{Z}[H]$-module of 
order prime to $p$. Then  Proposition \ref{prop:subquots} holds.
\end{cor}

For the rest of this section we assume the hypotheses of the following
lemma.  

\begin{lemma}
\label{lem:assume} Suppose $J$ is abelian and is a simple $\mathbb{Z}[H]$-module of 
$p$-power order.  Let $c = \# C$ be the order of the prime to $p$-part
of $\# G$ (and of $\# H$). There is a divisor $c'$ of $c$ such that $J$ has the following description.
There is an isomorphism of abelian groups between $J$ and the additive
group $\mathbb{F}_{p^d}^+$ of the finite field $\mathbb{F}_{p^d}$ of order
$p^d$ such that the action of $H$ on $J$ is given by  the inflation to $H$
of a multiplicative character $\chi:C \to \mathbb{F}_{p^d}^*$ of order $c'$.
The field $\mathbb{F}_{p^d}$ is generated over $\mathbb{F}_p$ by a primitive $c'^{th}$
root of unity.  We can choose a uniformizer $w$ in $k[[t]]^{H(p)}$ in such a way that 
there is a faithful character $\chi':C \to k^*$ with the property that $\phi_H(\sigma)$
sends 
$w$ to $\chi'(\sigma) w$ for all $\sigma \in C$ under the natural identification
of $C = H/H(p)$ with $\mathrm{Gal}(k((w))/k((t))^H)$.
\end{lemma}

\begin{proof} Recall that since $J$ is a $p$-group, the surjection $G \to H$
is an isomorphism on $C$; so we can view $C$ as a subgroup of $H$.  Thus
$H$ is the semi-direct product $H(p).C$.  All simple $(\mathbb{Z}/p)[H]$-modules
must be inflated from simple $(\mathbb{Z}/p)[C]$-modules since the 
kernel of the natural surjection $(\mathbb{Z}/p)[H] \mapsto (\mathbb{Z}/p)[C]$
is the radical of $(\mathbb{Z}/p)[H]$. The description of $J$ is now a consequence
of the well-known description of the simple modules in characteristic $p$
for a cyclic group $C$ of order prime to $p$.  The action of $C$ on $k[[t]]^{H(p)}$
via $\phi_H$ makes $k((t))^{H(p)}$ into a tame Kummer extension of $k((t))^H$.
{From} this we get the existence of a uniformizer $w$ in $k[[t]]^{H(p)}$ and of a character $\chi'$ with the properties
stated in the lemma.
\end{proof}

\begin{lemma}
\label{lem:okey}
With the notations and assumptions of Lemma \ref{lem:assume}, let $q = p^d$.  There
are arbitrarily large integers $n > 0$ which are relatively prime to $p$ such that $\chi = \chi'^n$
as characters from $C$ to $\mathbb{F}_q$.  By Lemma \ref{lem:assume},
$J$ as a $(\mathbb{Z}/p)[H]$-module is inflated from a $(\mathbb{Z}/p)[C]$-module
$\tilde J$.   The polynomial  
\begin{equation}
\label{eq:ohyeah}
y^q - y - w^{-n}
\end{equation}
is irreducible in $k((w))[y]$.  Let $L$ be the splitting field of this polynomial over $k((w)) = k((t))^{H(p)}$.
Define $F = k((t))^H$.  When we identify $\tilde J$ with $\mathbb{F}_q^+$,
there is an isomorphism $\mathrm{Gal}(L/k((w))) \to \tilde J$
defined by $\sigma \mapsto \sigma(y) - y$.  The group $C$ embeds into $\mathrm{Aut}_k(L)$
via the map which sends $\tau \in C$ to the automorphism defined by $\tau(w) = \chi'(\tau) w$
and $\tau(y) = \chi'(\tau)^{-n} y$.  This extends to an action of the semi-direct product
$\tilde J.C$ on $L$ which fixes $F$, and in this way $\mathrm{Gal}(L/F) = \tilde J.C$.
The corresponding lower ramification group $\tilde J_{n}$ equals $\tilde J$ while
$\tilde J_{n+1} = \{e\}$.
\end{lemma}

\begin{proof} All of the assertions are clear from Artin-Schreier theory except for the
fact that $\tilde J_n = \tilde J$ and $\tilde J_{n+1} = \{e\}$.  For this observe that
if  $a, b \in \mathbb{Z}$ are such 
that $aq - bn = 1$ then $w^a y^b$ is a uniformizer in $L$. If $e \ne \sigma \in \mathrm{Gal}(L/k((w)))$
then $0 \ne \sigma(y) - y = \zeta \in \mathbb{F}_q$, $b$ is prime to $p$ and 
\begin{equation}
\label{eq:doit}
\mathrm{ord}_L(\frac{\sigma(w^a y^b)}{w^a y^b} - 1) = \mathrm{ord}_L( \frac{(y+\zeta)^b}{y^b} - 1) = \mathrm{ord}_L((1 + \zeta y^{-1})^b - 1) = n.
\end{equation}
Thus $\sigma$ lies in the ramification group $\tilde J_n$ but not in $\tilde J_{n+1}$.
\end{proof}  

In view of Lemma \ref{lem:step1}, part (v) of the following lemma completes
the proof of Proposition \ref{prop:subquots}. 

\begin{lemma}
\label{lem:anticip}
Assume the hypothesis and notations of Lemmas \ref{lem:assume}  and
\ref{lem:okey}.  By Lemma \ref{lem:step1}, there is an injection $\phi_G:G \to \mathrm{Aut}_k(k[[z]])$
inducing $\phi_H$.   Let $F = k((t))^H = L^{\tilde J.C}$ as in Lemma \ref{lem:okey}, so that 
$k((z))/F$ is a Galois $G = P.C$-extension while $L/F$ is a Galois $\tilde J.C$ extension.  If $n$ in Lemma \ref{lem:okey} is
sufficiently large, then the following hold:
\medbreak
\begin{enumerate}
\item[i.] The fields $k((z))$ and $L$ are linearly disjoint over their common
subfield, $k((w)) = k((t))^{H(p)} = k((z))^{P} = L^{\tilde J}$.    
\medbreak
\item[ii.]  Let $N = L\cdot k((z))$
be the compositum of these two extensions of $k((w))$.  
Then
$$\mathrm{Gal}(N/F) = (\tilde J \times P).C$$
where the action of $C$ on the product group $\tilde J \times P$ is via
the conjugation action of $C$ on both factors. 
\medbreak
\item[iii.]  Fix identifications
of $\tilde J$ and $J \subset P$ with $\mathbb{F}_q^+$ as in
Lemmas \ref{lem:assume} and \ref{lem:okey}.  This gives an isomorphism
$\psi:\tilde J \to J$, with the property that 
\begin{equation}
\label{eq:urp22}
\Delta = \{(\tilde t,\psi(\tilde t)): \tilde t \in \tilde J\}
\end{equation}
is a subgroup of $\tilde J \times J \subset \tilde J \times P$ that is normal in 
$\mathrm{Gal}(N/F) = (\tilde J \times P).C$.  
\medbreak
\item[iv.] The fixed field $N^\Delta$ is Galois over $F$, with
$\mathrm{Gal}(N^\Delta/F)$ isomorphic to $G$.  We can choose a uniformizer
$z'$ in $N^\Delta$ and an injection $\phi'_G:G \to \mathrm{Aut}_k(k[[z']])$
having the following properties:  
\begin{enumerate}
\item[a.] There is an isomorphism of $k[[z']]^{\phi'_G(J)}$ with $k[[t]]$ such that $\phi'_G$
induces $\phi_H:H \to \mathrm{Aut}_k(k[[t]])$.
\item[b.] The ramification group $\phi'_G(G)_n$ equals $\phi'_G(J)$, while
$\phi'_G(G)_{n+1} = \{e\}$.   
\end{enumerate}
\medbreak
\item[v.]  Suppose $1 \le M \in \mathbb{Z}$.  We can choose $n$ to be arbitrarily large with $n \equiv -1$ mod $p^{M-1}$.  For such $n$,  the following will be true
for each cyclic subgroup $T$ of $J$:  
\begin{enumerate}
\item[a.] $\iota_G(T) = n+1 \equiv 0$ mod $p^{M-1}\mathbb{Z}$,s where $n$ is the largest
integer such that $\phi'_G(G)_n$ contains $T$ and where we compute $\iota_G$ using
$\phi'_G$. 
\item[b.] $\iota_G(T) > \iota_G(\Gamma) + M$ for all cyclic groups $\Gamma\subset P$ that properly contain $T$. 
\end{enumerate}
For such $n$, $\phi'_G$ will have properties (i) and (iv) in Proposition \ref{prop:subquots}. 
\end{enumerate}
\end{lemma}

\begin{proof} Part (i) follows from the fact that if $n$ in Lemma \ref{lem:okey} is
sufficiently large, the valuation of the relative discriminant of every non-trivial extension 
of $k((w))$ inside $L$ is larger than that of every non-trivial extension
of $k((w))$ inside $k((z))$. 

Parts (ii), (iii) and (iv)(a) are straightforward from Galois theory.  

To prove (iv)(b), consider the $j^{th}$ upper ramification subgroup of the $p$-group 
\begin{equation}
\label{eq:alright}
\mathrm{Gal}(N/k((w))) = \tilde J \times P
\end{equation} This must surject onto $\mathrm{Gal}(L/k((w)))^j$ as well as onto
$\mathrm{Gal}(k((z))/k((w)))^j = P^j$.  By Lemma \ref{lem:okey}, $\mathrm{Gal}(L/k((w)))^n = \mathrm{Gal}(L/k((w)))_n = \mathrm{Gal}(L/k((w)))$, and this group is identified with
$$(\tilde J \times P)/ (1 \times P) \cong \tilde J.$$ 
Furthermore, $\mathrm{Gal}(L/k((w)))^{n+\epsilon} = \{e\}$ if $\epsilon > 0$. If we choose $n$ sufficiently large,
then $\mathrm{Gal}(k((z))/k((w)))^n = P^n$ will be the trivial group, where
$\mathrm{Gal}(k((z))/k((w)))$ is identified with the quotient group
$$(\tilde J \times P)/ (\tilde J \times 1) \cong \tilde P.$$  It follows that if
$n$ is sufficiently large, then $\mathrm{Gal}(N/k((w)))^n$ must lie in $\tilde J \times 1$
and surject onto $\tilde J$, so in fact $\mathrm{Gal}(N/k((w)))^n = \tilde J \times 1$
relative to the description of $\mathrm{Gal}(N/k((w))$ in (\ref{eq:alright}).    Hence
the image of $\mathrm{Gal}(N/k((w)))^n$ in 
$\mathrm{Gal}(N^\Delta/k((w))) = (\tilde J \times P)/\Delta$ is the image of $\tilde J \times 1$,
and this group is identified with $J$ when we identify $G $ with
$((\tilde J\times P).C)/\Delta$ as above.  Thus the action of $G$ on $k[[z']]$
specified by $\phi'_G$ leads to the ramification group $J^n$ being equal to $J$ and $J^{n+\epsilon}$ being $\{e\}$ if $\epsilon > 0$.
Since $J$ is a $p$-group, we conclude that $J_n = J$ and $J_{n+1} = \{e\}$, which completes the proof of part (iv).

For part (v), we observe that the condition on $n$ in Lemma \ref{lem:okey} is
that it be sufficiently large,  relatively prime to $p$, and satisfy $\chi = \chi'^n$
as characters from $C$ to $\mathbb{F}_q$.  The last condition is one
on $n$ mod $\# C$; so since $\# C$ is prime to $p$, we can always find
such $n$ for which $n \equiv -1$ mod $p^{M-1}$. Since $\iota_G(T) = n+1$
for $T$ a non-trivial cyclic subgroup of $J$ by part ($\mathrm{iv.b}$),
we conclude that $\iota_G(T) \equiv 0$ mod $p^{M-1}$ as in ($\mathrm{v.a}$).
Suppose now that $\Gamma$ is a cyclic subgroup of $P$ which contains
$T$ but is not contained in $J$.  Then $\Gamma(H) = \Gamma/(\Gamma \cap J)$
is identified with a non-trivial subgroup of $G/J = H = \mathrm{Gal}(k[[z']]^{\phi'_G(J)}/F)$, where $k[[z']]^{\phi'_G(J)}$ is identified with $k[[t]]$ as in part (iv.a).
This last identification shows that there is a integer $c_0 \ge 0$  independent
of the choice of $n$ such that the upper ramification group $\Gamma(H)^c$
equals $\{e\}$ if $c \ge c_0$.  Then $\Gamma^c \subset J$ for such $c$
since $\Gamma^c$ surjects onto $\Gamma(H)^c$.  We have $\Gamma_{\# G c} \subset \Gamma^c$ by the crudest estimates for the Herbrand function, 
so $\Gamma_{\#G c_0} \subset J$.  Thus $\Gamma_{\# G c_0} \ne \Gamma$, so $\iota_G(\Gamma) \le \#G c_0$.   We now choose $n$ so that $n \ge \# G c_0 + M$,
to have $\iota_G(T) = n+1 > \#G c_0 + M  \ge \iota_G(\Gamma) + M$
as required in part (v).  Finally, we note that having chosen $n$ so that
all of parts (i) - (v) hold, $\phi'_G$ will satisfy conditions (i) and (iv) of
Proposition \ref{prop:subquots}.  
\end{proof}

\section{Appendix 2:  Distinguishing the Bertin and KGB obstructions.}
\label{s:KGBversusBertin}
\setcounter{equation}{0}

In this section we show the KGB obstruction to lifting an injection $\phi:G \to \mathrm{Aut}_k(k[[t]])$
can be non-zero when  the Bertin obstruction vanishes, by proving the following result.

Recall that the first jump in the lower ramification filtration of
$G$ occurs at the largest integer $i_0$ such that $G = G_{i_0}$.  

\begin{prop}
\label{prop:theexam}  Suppose $G$ is isomorphic to $\mathbb{Z}/p \times \mathbb{Z}/p$.  
\begin{enumerate}
\item[a.]{\rm (Bertin)} The Bertin obstruction for lifting $\phi$ vanishes if and only if
$p|(i_0+1)$.  
\item[b.] When $p|(i_0+1)$, 
the KGB obstruction for lifting $\phi$ does not vanish if and only if $p = 3 = i_0+1$
and  $G_{i_0+1} = \{e\}$. 
\end{enumerate} 
\end{prop}

While this shows that the KGB obstruction is in general stronger than the Bertin
obstruction for particular $\phi$, our results in \S \ref{s:intro} show that every 
Bertin group for $k$ is a KGB group for $k$ and vice versa.

\begin{example}
\label{ex:explicit}
When $p = 3$, one obtains from statement (b) an explicit example of a 
$\phi$ with vanishing Bertin obstruction and non-vanishing KGB obstruction
in the following way.  Let $i_0 > 0$ be any integer such that $p|(i_0 + 1)$.  Let $u$ be an indeterminate, and let $N$ be the extension $k((u))[X]/(X^9 - u^{-i_0})$ of $k((u))$.  Then $G = \mathrm{Gal}(N/k((u)))$ is isomorphic to
the finite field $\mathbb{F}_9 \cong  \mathbb{Z}/p \times \mathbb{Z}/p$
via the map sending $\alpha \in \mathbb{F}_9$ to the automorphism $\sigma_\alpha$ for
which $\sigma_\alpha(X) = X + \alpha$.  On has $\mathrm{ord}_N(u) = 9$, $\mathrm{ord}_N(X) = 
-i_0$ and $\mathrm{ord}_N(t) = 1$ when $t = X^a u^b$ and $9b - ai_0 = 1$ for some integers $a$ and $b$.
Thus $t$ is a uniformizer in $N$, and $(\sigma_{\alpha}(t)/t) - 1 = (1 + \alpha X^{-1})^a  -1$
has valuation $\mathrm{ord}_N(a \alpha X^{-1}) = i_0$ for $0 \ne \alpha \in \mathbb{F}_9$.  It follows $G = G_{i_0} \supset G_{i_0+1} = \{e\}$.
\end{example}

\medbreak
\noindent {\bf Proof of Proposition \ref{prop:theexam}}
\medbreak
Statement (a) is a special case of Example 1 of \S 4 of \cite{Bertin}.
As noted there, Green and Matignon proved earlier in \cite{GM} for
$G = \mathbb{Z}/p \times \mathbb{Z}/p$ there is no lift of $\phi$ to characteristic
$0$ unless $p|(i_0+1)$. 

We now focus on statement (b).  The (unique) set $\mathcal{C}$ of representatives for the conjugacy classes of cyclic subgroups $T$ of $G$ consists of the trivial subgroup $\{e\}$ together with the $p+1$ subgroups of $G$ of order $p$.  By Proposition 
\ref{prop:special} and Theorem \ref{thm:nontrivcase}, the set  $S$ appearing in Theorem \ref{thm:KGBtwo} is the disjoint union
over the non-trivial $T \in \mathcal{C}$ of $b_T = \iota(T)/p$ copies of the left $G$-set $G/T$.  
There are always $T \in \mathcal{C}$ not contained in $G_{i_0+1}$, and for these $T$
one has $\iota(T) = i_0+1$.  By the Hasse-Arf Theorem, 
if $T \subset G_{i_0+1}$ then  $T = G_{i_0 + 1}$ and $\iota(T) = i_0+pi_1+1$, where $i_1 \ge 1$ is an integer and the second jump in
the upper numbering of the ramification groups of $G$ is at $i_0+i_1$.   
The KGB obstruction vanishes if and only
if for each non-trivial $T \in \mathcal{C}$ and each integer $j$ such that $1 \le j \le b_T$, we can choose a generator $g_{T,j}$ for $T$
such that
\begin{equation}
\label{eq:KGBcompute}
\prod_{\{e\} \ne T\in \mathcal{C}} \prod_{j = 1}^{b_T} g_{T,j} = e\quad \mathrm{in}\quad G.
\end{equation} 
Note that $b_T = \iota(T)/p > 0$ for all $\{e\} \ne T \in \mathcal{C}$ so $\{g_{T,j}\}_{T,j}$ generates $G$.

Suppose first that $i_0 + 1 > p$. Then $b_T \ge (i_0+1)/p > 1$ for all $\{e\} \ne T \in \mathcal{C}$.  For each such $T$, we can therefore
choose the generators $g_{T,j}$ for $1 \le j \le b_T$ so that $\prod_{j = 1}^{b_T} g_{T,j} = e$.
This makes (\ref{eq:KGBcompute}) hold, so the KGB obstruction vanishes.  

We
now suppose that $i_0 + 1 = p$, but that $G_{i_0+1} \ne \{e\}$.  Then $G_{i_0+1}$
has order $p$, and there are exactly $p$ order $p$ subgroups $T_0,\ldots,T_{p-1}$ of $G$
different from $G_{i_0+1}$.  We can choose the generators for $G = \mathbb{Z}/p \times \mathbb{Z}/p$
so that $G_{i_0+1}$ corresponds to the subgroup $\{0\} \times \mathbb{Z}/p$.  Then
$g_{T_i,1} = (1,i)$ is a generator for $T_i$ for $0 \le i \le p-1$.  We have
$$\prod_{i = 0}^{p-1} g_{T_i,1} = (0, (p-1)p/2)\quad \mathrm{in} \quad G = \mathbb{Z}/p \times \mathbb{Z}/p.$$
Thus this product is in the last cyclic order $p$ subgroup $G_{i_0 + 1}$, and
as noted above, $b_{G_{i_0 + 1}} = \frac{i_0 + p i_1 + 1}{p} > 1$. Hence we can choose the final 
generators $g_{G_{i_0+1},j}$ for $1 \le j \le b_{G_{i_0 + 1}}$ in such a way that
(\ref{eq:KGBcompute}) holds, which shows that the KGB obstruction vanishes.

We are thus reduced to showing
that if $p = i_0 + 1$ and $G_{i_0 + 1} = \{e\}$, then the KGB obstruction vanishes
if and only if $p \ne 3$.  Fix an isomorphism of $G$ with $\mathbb{Z}/p \times \mathbb{Z}/p$,
and define $h_i = (1,i)$ for $0 \le i \le p-1$ and $h_{p} = (0,1)$.  Any generator
for the subgroup $T_i$ generated by $h_i$ has the form $g_{i,1} = c_i \cdot h_i$
for some $c_i \in (\mathbb{Z}/p)^*$.  The question of whether we can choose
generators of these groups for which (\ref{eq:KGBcompute}) holds is that same
as asking where there are $c_i \in (\mathbb{Z}/p)^*$ such that
\begin{equation}
\label{eq:oomph}
\left (\sum_{i = 0}^{p-1} c_i\cdot (1,i)\right ) + c_p \cdot (0,1) = (0,0) \quad \mathrm{in}\quad G = \mathbb{Z}/p \times \mathbb{Z}/p
\end{equation}
Such $c_i$ exist if and only if the KGB obstruction vanishes. 

We will leave it to the reader to check the following facts. If $p = 2$ then $c_0 = c_1 = c_2 = 1$
is a solution of (\ref{eq:oomph}).  If $p = 3$ there is no solution with the $c_i \in (\mathbb{Z}/p)^*$.
Finally, if $p > 2$ then a solution is given by $c_i = 1$ if $0 \le i \le p-3$, $c_{p-2} = -1$,
$c_{p-1} = 3$ and $c_p = -2$.  This completes the proof.  $\square$

\end{document}